\documentclass[reqno]{amsart}
\usepackage{amssymb,latexsym}
\usepackage{amsmath}
\usepackage{amsthm}
\usepackage{graphicx}
\usepackage{hyperref}
\usepackage{titletoc}
\numberwithin{equation}{section}
\newtheorem{theorem}{Theorem}[section]

\newtheorem{lemma}[theorem]{Lemma}
\newtheorem{corollary}[theorem]{Corollary}

\theoremstyle{definition}

\newtheorem{remark}{Remark}

\def\l{\langle}
\def\r{\rangle}

\def\XXint#1#2#3{{\setbox0=\hbox{$#1{#2#3}{\int}$}
     \vcenter{\hbox{$#2#3$}}\kern-.5\wd0}}

\def\l{\langle}
\def\r{\rangle}

\begin{document}
\title[Degree counting and Shadow system]{Degree counting and Shadow system for $SU(3)$ Toda system: one bubbling}

\author{ Chang-shou Lin}
\address{ Chang-shou ~Lin,~Taida Institute for Mathematical Sciences, Center for Advanced Study in
Theoretical Sciences, National Taiwan University, Taipei 106, Taiwan}
\email{cslin@math.ntu.edu.tw}
\author{ Juncheng Wei}
\address{ Juncheng ~Wei,~Department of Mathematics, University of British Columbia,
Vancouver, BC V6T 1Z2, Canada, and Department of Mathematics, Chinese
University of Hong Kong, Shatin, NT, Hong Kong}
\email{jcwei@math.ubc.ca}
\author{ Wen Yang}
\address{ Wen ~Yang,~Department of Mathematics, University of British Columbia,
Vancouver, BC V6T 1Z2, Canada}
\email{wyang@math.ubc.ca}
\date{}\maketitle
\begin{abstract}
Here we initiate the program for computing the Leray-Schauder topological degree for $SU(3)$ Toda system. This program still contains a lot of challenging problems for analysts. The first step of our approach is to answer whether concentration phenomena holds or not. In this paper, we prove the concentration phenomena holds while $\rho_1$ crosses $4\pi,$ and $\rho_2\notin 4\pi\mathbb{N}.$ However, for $\rho_1\geq8\pi,$ the question whether concentration holds or not still remains open up to now.  The second step is to study the corresponding shadow system and its degree counting
formula. The last step is to construct bubbling solution of $SU(3)$ Toda system via a non-degenerate solution of the shadow system. Using this
construction, we succeed to calculate the degree for $\rho_1\in(0,4\pi)\cup(4\pi,8\pi)$ and $\rho_2\notin 4\pi\mathbb{N}$.
\end{abstract}

\section{Introduction}
Let $(M,g)$ be a compact Riemann surface with volume $1,$ $h_1^*$ and $h_2^*$ be a $C^1$ positive function on $M$ and $\rho_1,\rho_2\in\mathbb{R}^+.$ We
consider the following $SU(3)$ Toda system on the compact surface $M.$
\begin{equation}
\label{1.1}
\left\{\begin{array}{ll}
\Delta u_1^*+2\rho_1(\frac{h_1^*e^{u_1^*}}{\int_Mh_1^*e^{u_1^*}}-1)-\rho_2(\frac{h_2^*e^{u_2^*}}{\int_{M}h_2^*e^{u_2^*}}-1)=
4\pi\sum_{q\in S_1}\alpha_{q}(\delta_{q}-1),\\
\Delta u_2^*-\rho_1(\frac{h_1^*e^{u_1^*}}{\int_Mh_1^*e^{u_1^*}}-1)+2\rho_2(\frac{h_2^*e^{u_2^*}}{\int_{M}h_2^*e^{u_2^*}}-1)=
4\pi\sum_{q\in S_2}\beta_{q}(\delta_{q}-1),
\end{array}
\right.
\end{equation}
where $\Delta$ is the Beltrami-Laplace operator, $\alpha_{q}\geq0$ for every $q\in S_1,$ $\beta_q\geq0$ for every $q\in S_2$ and $\delta_q$ is the Dirac
measure at $q\in M$.

When the two equations in (\ref{1.1}) are identical, i.e., $S_1=S_2,~\alpha_q=\beta_q$, $u_1^*=u_2^*=u^*,$ $h_1^*=h_2^*=h^*$ and
$\rho_1=\rho_2=\rho$, system (\ref{1.1}) is reduced to the following mean field equation
\begin{equation}
\label{1.2}
\Delta u^*+\rho(\frac{h^*e^{u^*}}{\int_{M}h^*e^{u^*}}-1)=4\pi\sum_{q\in S_1}\alpha_{q}(\delta_{q}-1).
\end{equation}

Equation (\ref{1.1}) and equation (\ref{1.2}) arise in many physical and geometric problems. In physics, (\ref{1.2}) or (\ref{1.1})
are one of the limiting equations of the abelian gauge field theory or non-abelian Chern-Simons gauge field theory,
one can see \cite{d, dj, ly1, nt2, nt3, y} and references therein. In conformal geometry,
a solution $u^*$ of
\begin{equation}
\label{1.3}
\Delta_gu^*+e^{u^*}-2K=4\pi\sum_{q\in S_1}\alpha_{q}\delta_{q}~\mathrm{in}~M,
\end{equation}
where $K(x)$ is the Gaussian curvature of the given metric $g$ at $x\in M,$ is equivalent to saying that the new metric
$e^{2v}g$ (where $2v=u^*-\log2$) has constant Gaussian curvature $\tilde{K}=1$. By integrating (\ref{1.3}), it is easy to see that
(\ref{1.3}) is a special case of (\ref{1.2}) with $\rho=4\pi\sum_{q\in S_1}\alpha_q$. Since $u^*$ has a logarithmic singularity at each $q\in S_1$,
the new metric $e^{2v}g$ has a conic singularity at each $q$. Equation (\ref{1.3}) has been extensively studied in the last three decades,
see \cite{cgy, cl4, lw0, lwy, lz2, t}
and references therein. However, when the number of the singularities is greater than three, there are very few existence results for
equation (\ref{1.3}). Studies on the case of four singularities are referred to \cite{cl4} and \cite{lw0}. For the recent development of the mean field equation (\ref{1.2}), we refer the readers
to \cite{bt1, bclt, bm, cl1, cl2, cl3, cl4, djlw1, l1, m1, nt1, nt2, y}.

For equation (\ref{1.2}), we let the set $\Sigma$ of the critical parameters be defined by
\begin{align*}
\Sigma:&=\big\{8N\pi+\Sigma_{q\in A}8\pi(1+\alpha_{q})\mid A\subseteq S_1,~N\in\mathbb{N}\cup\{0\}\big\}\setminus\{0\}\\
&=\{8\pi \mathfrak{a}_{k}\mid k=1,2,3,\cdots.\},
\end{align*}
where $\mathfrak{a}_{k}$ will be defined in (\ref{1.4}). It was proved that if $\rho\notin\Sigma$, then the a-priori estimate for any solution of (\ref{1.2}) holds
in $C^2_{\mathrm{loc}}(M\setminus S_1).$ This a-priori bound was obtained by Li and Shafrir \cite{ls} for the case without singular sources, and by
Bartolucci and Tarantello \cite{bt1} for the general case with singular sources. After establishing the a-priori bound for a non-critical parameter $\rho$,
it is natural to count the Leray-Schauder topological degree for the equation (\ref{1.2}). It was proved by Li \cite{l} that this degree counting
should depend only on the topology of $M$ for the case without singularities. In a series of papers \cite{cl1}-\cite{cl4}, Chen and Lin has derived the topological degree counting formulas as described below.

We denote the topological degree of (\ref{1.2}) for $\rho\notin\Sigma$ by $d_{\rho}$. By the homotopic invariant of the topological degree, $d_{\rho}$ is
a constant for $8\pi \mathfrak{a}_k<\rho<8\pi \mathfrak{a}_{k+1},~k=0,1,2,\cdots,$ where $\mathfrak{a}_0=0$. Set $d_m=d_{\rho}$ for $8\pi\mathfrak{a}_m<\rho<8\pi\mathfrak{a}_{m+1}$. To
state the result, we introduce the following generating function $\Xi_0:$
\begin{align*}
\Xi_0(x)=&(1+x+x^2+x^3+\cdots)^{-\chi(M)+|S_1|}\Pi_{q\in S_1}(1-x^{1+\alpha_{q}})\nonumber\\
=&1+\mathfrak{c}_1x^{\mathfrak{a}_1}+\mathfrak{c}_2x^{\mathfrak{a}_2}+\cdots+\mathfrak{c}_kx^{\mathfrak{a}_k}+\cdots.
\end{align*}
The degree $d_m$ can be written in terms of $\mathfrak{c}_j$, as shown in the following theorem.

\vspace{0.5cm}
\noindent {\bf{Theorem A}}. (\cite{cl4}) {\em Let $d_{\rho}$ be the Leray-Schauder degree for (\ref{1.2}). Suppose $8\mathfrak{a}_k\pi<\rho<8\mathfrak{a}_{k+1}\pi$. Then
$$d_{\rho}=\sum_{j=0}^{k}\mathfrak{c}_j,$$
where $d_0=1.$}
\vspace{0.5cm}

For the application, it often requires that $\alpha_{q}\in\mathbb{N}$ for all $q\in S_1.$ In this case,
$\Sigma=\{8m\pi\mid m\in\mathbb{N}\}$ and let $d_m=d_{\rho}$ for $\rho\in(8m\pi,8(m+1)\pi).$ Then the generating function
\begin{align}
\label{1.4}
\Xi_1(x)=&\sum_{k=0}^{\infty}d_kx^k=(1+x+x^2+\cdots)^{-\chi(M)+1+|S_1|}\Pi_{q\in S_1}(1-x^{\alpha_q+1})\nonumber\\
=&(1+x+x^2+\cdots)^{-\chi(M)+1}\Pi_{q\in S_1}(1+x+x^2+\cdots+x^{\alpha_q})
\end{align}
Clearly, we have $d_m\geq1,~\forall m$ provided $\chi(M)\leq0$. Hence we can obtain the existence of the solution to (\ref{1.2}) when the genus
of $M$ is nonzero. When $M$ is a torus and $\sum_{q\in S_1}\alpha_q$ is an odd integer, by applying the Theorem A, we can get the
degree formula for (\ref{1.3})
$$d=\frac{\Pi_{q\in S_1}(1+\alpha_q)}{2}.$$
\vspace{0.25cm}

Similarly, we could consider the following Toda system
\begin{align}
\label{1.5}
\left\{\begin{array}{l}
\Delta u_1^*+2e^{u_1^*}-e^{u_2^*}-2K=4\pi\sum_{q\in S_1}\alpha_q\delta_q,\\
\Delta u_2^*+2e^{u_2^*}-e^{u_1^*}-2K=4\pi\sum_{q\in S_2}\beta_q\delta_q,
\end{array}\right.
\end{align}
on $M$, which is a natural generalization of (\ref{1.3}), but is a special case of (\ref{1.1}). In geometry, it is closely related to the classical
Pl$\ddot{u}$cker formula for a holomorphic curve from $M$ to $\mathbb{CP}^2$, the vortex points and $\alpha_q$ are exactly the branch points and
its ramification index of this
holomorphic curve. See \cite{lwy} for more precise formulation and also \cite{bw, bjrw, c, cw, g, les}
for connection with different aspects of geometry. For the past decades, there are many studies for the $SU(3)$ Toda system, or more generally,
system of equations with exponential nonlinearity. We refer the readers to \cite{m5, jw, jlw, ll, l2, lwz0, lwz1, lwz2, ly2, lz1, lz2, lz3, m3, m4, m6, nt3, wzz, y1} and references therein.

In this paper, we want to initiate the program for computing the Leray-Schauder degree formula for the system (\ref{1.1}). However, it seems still a very
challenging problem even now. Hence in this article we shall consider the simplest (but nontrivial) case, described below. We assume
\begin{itemize}
  \item [(i)] $S_1,S_2=\emptyset$,
  \item [(ii)]$\rho_1\in(0,4\pi)\cup(4\pi,8\pi)$ and $\rho_2\notin\Sigma_1=\{4N\pi\mid N\in\mathbb{N}\}.$
\end{itemize}

To eliminate the singularities on the right hand side of (\ref{1.1}), we introduce the Green function $G(x,p)$:
\begin{equation*}
-\Delta G(x,p)=\delta_p-1~\mathrm{in}~M,~\quad\mathrm{with}~\int_MG(x,p)=0.
\end{equation*}
and let
\begin{align*}
u_1(x)=u_1^*(x)-4\pi\sum_{q\in S_1}\alpha_{q}G(x,q),~u_2(x)=u_2^*(x)-4\pi\sum_{q\in S_2}\beta_{q}G(x,q).
\end{align*}
Then (\ref{1.1}) is equivalent to the following system
\begin{equation}
\label{1.6}
\left\{\begin{array}{l}
\Delta u_1+2\rho_1(\frac{h_1e^{u_1}}{\int_Mh_1e^{u_1}}-1)-\rho_2(\frac{h_2e^{u_2}}{\int_{M}h_2e^{u_2}}-1)=0,\\
\Delta u_2-\rho_1(\frac{h_1e^{u_1}}{\int_Mh_1e^{u_1}}-1)+2\rho_2(\frac{h_2e^{u_2}}{\int_{M}h_2e^{u_2}}-1)=0,
\end{array}
\right.
\end{equation}
where $h_1,h_2\geq0$ in $M$ and $h_1(x)=0$ iff $x\in S_1$, $h_2=0$ iff $x\in S_2$. Near each $q\in S_1$, $h_1$ has the form in local coordinate:
\begin{align*}
h_1(x)=h_{1,q}(x)|x-q|^{2\alpha_q},~\mathrm{for}~|x-q|\ll1,~\forall q\in S_1,
\end{align*}
where $h_{1,q}(x)>0$ for any $q\in S_1$. Near each $q\in S_2$, $h_2$ has the form in local coordinate:
\begin{align*}
h_2(x)=h_{2,q}(x)|x-q|^{2\beta_q},~\mathrm{for}~|x-q|\ll1,~\forall q\in S_2,
\end{align*}
where $h_{2,q}(x)>0$ for any $q\in S_2$.\\

We notice that equation (\ref{1.6}) is invariant by adding constant to the solutions. Hence we can always normalize $u_1,u_2$ to satisfy
$\int_Mu_1=\int_Mu_2=0.$ Let $\mathring{H}^1$ be the space:
\begin{align*}
\mathring{H}^1=\{u\in H^1(M):\int_Mu=0\}.
\end{align*}
From now on, we will restrict our discussion in $\mathring{H}^1\times\mathring{H}^1$. In order to compute the Leray-Schauder degree of the system,
we need to get well understand of the blow-up phenomena for (\ref{1.6}). The first main issue for system is to determine the set of critical parameters,
i.e., those $\rho=(\rho_1,\rho_2)$ such that the a-priori bounds for solutions of (\ref{1.1}) fail. In \cite{jlw}, the authors claimed that if
$\rho_i\notin 4\pi\mathbb{N},~i=1,2,$ then the a-priori bound for all solution exists. See Theorem 1.2 in \cite{jlw}. However they did not give a proof
of this fact, but just said that it follows immediately from Proposition 2.4 in \cite{jlw}, where the local masses of solutions $(u_1,u_2)$ at a blow up
point is calculated. In addition to Proposition 2.4 in \cite{jlw}, for their claim of Theorem 1.2, it requires that the concentration phenomena holds, i.e., $\frac{h_ie^{u_{ik}}}{\int_Mh_ie^{u_{ik}}}$
tends to a sum of Dirac measures. However, as far as the authors know, the concentration has not been proved yet. Under the assumption $(i)$ and $(ii)$, we can show that concentration holds, i.e., if $u_{1k}$ blows up at some points, then $\frac{h_1e^{u_{1k}}}{\int_Mh_1e^{u_{1k}}}$ tends to a sum of Dirac measures.

Our first main theorem is the following a-priori estimate.
\begin{theorem}
\label{th1.1}
Suppose $h_i$ are positive smooth functions and the assumption $(i)-(ii)$. Then there exists a positive constant $c$ such that for any solution of
equation
(\ref{1.6}), there holds:
$$|u_1(x)|,|u_2(x)|\leq c,~\forall x\in M,~i=1,2.$$
\end{theorem}

For the general case, we shall study the concentration
phenomena for $\rho_1\geq8\pi$ in a future work. By Theorem \ref{th1.1}, the Leray-Schauder degree $d_{\rho_1,\rho_2}^{(2)}$ for (\ref{1.1}), or equivalently (\ref{1.6}), is
well-defined for $\rho_1\in(0,4\pi)\cup(4\pi,8\pi)$ and
$\rho_2\notin4\pi\mathbb{N}$. Clearly, $d_{\rho_1,\rho_2}^{(2)}=d_{\rho_2}^{(1)}$ if $0<\rho_1<4\pi$, and $\rho_2\notin4\pi\mathbb{N}.$ Hence, the main
contribution of our paper is to compute the degree $d_{\rho_1,\rho_2}^{(2)}$ for $4\pi<\rho_1<8\pi$. By the homotopic invariant, for any
fixed $\rho_2\notin4\pi\mathbb{N}$, $d_{\rho_1,\rho_2}^{(2)}$ is a constant for $\rho_1\in(0,4\pi)$, and the same holds true for
$\rho_1\in(4\pi,8\pi)$. For the simplicity, we might let $d_-^{(2)}$ and $d_+^{(2)}$ denotes $d_{\rho_1,\rho_2}^{(2)}$
for $\rho_1\in(0,4\pi)$ and $\rho_1\in(4\pi,8\pi)$. Since $d_-^{(2)}$ is known by Theorem A, computing $d_+^{(2)}$ is equivalent to
computing the difference of $d_+^{(2)}-d_-^{(2)}$, which might be not zero due to the bubbling phenomena of (\ref{1.1}) at $(4\pi,\rho_2)$.

To calculate $d_+^{(2)}-d_-^{(2)}$, we need to compute the topological degree of the bubbling solution of (\ref{1.6}) when
$\rho_1$ crosses $4\pi$, $\rho_2\notin 4\pi\mathbb{N}.$ For convenience, we rewrite (\ref{1.6}) as
\begin{equation}
\label{1.7}
\left\{\begin{array}{l}
\Delta v_{1}+\rho_{1}(\frac{h_1e^{2v_{1}-v_{2}}}{\int_Mh_1e^{2v_{1}-v_{2}}}-1)=0,\\
\Delta v_{2}+\rho_{2}(\frac{h_2e^{2v_{2}-v_{1}}}{\int_Mh_2e^{2v_{2}-v_{1}}}-1)=0,
\end{array}
\right.
\end{equation}
where $v_{1}=\frac{1}{3}(2u_{1}+u_{2}),~v_{2}=\frac{1}{3}(u_{1}+2u_{2}).$ It is known that the Leray-Schauder degree for (\ref{1.6}) and (\ref{1.7}) are
the same. So, our aim is to compute the degree contribution of the bubbling solution of (\ref{1.7}) when $\rho_{1}$ crosses $4\pi$,
$\rho_2\notin 4\pi\mathbb{N}.$ We consider $(v_{1k},v_{2k})$ to be a sequence of solutions of (\ref{1.7}) with $(\rho_{1k},\rho_{2k})\rightarrow(4\pi,\rho_2)$,
and assume $\max_M(v_{1k},v_{2k})\rightarrow\infty$. Then we have the following theorem

\begin{theorem}
\label{th1.2}
Let $(v_{1k},v_{2k})$ be described as above. Then, the followings hold:
\begin{itemize}
  \item [(i)]
  \begin{align}
  \label{1.8}
  \rho_{1k}\frac{h_1e^{2v_{1k}-v_{2k}}}{\int_Mh_1e^{2v_{1k}-v_{2k}}}\rightarrow4\pi\delta_p~\mathrm{for~some}~p\in M,
  \end{align}
  \item [(ii)]
  $v_{2k}\rightarrow\frac12w~\mathrm{in}~C^{2,\alpha}(M)$ where $(p,w)$ satisfies
  \begin{align}
  \label{1.9}
  \nabla\big(\log (h_1e^{-\frac12w})(x)+4\pi R(x,x)\big)\mid_{x=p}=0,
  \end{align}
  and
  \begin{align}
  \label{1.10}
  \Delta w+2\rho_2(\frac{h_2e^{w-4\pi G(x,p)}}{\int_{M}h_2e^{w-4\pi G(x,p)}}-1)=0.
  \end{align}
  Here $R(x,p)$ refers to the regular part of the Green function $G(x,p)$.
\end{itemize}
\end{theorem}

We write (\ref{1.9}) and (\ref{1.10}) as
\begin{align}
\label{1.11}
\left\{\begin{array}{l}
\Delta w+2\rho_2(\frac{h_2e^{w-4\pi G(x,p)}}{\int_{M}h_2e^{w-4\pi G(x,p)}}-1)=0,\\
\nabla\big(\log (h_1e^{-\frac12w})(x)+4\pi R(x,x)\big)\mid_{x=p}=0.
\end{array}\right.
\end{align}

The system of equation (\ref{1.11}) is called the {\em shadow system} of (\ref{1.7}). This kind of systems also appear while studying the self-dual system \cite{kl} for the Jackiw-Weinberg electroweak theory. After Theorem \ref{th1.2}, it is natural to ask the following question: Given any pair of solution $(p,w)$, can we find a bubbling solution $(v_{1k},v_{2k})$ of (\ref{1.7}) with $(\rho_{1k},\rho_{2k})\rightarrow(4\pi,\rho_2)$ such that (\ref{1.11}) holds and $v_{2k}$ converges to $\frac12w$ in $C^{2,\alpha}(M)$. One of main results in this article is to give an answer of this question.\\

For the application in other problems, we want to consider a more general class of equation than (\ref{1.11}). Let
\begin{align*}
Q=P_w\cup S=\{p_1^0,p_2^0,\cdots,p_m^0\}\cup S,~\mathrm{where}~P_w\cap S_1=\emptyset,~S\subseteq S_1,
\end{align*}
and $(P_w,w)$ be a solution of
\begin{equation}
\label{1.12}
\left\{\begin{array}{ll}
\Delta w+2\rho_2\big(\frac{h_2e^{w-4\pi\sum_{j=1}^mG(x,p_j^0)-4\pi\sum_{q\in S}(1+\alpha_q)G(x,q)}}{\int_Mh_2e^{w-4\pi\sum_{j=1}^mG(x,p_j^0)-
4\pi\sum_{q\in S}(1+\alpha_q)G(x,q)}}-1\big)=0,\\
\nabla_{x_i}f_{Q}(p_1^0,p_2^0,\cdots,p_m^0)=0,
\end{array}\right.
\end{equation}
where
\begin{align*}
f_{Q}(x_1,x_2,\cdots,x_m)=&\sum_{j=1}^m\big[\log(h_1e^{-\frac12w})(x_j)+4\pi R(x_j,x_j)\big]+4\pi\sum_{i,j=1,i\neq j}^mG(x_i,x_j)\nonumber\\
&+8\pi\sum_{q\in S}\sum_{j=1}^m(1+\alpha_{q})G(x_j,q).
\end{align*}

It is clear to see (\ref{1.12}) is a shadow system of (\ref{1.11}) corresponding some more complicate bubbling phenomena. Note that in (\ref{1.12}),
there might allow $(P_w\cup S)\cap S_2\neq\emptyset.$ We say $(P_w,w)$ is called a
{\em non-degenerate solution} of (\ref{1.12}) if the linearized equation. i.e., for $(\phi,\overrightarrow{\nu}),$
$\overrightarrow{\nu}=(\nu_1,\nu_2,\cdots,\nu_m),$ where $\nu_i\in\mathbb{R}^2$
\begin{align}
\label{1.13}
\left\{\begin{array}{l}
\Delta\phi+2\rho_2\frac{\overline{h}_2e^{w-4\pi\sum_{j=1}^mG(x,p_j^0)}}{\int_{M}\overline{h}_2e^{w-4\pi\sum_{j=1}^mG(x,p_j^0)}}\phi\\
\quad\quad-2\rho_2\frac{\overline{h}_2e^{w-4\pi\sum_{j=1}^mG(x,p_j^0)}}{\big(\int_{M}\overline{h}_2e^{w-4\pi\sum_{j=1}^mG(x,p_j^0)}\big)^2}
\int_{M}\big(\overline{h}_2e^{w-4\pi\sum_{j=1}^mG(x,p_j^0)}\phi\big)\\
\quad\quad-8\pi\rho_2\frac{\overline{h}_2e^{w-4\pi\sum_{j=1}^mG(x,p_j^0)}\sum_{j=1}^m(\nabla G(x,p_j^0)\nu_j)}{\int_{M}\overline{h}_2
e^{w-4\pi\sum_{j=1}^mG(x,p_j^0)}}\\
\quad\quad+8\pi\rho_2\frac{\overline{h}_2e^{w-4\pi\sum_{j=1}^mG(x,p_j^0)}\int_{M}\big(\overline{h}_2e^{w-4\pi\sum_{j=1}^mG(x,p_j^0)}
\sum_{j=1}^m(\nabla G(x,p_j^0)\nu_j)\big)}{\big(\int_{M}\overline{h}_2e^{w-4\pi\sum_{j=1}^mG(x,p_j^0)}\big)^2}=0,\\
\nabla^2_{x_i}f_{Q}(p_1^0,p_2^0,\cdots,p_m^0)\nu_i+\mathcal{F}_i-\frac12\nabla\phi(p_i^0)=0,~i=1,2,\cdots,m,~\int_M\phi=0,
\end{array}\right.
\end{align}
admits only trivial solution, i.e., $(\phi,\overrightarrow{\nu})=(0,0)$. Here
\begin{equation}
\label{1.14}
\overline{h}_2=h_2e^{-4\pi\sum_{q\in S}(1+\alpha_{q})G(x,q)}
\end{equation}
and
\begin{equation}
\label{1.15}
\mathcal{F}_i=8\pi\sum_{j=1,j\neq i}^m\nabla_x^2G(p_i^0,x)\mid_{x=p_j^0}\cdot\nu_j.
\end{equation}
\\

For the shadow system  (\ref{1.12}), the set of non-critical parameters $\Sigma_2$ is
defined as
\begin{align}
\label{1.16}
\Sigma_2=\{4N\pi+4\pi\sum_{q\in S_2\cup S}(1+\alpha_q),~N\in\mathbb{N}\}.
\end{align}

Our third main result is the following.
\begin{theorem}
\label{th1.3}
Suppose $P_w=(p_1^0,p_2^0,\cdots,p_m^0)$ and $(P_w,w)$ is a non-degenerate solution of (\ref{1.12}) and the quantity $l(Q)\neq0$ \footnote{The definition of $l(Q)$ is given in section 4.}. Suppose $\alpha_q\notin\mathbb{N}$ for $q\in S$ and
$\rho_2\notin \Sigma_2.$ Then there exists a sequence of solutions $(v_{1k},v_{2k})$ of (\ref{1.7}) with $(\rho_{1k},\rho_{2})$ such that
$\lim_{k\rightarrow+\infty}\rho_{1k}=4m\pi+\sum_{q\in S}4\pi(1+\alpha_{q})$.
Furthermore we have:
\begin{itemize}
  \item [(i)] $\rho_{1k}\frac{h_1e^{2v_{1k}-v_{2k}}}{\int_Mh_1e^{2v_{1k}-v_{2k}}}\rightarrow
  4\pi\sum_{j=1}^m\delta_{p_j^0}+4\pi\sum_{q\in S}(1+\alpha_{q})\delta_{q},$
  \item [(ii)] $v_{2k}\rightarrow \frac12w$ in $C^{1,\alpha}(M).$
\end{itemize}
\end{theorem}

The proof of Theorem \ref{th1.3} will be given in section 4-5. The main difficulty for constructing such solutions would be the one for $v_{1k}$ component.
Here we follow the arguments in \cite{cl2}, \cite{cl4} which procedures simultaneously have the advantage in computing the Morse index contributed by the
bubbling solutions $(v_{1k},v_{2k})$. The calculation of the Morse index is the key step towards computing the Leray-Schauder degree for system (\ref{1.7}),
once the degree for the shadow system is known. For a given solution $(P_w,w)$ of (\ref{1.12}), we say $\mu$ is a eigenvalue of the linearized system of
equation (\ref{1.12}) if there exists a nontrivial pair $(\phi,\overrightarrow{\nu})=(\phi;\nu_1,\nu_2,\cdots,\nu_m)$ such that
\begin{align*}
\left\{\begin{array}{l}
\Delta\phi+2\rho_2\frac{\overline{h}_2e^{w-4\pi\sum_{j=1}^mG(x,p_j^0)}}{\int_{M}\overline{h}_2e^{w-4\pi\sum_{j=1}^mG(x,p_j^0)}}\phi\\
\quad\quad-2\rho_2\frac{\overline{h}_2e^{w-4\pi\sum_{j=1}^mG(x,p_j^0)}}{\big(\int_{M}\overline{h}_2e^{w-4\pi\sum_{j=1}^mG(x,p_j^0)}\big)^2}
\int_{M}\big(\overline{h}_2e^{w-4\pi\sum_{j=1}^mG(x,p_j^0)}\phi\big)\\
\quad\quad-8\pi\rho_2\frac{\overline{h}_2e^{w-4\pi\sum_{j=1}^mG(x,p_j^0)}\sum_{j=1}^m(\nabla G(x,p_j^0)\nu_j)}
{\int_{M}\overline{h}_2e^{w-4\pi\sum_{j=1}^mG(x,p_j^0)}}\\
\quad\quad+8\pi\rho_2\frac{\overline{h}_2e^{w-4\pi\sum_{j=1}^mG(x,p_j^0)}\int_{M}\big(\overline{h}_2e^{w-4\pi\sum_{j=1}^mG(x,p_j^0)}
\sum_{j=1}^m(\nabla G(x,p_j^0)\nu_j)\big)}{\big(\int_{M}\overline{h}_2e^{w-4\pi\sum_{j=1}^mG(x,p_j^0)}\big)^2}+\mu\phi=0,\\
\nabla^2_{x_i}f_{Q}(p_1^0,\cdots,p_m^0)\nu_i+\mathcal{F}_i-\frac12\nabla\phi(p_i^0)+\mu\nu_i=0,i=1,2,\cdots,m,\quad\int_M\phi=0.
\end{array}\right.
\end{align*}
The Morse index of the solution $(P_w,w)$ to (\ref{1.12}) is the total number (counting multiplicity) of the negative eigenvalue of the linearized system,
and the Leray-Schauder topological degree contributed by $(P_w,w)$ is given by $(-1)^{N}$, where $N$ is the Morse index.\\

From Theorem \ref{th1.3}, it is known when $\rho_{1k}\rightarrow 4m\pi+4\pi\sum_{q\in S}(1+\alpha_{q}),$ there exists a sequence of solutions
$(v_{1k},v_{2k})$ to (\ref{1.7}) such that $v_{1k}$ blow up at $Q$ and $v_{2k}\rightarrow\frac12w$. Our aim is to compute the topological degree of
 (\ref{1.7}) contributed by those bubbling solutions satisfying the conclusion of Theorem \ref{th1.3}, that is, all the bubbling solutions contained
 in $S_{\rho_1}(Q,w)\times S_{\rho_2}(Q,w)$. For the precise definition of $S_{\rho_i}(Q,w),~i=1,2,$ see section 4 in this paper. Let $d_T(Q,w)$ denotes
 the degree contributed by the solutions $(v_{1k},v_{2k})\in S_{\rho_1}(Q,w)\times S_{\rho_2}(Q,w)$ and $d_S(Q,w)$ denotes the degree of the shadow
 system (\ref{1.12}) contributed by the Morse index of $(P_w,w)$. We have the following theorem
\begin{theorem}
\label{th1.4}
Suppose $\alpha_q\notin \mathbb{N}$ for $q\in S$, $(p_w,w)$ is a non-degenerate solution of (\ref{1.12}) and $l(Q)\neq0$. Let $d_T(Q,w)$ and $d_S(Q,w)$ are defined
above. Then
\begin{equation*}
d_T(Q,w)=(-1)^{n}d_S(Q,w),
\end{equation*}
where $n=|Q|.$
\end{theorem}

We shall apply Theorem \ref{th3.1}, a stronger version of Theorem \ref{th1.2}, and Theorem \ref{th1.4} to calculate the degree of (\ref{1.6}), or equivalently (\ref{1.7}).
Here we assume $S_1=S_2=\emptyset,$ i.e., $h_1,h_2$
are $C^{2,\alpha}$ positive functions on $M$. It is still very difficult for us to compute the topological degree for $SU(3)$ Toda system while
$S_1,S_2\neq\emptyset$. In general, the main difficulties are to prove the concentration and to get the topological degree for system (\ref{1.12}).
Until now, we are only able to over those difficulties under the assumption $(i)$ and $(ii)$. Our approach to obtain the degree of (\ref{1.11}) is to use a homotopic deformation to decouple the system. However, this method can not work for (\ref{1.12}) in general. The main difficulty is due to the collapse of the vortices.\\

In order to state our degree formulas for $SU(3)$ Toda system and the corresponding shadow system (\ref{1.11}), we first introduce the following
generating
function
\begin{align*}
\Xi_1(x)=(1+x+x^2+x^3\cdots)^{-\chi(M)+1}=b_0+b_1x^1+b_2x^2+\cdots+b_kx^k+\cdots,
\end{align*}
which is (\ref{1.4}) provided $\alpha_q=0,~\forall q\in S_1.$ It is easy to see that
\begin{align}
\label{1.17}
b_k=\left(\begin{array}{l}k-\chi(M)\\ ~\quad k\end{array}\right),
\end{align}
where
\begin{align*}
\left(\begin{array}{l}k-\chi(M)\\ ~\quad k\end{array}\right)=
\left\{\begin{array}{ll}
\frac{(k-\chi(M))\cdots(1-\chi(M))}{k!},&\mathrm{if}~k\geq1\\
1,&\mathrm{if}~k=0.
\end{array}
\right.
\end{align*}

\begin{theorem}
\label{th1.5}
Assume $S_1=S_2=\emptyset$ and $\rho_2\notin 4N\pi,N\in\mathbb{N}$. The set of solutions $(p,w)$ for (\ref{1.11}) is pre-compact in the space
$M\times\mathring{H}_1(M)$. Let $d_S$ denotes the topological degree for $(\ref{1.11})$ when $\rho_2\in(4k\pi,4(k+1)\pi).$ Then
\begin{equation}
\label{1.18}
d_S=\chi(M)\cdot (b_k+b_{k-1}),
\end{equation}
where $b_{-1}=0.$
\end{theorem}

Under the assumption $(i)-(ii)$ and the previous discussion, we can obtain the partial results on computing the Leray-Schauder degree for system
(\ref{1.6}), or equivalently (\ref{1.7}).
\begin{theorem}
\label{th1.6}
Suppose $S_1=S_2=\emptyset$ and $d_{\rho_1,\rho_2}^{(2)}$ denotes the topological degree for (\ref{1.7}) when $\rho_2\in(4k\pi,4(k+1)\pi)$, then
\begin{align*}
d_{\rho_1,\rho_2}^{(2)}=
\left\{\begin{array}{ll}
b_k,&\rho_1\in(0,4\pi),\\
b_k-\chi(M)(b_k+b_{k-1}),&\rho_1\in(4\pi,8\pi).
\end{array}\right.
\end{align*}
\end{theorem}
Set $d_k^{(2)}=d_{\rho_1,\rho_2}^{(2)}$ for $\rho_1\in(4\pi,8\pi)$ and $\rho_2\in(4k\pi,4(k+1)\pi)$. Then the generating function for
$d_{\rho_1,k}^{(2)},\rho_1\in (4\pi,8\pi)$ is
$$\Xi_2(x)=\sum_{k=0}^{\infty}d_{k}^{(2)}x^k=\big[1-\chi(M)(1+x)\big]\Xi_1(x).$$

As a consequence of Theorem \ref{th1.6}, we have the following corollaries.
\begin{corollary}
\label{co1.1}
Suppose $S_1=S_2=\emptyset$, $M$ is the sphere $S^2,$ $\rho_1\in(4\pi,8\pi)$ and $\rho_2\in(4k\pi,4(k+1)\pi)$. Then
\begin{align*}
d_{\rho_1,\rho_2}^{(2)}=
\left\{\begin{array}{ll}
-1,~&\mathrm{if}~k=0,\\
-1,~&\mathrm{if}~k=1,\\
2,~&\mathrm{if}~k=2,\\
0,~&\mathrm{if}~k\geq3.
\end{array}\right.
\end{align*}
\end{corollary}

\begin{corollary}
\label{co1.2}
Suppose $S_1=S_2=\emptyset,$ $\rho_1,\rho_2\in(4\pi,8\pi),$ and $d_{2,2}^{(2)}$ denotes the topological degree for (\ref{1.6}). We have
\begin{equation}
\label{1.19}
d_{2,2}^{(2)}=(\chi(M))^2-3\chi(M)+1.
\end{equation}
\end{corollary}
A consequence of the degree counting formula is the existence of (\ref{1.6}). Suppose $S_2=\emptyset$ and $\chi(M)\leq0,$ then for any
$(\rho_1,\rho_2)\in(4\pi,8\pi)\times(4k\pi,4(k+1)\pi),$ the system (\ref{1.6}) has a solution.\\

This paper is organized as follows. In section 2, we prove Theorem \ref{th1.1}, Theorem \ref{th1.2} and use the transversality theorem to show
that there exists smooth function $h_1^*$ and $h_2^*$ such that any solution of shadow system (\ref{1.12}) is non-degenerate. In section 3,
we get the a-priori estimate for solutions of (\ref{1.7}) when $\rho_1\rightarrow 4\pi$ and $\rho_2\notin 4\pi\mathbb{N}$. In section 4 and
section 5, we use the solution of shadow system (\ref{1.12}) to get a good approximation of some bubbling solutions of (\ref{1.7}) and thereby
prove Theorem \ref{th1.3} and
Theorem \ref{th1.4} except some important estimates which are shown in section 8. In section 6, we prove Theorem \ref{th1.5} and derive the degree
counting formula for the Shadow system (\ref{1.11}). In section 7, we give a brief account for the
Dirichlet problem on a bounded smooth domain of $\mathbb{R}^2.$

\vspace{1cm}
\section{Shadow system}
We shall prove Theorem \ref{th1.1} and Theorem \ref{th1.2}. As mentioned in the Introduction, this result is not an immediate consequence of Proposition
2.4 in Jost-Lin-Wang \cite{jlw}, due to the fact of concentration has not yet been proved. Therefore, we  want to provide a correct proof of this a-priori estimate. For the concentration phenomena in the general case, we shall discuss it in another paper.

For a sequence of bubbling solution $(u_{1k},u_{2k})$ of (\ref{1.6}). We set $$\tilde{u}_{ik}=u_{ik}-\int_M h_ie^{u_{ik}},~i=1,2.$$ Then $\tilde{u}_{ik}$
satisfy
\begin{equation}
\label{2.1}
\left\{\begin{array}{ll}
\Delta \tilde{u}_{1k}+2\rho_1(h_1e^{\tilde{u}_{1k}}-1)-\rho_2(h_2e^{\tilde{u}_{2k}}-1)=0,\\
\Delta \tilde{u}_{2k}-\rho_1(h_1e^{\tilde{u}_{1k}}-1)+2\rho_2(h_2e^{\tilde{u}_{2k}}-1)=0.
\end{array}\right.
\end{equation}
We define the blow up set for $\tilde{u}_{ik}$
\begin{align}
\label{2.2}
\mathfrak{S}_{i}=\{p\in M\mid \exists \{x_k\},~x_k\rightarrow p,~\lim \tilde{u}_{ik}(x_k)\rightarrow+\infty\}
\end{align}
and define $\mathfrak{S}=\mathfrak{S}_1\cup\mathfrak{S}_2.$  We note that
$$u_{ik}=\tilde{u}_{ik}+\int_Mh_ie^{u_{ik}}\geq\tilde{u}_{ik}+Ce^{\int_Mu_{ik}}\geq\tilde{u}_{ik}+C,$$
where we used the Jensen's inequality and $h_i$ (here $h_i=h_i^*$) is a positive function in $M.$ So, if $p$ is a blow up point of $\tilde{u}_{ik}$, then
$p$ is also a blow up point of $u_{ik}$. For any $p\in \mathfrak{S},$ we define the local mass by
\begin{equation}
\label{2.3}
\sigma_{ip}=\lim_{\delta\rightarrow0}\lim_{k\rightarrow+\infty}\frac{1}{2\pi}\int_{B_{\delta}(p)}\rho_ih_ie^{\tilde{u}_{ik}}.
\end{equation}

\begin{lemma}
\label{le2.1}
If $\sigma_{1p},\sigma_{2p}<\frac{2\pi}{3},$ we have $p\notin\mathfrak{S}.$
\end{lemma}

\begin{proof}
The proof is a standard by using the argument in \cite{bm}. We provide a detail proof for the sake of completeness. Since $\sigma_{ip}\leq\frac{2\pi}{3},$
then we can choose small $r_0$, such that in $B_{r_0}(p),$ the following holds
\begin{equation}
\label{2.4}
\int_{B_{r_0}(p)}\rho_ih_ie^{\tilde{u}_{ik}}<\pi,
\end{equation}
which implies $\int_{B_{r_0}(p)}\tilde{u}_{ik}^+\leq C$, where $C$ is some constant independent of $k$. In the following, $C$ always denotes some generic
constant independent of $k$, and may depend on the domain $B_{r_0}(p).$ For the first equation in (\ref{2.1}), we decompose
$\tilde{u}_{1k}=\sum_{j=1}^3\tilde{u}_{1k,j},$ where $\tilde{u}_{1k,j}$ satisfy the following equation
\begin{equation}
\label{2.5}
\left\{\begin{array}{llll}
-\Delta \tilde{u}_{1k,1}=2\rho_1h_1e^{\tilde{u}_{1k}}-\rho_2h_2e^{\tilde{u}_{2k}}~&\mathrm{in}~B_{r_0}(p),
\quad&\tilde{u}_{1k,1}=0~&\mathrm{on}~\partial B_{r_0}(p),\\
-\Delta \tilde{u}_{1k,2}=-2\rho_1+\rho_2~&\mathrm{in}~B_{r_0}(p),\quad&\tilde{u}_{1k,2}=0~&\mathrm{on}~\partial B_{r_0}(p),\\
-\Delta \tilde{u}_{1k,3}=0~&\mathrm{in}~B_{r_0}(p),\quad&\tilde{u}_{1k,3}=\tilde{u}_{1k}~&\mathrm{on}~\partial B_{r_0}(p).
\end{array}\right.
\end{equation}
For the first equation in (\ref{2.5}), since
$$\int_{B_{r_0}(p)}\Big|2\rho_1h_1e^{\tilde{u}_{1k}}-\rho_2h_2e^{\tilde{u}_{2k}}\Big|<3\pi,$$
By \cite[Theorem 1]{bm}, we have
\begin{equation}
\label{2.6}
\int_{B_{r_0}(p)}\exp((1+\delta)|\tilde{u}_{1k,1}|)dx\leq C,
\end{equation}
where $\delta\in(0,\frac{1}{3})$. Therefore, we have
\begin{equation}
\label{2.7}
\int_{B_{r_0}(p)}|\tilde{u}_{1k,1}|\leq C.
\end{equation}
For the second equation in (\ref{2.5}), we can easily get
\begin{equation}
\label{2.8}
\int_{B_{r_0}(p)}|\tilde{u}_{1k,2}|\leq C,~\mathrm{and}~|\tilde{u}_{1k,2}|\leq C.
\end{equation}
For the third equation in (\ref{2.5}). By the mean value theorem for harmonic function we have
\begin{align}
\label{2.9}
\|\tilde{u}_{1k,3}^+\|_{L^{\infty}(B_{r_0/2}(p))}&\leq C\|\tilde{u}_{1k,3}^+\|_{L^{1}(B_{r_0}(p))}\nonumber\\
&\leq C\Big[\|\tilde{u}^+_{1k}\|_{L^1(B_{r_0}(p))}+\|\tilde{u}_{1k,1}\|_{L^{1}(B_{r_0}(p))}+\|\tilde{u}_{1k,2}\|_{L^{1}(B_{r_0}(p))}\Big]\nonumber\\
&\leq C.
\end{align}
From (\ref{2.8})-(\ref{2.9}), we have
\begin{equation}
\label{2.10}
2\rho_1h_1e^{\tilde{u}_{1k,2}+\tilde{u}_{1k,3}}\leq C~\mathrm{in}~B_{r_0/2}(p).
\end{equation}
By (\ref{2.6}), (\ref{2.10}) and H$\ddot{o}$lder inequality, we obtain
$$e^{\tilde{u}_{1k}}\in L^{1+\delta_1}(B_{r_0}(p))$$
with $\delta_1>0$ independent of $k$. Similarly, we have
$$e^{\tilde{u}_{2k}}\in L^{1+\delta_2}(B_{r_0}(p))$$
with $\delta_2>0$ independent of $k$. By using the standard elliptic estimate for the first equation in (\ref{2.5}), we get
$\|\tilde{u}_{1k,1}\|_{L^{\infty}}(B_{r_0/2}(p))$ is uniformly bounded. Combined with (\ref{2.8}) and (\ref{2.9}), we have $\tilde{u}_{1k}$ is
uniformly bounded above in $B_{\frac{r_0}{2}}(p)$. Following a same process, we can also obtain $\tilde{u}_{2k}$ is uniformly bounded above in
$B_{\frac{r_0}{2}}(p)$. Hence, we finish the proof of the lemma.
\end{proof}

From Lemma \ref{le2.1}, we get if $p\in\mathfrak{S},$ either $\sigma_{1p}\geq\frac{2\pi}{3}$ or $\sigma_{2p}\geq\frac{2\pi}{3}$. Thus
$|\mathfrak{S}|<\infty.$ Therefore $\mathfrak{S}$ is discrete in $M.$ In fact, in next lemma, we shall prove that if $p\in\mathfrak{S}_i,$
$\sigma_{ip}$ must be positive.

\begin{lemma}
\label{le2.2}
If $p\in\mathfrak{S}_i$, $\sigma_{ip}>0.$
\end{lemma}

\begin{proof}
We prove it by contradiction. Without loss of generality, we assume $\sigma_{2p}=0$. First, we claim that there is a constant $C_K>0$ that depends on the
compact set $K$ such that
\begin{equation}
\label{2.11}
|u_{ik}(x)|\leq C_K,~\forall x\in K\subset\subset M\setminus\mathfrak{S},~i=1,2.
\end{equation}
We only prove for $i=1,$ the other one can be obtained similarly
\begin{align*}
u_{1k}(x)=&\int_MG(x,z)\Big(2\rho_1(h_1e^{\tilde{u}_{1k}}-1)-\rho_2(h_2e^{\tilde{u}_{2k}}-1)\Big)\nonumber\\
=&\int_{M_1}G(x,z)\Big(2\rho_1(h_1e^{\tilde{u}_{1k}}-1)-\rho_2(h_2e^{\tilde{u}_{2k}}-1)\Big)\nonumber\\
&+\int_{M\setminus M_1}G(x,z)\Big(2\rho_1(h_1e^{\tilde{u}_{1k}}-1)-\rho_2(h_2e^{\tilde{u}_{2k}}-1)\Big),
\end{align*}
where $M_1=\cup_{p\in\mathfrak{S}}B_{r_0}(p)$ and $r_0$ is small enough to make $K\subset\subset M\setminus M_1.$ It is easy to see that
\begin{align*}
\int_{M_1}G(x,z)\Big(2\rho_1(h_1e^{\tilde{u}_{1k}}-1)-\rho_2(h_2e^{\tilde{u}_{2k}}-1)\Big)=O(1),
\end{align*}
because $G(x,z)$ is bounded due to the distance $d(x,z)\geq\delta_0>0$ for $z\in M_1$, and $x\in K.$ In $M\setminus M_1,$ we can see that $\tilde{u}_{ik}$
 are bounded above by some constant depends on $r_0$, then it is not difficult to obtain that
\begin{align*}
\int_{M\setminus {M_1}}G(x,z)\Big(2\rho_1(h_1e^{\tilde{u}_{1k}}-1)-\rho_2(h_2e^{\tilde{u}_{2k}}-1)\Big)=O(1).
\end{align*}
Therefore, we prove the claim. Since $\sigma_{2p}=0$, we can find some $r_0$, such that
\begin{align}
\label{2.12}
\int_{B_{r_0}(p)}\rho_2h_2e^{\tilde{u}_{2k}}\leq \pi
\end{align}
for all $k$ (passing to a subsequence if necessary) and $r_0\leq\frac12d(p,\mathfrak{S}\setminus\{p\})$. On $\partial B_{r_0}(p)$, by (\ref{2.11})
\begin{equation}
\label{2.13}
|u_{1k}|,~|u_{2k}|\leq C~\mathrm{on}~\partial B_{r_0}(p).
\end{equation}
Let $w_k$ satisfy the following equation
\begin{align}
\label{2.14}
\left\{\begin{array}{ll}
\Delta w_{k}=\rho_1(\frac{h_1e^{u_{1k}}}{\int_Mh_1e^{u_{1k}}}-1)~&\mathrm{in}~B_{r_0}(p),\\
w_k=u_{1k}~&\mathrm{on}~\partial B_{r_0}(p).
\end{array}\right.
\end{align}
We set $w_k=w_{k1}+w_{k2}$ where $w_{k1},w_{k2}$ satisfy
\begin{equation}
\label{2.15}
\left\{\begin{array}{llll}
\Delta w_{k1}=\rho_1\frac{h_1e^{u_{1k}}}{\int_Mh_1e^{u_{1k}}}~&\mathrm{in}~B_{r_0}(p),\quad &w_{k1}=u_{1k}~&\mathrm{on}~\partial B_{r_0}(p),\\
\Delta w_{k2}=-\rho_1~&\mathrm{in}~B_{r_0}(p),\quad &w_{k1}=0~&\mathrm{on}~\partial B_{r_0}(p).
\end{array}\right.
\end{equation}
By maximum principle, we have $w_{k1}\leq\max_{\partial B_{r_0}(p)}u_{1k}\leq C$ by (\ref{2.13}) for $x\in B_{r_0}(p).$ By elliptic estimate, we can
easily get $|w_{k2}|\leq C.$ Therefore,
\begin{equation}
\label{2.16}
w_k\leq C,~\forall x\in B_{r_0}(p).
\end{equation}
We set $u_{2k}=f_{k1}+f_{k2}+w_k,$ where $f_{k1}$ and $f_{k2}$ satisfy
\begin{equation}
\label{2.17}
\left\{\begin{array}{llll}
\Delta f_{k1}=-2\rho_2\frac{h_2e^{u_{2k}}}{\int_Mh_2e^{u_{2k}}}~&\mathrm{in}~B_{r_0}(p),\quad &f_{k1}=0~&\mathrm{on}~\partial B_{r_0}(p),\\
\Delta f_{k2}=2\rho_2~&\mathrm{in}~B_{r_0}(p),\quad &f_{k2}=u_{2k}-w_k~&\mathrm{on}~\partial B_{r_0}(p).
\end{array}\right.
\end{equation}
For the second equation in (\ref{2.17}), we have
\begin{equation*}
|f_{k2}|\leq|u_{2k}|+|w_k|=|u_{2k}|+|u_{1k}|\leq C~\mathrm{on}~\partial B_{r_0}(p),
\end{equation*}
Thus $|f_{k2}|\leq C$ in $B_{r_0}(p).$ We denote $g_k=e^{f_{k2}+w_k},$ then the first equation in (\ref{2.16}) can be written as
\begin{equation}
\label{2.18}
\Delta f_{k1}+2\rho_2\frac{h_2e^{g_k}}{\int_Mh_2e^{u_{2k}}}e^{f_{k1}}=0~\mathrm{in}~B_{r_0}(p),\quad f_{k1}=0~\mathrm{on}~\partial B_{r_0}(p).
\end{equation}
By using the Jensen's inequality, we have $\int_Mh_2e^{u_{2k}}\geq Ce^{\int_Mu_{2k}}\geq C>0.$ We set $V_k=2\rho_2\frac{h_2e^{g_k}}{\int_Mh_2e^{u_{2k}}},$
and have $V_k\leq C,$ this $C$ depends on $r_0$. Using (\ref{2.12}), we get $\int_{B_{r_0}(p)}V_ke^{f_k}\leq2\pi$. By \cite[Corollary 3]{bm}, we
have $|f_{k1}|\leq C$ and
$$u_{2k}\leq f_{k1}+f_{k2}+w_k\leq C.$$
This leads to $\tilde{u}_{2k}=u_{2k}-\int_Mh_2e^{u_{2k}}\leq C,$ which contradicts to the assumption $\tilde{u}_{2k}$ blows up at $p.$ Thus we
finish the proof of this lemma.
\end{proof}

By these two lemmas, we now begin to prove Theorem \ref{th1.1}.\\

\noindent {\em Proof of Theorem \ref{1.1}.} We note that it is enough for us to prove $\tilde{u}_{ik}$ is uniformly bounded above. We shall prove it by
contradiction.

First, we claim $\mathfrak{S}_1\neq\emptyset.$ If not, $\tilde{u}_{1k}$ is uniformly bounded above
and $\tilde{u}_{2k}$ blows up. We decompose $u_{2k}=u_{2k,1}+u_{2k,2},$ where
$u_{2k,1}$ and $u_{2k,2}$
satisfies the following
\begin{align*}
\left\{\begin{array}{ll}
\Delta u_{2k,1}-\rho_1(h_1\tilde{u}_{1k}-1)=0,~&\int_Mu_{2k,1}=0,\\
\Delta u_{2k,2}+2\rho_2(\frac{\tilde{h}_{2k}e^{u_{2k,2}}}{\int_M\tilde{h}_{2k}e^{u_{2k,2}}}-1)=0,~&\int_Mu_{2k,2}=0,\\
\end{array}\right.
\end{align*}
where $\tilde{h}_{2k}=h_2e^{u_{2k,1}}$. By the $L^p$ estimate, $u_{2k,1}$ is bounded in $W^{2,p}$ for any $p>1.$ Thus $u_{1k}$ is bounded in
$C^{1,\alpha}$ for any $\alpha\in(0,1),$ after passing to a subsequence if necessary, we gain $u_{2k,1}$ converges to $u_0$ in $C^{1,\alpha}.$
As a consequence, $\tilde{h}_{2k}\rightarrow h_2e^{u_0}$ in $C^{1,\alpha}$. Since $\tilde{u}_{2k}$ blows up, $u_{2k}$ and $u_{2k,2}$ both blow up. Then
applying the result of Li and Shafrir in \cite{ls}, we have
$\rho_2\in 4\pi\mathbb{N}$, which contradicts to our assumption. Thus $\mathfrak{S}_1\neq\emptyset$. Similarly, we can prove that
$\mathfrak{S}_2\neq\emptyset$.

We note that our argument above can be applied to the local case, which yields $\mathfrak{S}_1\cap\mathfrak{S}_2\neq\emptyset.$ Suppose
$\mathfrak{S}_1\cap\mathfrak{S}_2=\emptyset.$ For any point $p\in \mathfrak{S}_2,$ we consider
the behavior of $u_{1k}$ and $u_{2k}$ in $B_{r_0}(p)$, where $r_0$ is small enough such that $B_{r_0}(p)\cap(\mathfrak{S}\setminus\{p\})=\emptyset.$
We decompose
$\tilde{u}_{2k}=u_{2k,3}+u_{2k,4},$ where $u_{2k,3}$ and $u_{2k,4}$ satisfy
\begin{align}
\label{2.19}
\left\{\begin{array}{llll}
\Delta u_{2k,3}-\rho_1(h_1e^{\tilde{u}_{1k}}-1)=0~&\mathrm{in}~B_{r_0}(p),\quad
&u_{2k,3}=0~&\mathrm{on}~\partial B_{r_0}(p),\\
\Delta u_{2k,4}+2\rho_2(\tilde{h}_{2,k}e^{u_{2k,4}}-1)=0~&\mathrm{in}~B_{r_0}(p),\quad
&u_{2k,4}=\tilde{u}_{2k}~&\mathrm{on}~\partial B_{r_0}(p),
\end{array}\right.
\end{align}
where $\tilde{h}_{2,k}=h_2e^{u_{2k,3}}.$ By using $\tilde{u}_{1k}$ uniformly bounded from above in $B_{r_0}(p)$, we have $\tilde{h}_{2,k}$ converges
in $C^{1,\alpha}(B_{r_0}(p)).$ Since $u_{2k,4}$ blows up simply at $p,$ we have
\begin{align}
\label{2.20}
\Big|u_{2k,4}-\log\Big(\frac{e^{u_{2k,4}(p^{(k)})}}
{(1+\frac{\rho_2\tilde{h}_{2,k}(p^{(k)})e^{u_{2k,4}(p^{(k)})}}{4}|x-p^{(k)}|^{2})^2}\Big)\Big|\leq C,
\end{align}
where $u_{2k,4}(p^{(k)})=\max_{B_{r_0}(p)}u_{2k,4}$. (\ref{2.20}) is proved in \cite{bclt} and \cite{l}. From (\ref{2.20}), we have
\begin{equation}
\label{2.21}
u_{2k,4}\rightarrow-\infty~\mathrm{in}~ B_{r_0(p)}\setminus\{p\}\quad\mathrm{and}\quad \rho_2h_2e^{\tilde{u}_{2k}}\rightarrow 4\pi\delta_p~
\mathrm{in}~B_{r_0}(p),
\end{equation}
which implies
\begin{equation}
\label{2.22}
\rho_2=\lim_{k\rightarrow\infty}\int_M\rho_2h_2e^{\tilde{u}_{2k}}=4\pi|\mathfrak{S}_2|,
\end{equation}
a contradiction to our assumption $\rho_2\notin 4\pi\mathbb{N}$, so $\mathfrak{S}_1\cap\mathfrak{S}_2\neq\emptyset.$

Let $p\in \mathfrak{S}_1\cap\mathfrak{S}_2,$ and $\sigma_{ip},i=1,2$ be the local masses of them at $p.$ Applying the result of Jost-Lin-Wang
(Proposition 2.4 in \cite{jlw}), we have $(\sigma_{1p},\sigma_{2p})$ is one of $(2,4),~(4,2)$ and $(4,4).$ By the assumption $(ii),$
$\sigma_{1p}=2.$ Thus $\sigma_{2p}=4.$

In the following, we claim $\tilde{u}_{2k}$ concentrate, i.e., $\tilde{u}_{2k}\rightarrow-\infty$ uniformly in any compact set of
$M\setminus\mathfrak{S}_2.$ Then,
\begin{equation}
\label{2.23}
\rho_2h_2e^{\tilde{u}_{2k}}\rightarrow4\pi\sum_{q\in\mathfrak{S}_2\setminus\{p\}}\delta_q+8\pi\delta_p~\mathrm{and}~\rho_2\in4\pi\mathbb{N},
\end{equation}
which again yields a contradiction. This completes the proof of Theorem \ref{th1.1}. The proof of this claim is given in Lemma \ref{le2.3} below.
$\square$

\begin{lemma}
\label{le2.3}
Suppose $\tilde{u}_{ik},i=1,2$ both blow up at $p$, and let $2$ and $4$ be the local masses of $\tilde{u}_{1k}$ and $\tilde{u}_{2k}$ respectively.
Then $\tilde{u}_{2k}\rightarrow-\infty$ in $B_{r_0}(p)\setminus\{p\}.$
\end{lemma}

\begin{proof}
If the claim is not true, we have $\tilde{u}_{2k}$ is bounded by some constant $C$ in $L^{\infty}(\partial B_{r_0}(p))$. Let
$f_{1k}=-\rho_1(h_1e^{\tilde{u}_{1k}}-1)+2\rho_2(h_2e^{\tilde{u}_{2k}}-1)$ and $z_k$ be the solution of
\begin{align}
\label{2.24}
\left\{\begin{array}{ll}
-\Delta z_k=f_{1k}~&\mathrm{in}~B_{r_0}(p),\\
z_k=-C~&\mathrm{on}~\partial B_{r_0}(P).
\end{array}\right.
\end{align}
Note that $f_{1k}\rightarrow f_1$ uniformly in any compact set of $B_{r_0}(p)\setminus\{p\}$ and the integration of the RHS over $B_{r_0}(p)$ is
$12\pi+o(1)$ as $r_0\rightarrow0$. By maximum principle, $\tilde{u}_{2k}\geq z_k$ in $B_{r_0}(p).$ In particular
$$\int_{B_{r_0}(p)}e^{z_k}\leq \int_{B_{r_0}(p)}e^{\tilde{u}_{2k}}<\infty.$$
On the other hand, using Green representation formula for $z_k,$ we have
\begin{align}
\label{2.25}
z_k(x)=-\int_{B_{r_0}(p)}\frac{1}{2\pi}\ln|x-y|\big(-\rho_1(h_1e^{\tilde{u}_{1k}}-1)+2\rho_2(h_2e^{\tilde{u}_{2k}}-1)\big)+O(1),
\end{align}
where we used the regular part of the Green function is bounded. For any $x\in B_{r_0}(p)\setminus\{p\},$ we denote the distance between $x$ and $p$ by
$2r$. From (\ref{2.25}), we have
\begin{align*}
z_k(x)=&-\int_{B_{r_0}(p)}\frac{1}{2\pi}\ln|x-y|\big(-\rho_1(h_1e^{\tilde{u}_{1k}}-1)+2\rho_2(h_2e^{\tilde{u}_{2k}}-1)\big)+O(1)\nonumber\\
=&-\int_{B_{r_0}(p)\cap B_{r}(x)}\frac{1}{2\pi}\ln|x-y|\big(-\rho_1(h_1e^{\tilde{u}_{1k}}-1)+2\rho_2(h_2e^{\tilde{u}_{2k}}-1)\big)\nonumber\\
&-\int_{B_{r_0}(p)\setminus B_r(x)}\frac{1}{2\pi}\ln|x-y|\big(-\rho_1(h_1e^{\tilde{u}_{1k}}-1)+2\rho_2(h_2e^{\tilde{u}_{2k}}-1)\big)+O(1).
\end{align*}
It is easy to see
$$\Big|\int_{B_{r_0}(p)\cap B_{r}(x)}\ln|x-y|\big(-\rho_1(h_1e^{\tilde{u}_{1k}}-1)+2\rho_2(h_2e^{\tilde{u}_{2k}}-1)\big)\Big|\leq C,$$
due to $\tilde{u}_{ik}$ are uniformly bounded above in $B_{r}(x),~i=1,2$. Here $C$ depends only on $x$.
For $y\in B_{r_0}(p)\setminus B_r(x),$ we have $|x-y|\geq r$ and
$$\int_{B_{r_0}(p)\setminus B_r(x)}\ln|x-y|\big(-\rho_1(h_1e^{\tilde{u}_{1k}}-1)+2\rho_2(h_2e^{\tilde{u}_{2k}}-1)\big)=(12\pi+o(1))\ln|x-p|+O(1).$$
Therefore, we get $z_k(x)$ is uniformly bounded below by some constant that depends on $x$ only. Thus, we have
$z_k\rightarrow z$ in $C^2_{loc}(B_{r_0}(p)\setminus\{p\})$, where $z$ satisfies
\begin{align*}
\left\{\begin{array}{ll}
-\Delta z=f_1~&\mathrm{in}~B_{r_0}(p)\setminus\{p\},\\
z=-C~&\mathrm{on}~\partial B_{r_0}(P).
\end{array}\right.
\end{align*}
For $\varphi\in C_0^{\infty}(B_{r_0}(p))$,
\begin{align*}
\lim_{k\rightarrow+\infty}\int_{B_{r_0}(p)}\varphi\Delta z_k=&\int_{B_{r_0}(p)}(\varphi(x)-\varphi(p))\Delta z_k+\varphi(p)(\int_{B_{r_0}(p)}f_1+12\pi)\\
=&\int_{B_{r_0}(p)}\varphi(x)f_1+12\pi\varphi(p).
\end{align*}
Thus $-\Delta z=f_1+12\pi\delta_p.$ Therefore, we have $z(x)\geq 6\log\frac{1}{|x-p|}+O(1)$ as $x\rightarrow p$, which implies
$\int_{B_{r_0}(p)}e^{z}=\infty$, a contradiction. Hence
\begin{align}
\label{2.26}
\tilde{u}_{2k}\rightarrow-\infty~\mathrm{in}~B_{r_0}(p)\setminus \{p\}.
\end{align}
\end{proof}

Next, we prove Theorem \ref{th1.2} and derive the shadow system (\ref{1.11}).\\

\noindent{\em Proof of Theorem \ref{th1.2}.} As $\rho_{1k}\rightarrow 4\pi,\rho_{2k}\rightarrow\rho_2$ and $\rho_2\notin 4\pi\mathbb{N},$
we consider a sequence of solutions $(v_{1k},v_{2k})$ to (\ref{1.7}) such that $\max_{M}(v_{1k},v_{2k})\rightarrow+\infty.$ We claim
$\max_{M}(\tilde{u}_{1k},\tilde{u}_{2k})\rightarrow+\infty.$ Otherwise, $\tilde{u}_{1k},\tilde{u}_{2k}$ are
uniformly bounded above. From Green representation theorem and $L^p$ estimate, we can get $u_{1k},u_{2k}$ are uniformly bounded.
This implies $v_{1k},v_{2k}$ are uniformly bounded, which contradicts to our assumption. Let $\mathfrak{S}_i$ denotes the blow up point of
$\tilde{u}_{ik},~i=1,2$ as before.

We claim $\mathfrak{S}_2=\emptyset$ and $\mathfrak{S}_1$ consists of one point only. Suppose first $\mathfrak{S}_2\neq\emptyset.$
From the proof of Theorem \ref{th1.1}, if $\mathfrak{S}_1\cap\mathfrak{S}_2=\emptyset$, then $\tilde{u}_{2k}$
would concentrate, i.e., $\tilde{u}_{2k}\rightarrow-\infty,~\forall x\in M\setminus\mathfrak{S}_2$, which implies
$\rho_2=\lim_{k\rightarrow+\infty}\int_Mh_2e^{\tilde{u}_{2k}}\in 4\pi\mathbb{N},$ a contradiction. Thus $\mathfrak{S}_1\cap\mathfrak{S}_2\neq\emptyset.$
Suppose $q\in\mathfrak{S}_1\cap\mathfrak{S}_2,$ from Proposition 2.4 in \cite{jlw} and the condition $\rho_{1k}<8\pi,$ we conclude
$\sigma_{1q}=2,\sigma_{2q}=4.$ By Lemma \ref{le2.3}, we have $\tilde{u}_{2k}$ concentrate, which implies $\rho_2\in 4\pi\mathbb{N},$
a contradiction again. Hence $\mathfrak{S}_2=\emptyset$. By Lemma \ref{le2.2}, $\tilde{u}_{2k}$ is uniformly bounded from above in $M.$ Since $\max_{M}(\tilde{u}_{1k},\tilde{u}_{2k})\rightarrow+\infty,$ we get $\mathfrak{S}_1\neq\emptyset.$ By the fact $\rho_{1k}\rightarrow 4\pi,$ we have $\mathfrak{S}_1$ contains only one point.

We write the equation for $v_{ik},~i=1,2$ as
\begin{equation}
\label{2.27}
\left\{\begin{array}{ll}
\Delta v_{1k}+\rho_{1k}(\frac{h_1e^{2v_{1k}-v_{2k}}}{\int_Mh_1e^{2v_{1k}-v_{2k}}}-1)=0,\\
\Delta v_{2k}+\rho_{2k}(\frac{h_2e^{2v_{2k}-v_{1k}}}{\int_Mh_2e^{2v_{2k}-v_{1k}}}-1)=0.
\end{array}\right.
\end{equation}
Since $\tilde{u}_{2k}$ is uniformly bounded above, the second equation of (\ref{2.27}) implies that $v_{2k}$ is uniformly bounded in $M$ and converges to some
function $\frac12w$ in $C^{1,\alpha}(M).$ From the first equation of (\ref{2.27}) and $\rho_{1k}\rightarrow4\pi,$ $v_{1k}$ blows up at only one point, say $p\in M.$

We write the first equation in (\ref{2.27}) as
\begin{equation}
\label{2.28}
\Delta v_{1k}+\rho_{1k}(\frac{\tilde{h}_ke^{2v_{1k}}}{\int_{M}\tilde{h}_ke^{2v_{1k}}}-1)=0,
\end{equation}
where $\tilde{h}_k=h_1e^{-v_{2k}}$. We define $\tilde{v}_{1k}=v_{1k}-\frac12\log\int_{M}\tilde{h}_ke^{2v_{1k}}.$ Due to the $C^{1,\alpha}$ convergence of $\tilde{h}_k,$ $\tilde{v}_{1k}$ simply blows up at $p$ by a result of Li \cite{l} (one can also see \cite{bclt}), i.e., the following inequality holds:
\begin{align}
\label{2.29}
\Big|2\tilde{v}_{1k}-\log\frac{e^{\lambda_{k}}}{\big(1+\frac{\rho_{1k}\tilde{h}_k(p^{(k)})e^{\lambda_{k}}}{4}|x-p^{(k)}|^2\big)^2}\Big|<c~
\mathrm{for}~|x-p^{(k)}|<r_0,
\end{align}
where $\lambda_{k}=2\tilde{v}_{1k}(p^{(k)})=\max_{x\in B_{r_0}(p)}2\tilde{v}_{1k}.$ By using this sharp estimate, we get
\begin{equation}
\label{2.30}
\tilde{v}_{1k}\rightarrow-\infty~\mathrm{in}~M\setminus\{p\},\quad
\rho_{1k}\frac{h_1e^{2v_{1k}-v_{2k}}}{\int_Mh_1e^{2v_{1k}-v_{2k}}}\rightarrow 4\pi\delta_p,
\end{equation}
and
\begin{align}
\label{2.31}
\nabla\big(\log(h_1e^{-\frac12w})+4\pi R(x,x))\mid_{x=p}=0,
\end{align}
which proves (\ref{1.8}) and (\ref{1.9}).

In the following, we claim $v_{2k}\rightarrow\frac12w$ in $C^{2,\alpha}(M)$. From this claim and (\ref{2.28}), it is easy to get
$$v_{1k}\rightarrow 8\pi G(x,p)~\mathrm{in}~C^{2,\alpha}(M\setminus\{p\}).$$
Combined with $v_{2k}\rightarrow\frac12w$ in $C^{2,\alpha}(M)$, we have $w$ satisfies the following equation
\begin{equation}
\label{2.32}
\Delta w+2\rho_2(\frac{h_2e^{w-4\pi G(x,p)}}{\int_Mh_2e^{w-4\pi G(x,p)}}-1)=0.
\end{equation}
This proves (\ref{1.10}). Therefore, we finish the proof of Theorem \ref{1.2}. The proof of the claim is given in the following Lemma \ref{le2.4}. \quad\quad\quad\quad\quad\quad\quad\quad\quad\quad\quad\quad\quad\quad\quad\quad\quad $\square$

\begin{lemma}
\label{le2.4}
Let $v_{1k},v_{2k}$ be a sequence of blow up solutions of (\ref{2.27}), which $v_{1k}$ blows at $p$ and $v_{2k}\rightarrow\frac12w$ in $C^{1,\alpha}(M)$. Then $v_{2k}\rightarrow\frac12w$ in $C^{2,\alpha}(M).$
\end{lemma}

\begin{proof}
By (\ref{2.29}), we have
\begin{align}
\label{2.33}
|\lambda_{k}-\log\int_{M}\tilde{h}_ke^{2v_{1k}}|<c.
\end{align}
To prove $v_{2k}\rightarrow\frac12w$ in $C^{2,\alpha},$ we need the following estimate
\begin{align}
\label{2.34}
\Big|2\nabla\tilde{v}_{1k}-\nabla\Big(\log\frac{e^{\lambda_{k}}}{\big(1+\frac{\rho_{1k}h(p)e^{\lambda_{k}}}{4}|x-p|^2\big)^2}\Big)\Big|
<c~\mathrm{for}~|x-p|<r_0,
\end{align}
where (\ref{2.34}) comes from the error estimate of \cite[Lemma 4.1]{cl1}. We write
$$h_2e^{2v_{2k}-v_{1k}}=h_2e^{-v_{1k}}e^{2v_{2k}}.$$
By (\ref{2.29}) and (\ref{2.34}), it is not difficult to show
$$\nabla\big(h_2e^{-v_{1k}}\big)\in L^{\infty}(M).$$
Therefore, by classical elliptic regularity and Sobolev inequality, we can show that
\begin{equation}
\label{2.35}
v_{2k}\rightarrow\frac12w~\mathrm{in}~C^{2,\alpha}~\mathrm{for~any}~\alpha\in(0,1).
\end{equation}
Then we finished the proof of this lemma.
\end{proof}
\vspace{0.25cm}

After deriving the shadow system (\ref{1.11}), we show the non-degeneracy of (\ref{1.12}) by applying the well-known transversality theorem.\\

For convenience, we write (\ref{1.12}) as
\begin{equation}
\label{2.36}
\left\{\begin{array}{l}
\Delta w+2\rho_2\Big(\frac{h_2^*e^{w-4\pi\sum_{j=1}^mG(x,p_j^0)-F_0(x)}}
{\int_Mh_2^*e^{w-4\pi\sum_{j=1}^mG(x,p_j^0)-F_0(x)}}-1\Big)=0,\\
\nabla\big(\log(h_1^*e^{-\frac12w})+4\pi R(x,x)+ F_i(x)\big)\mid_{x=p_i^0}=0,~i=1,2,\cdots,m,
\end{array}\right.
\end{equation}
where
\begin{align*}
F_0(x)=4\pi\sum_{q\in S_2}\beta_qG(x,q)+4\pi\sum_{q\in S}(1+\alpha_q)G(x,q),
\end{align*}
and
\begin{align*}
F_{i}(x)=8\pi\sum_{j=1,j\neq i}^mG(x,p_j^0)+8\pi\sum_{q\in S}(1+\alpha_{q})G(x,q)-4\pi\sum_{q\in S_1}\alpha_qG(x,q).
\end{align*}

\noindent In order to show (\ref{2.36}) has a non-degenerate solution, we need the following theorem, which can be found in \cite{a}, \cite{q} and
references therein. First, we recall that
\begin{theorem}
\label{th2.1}
Let $F:\mathcal{H}\times \mathcal{B}\rightarrow \mathcal{E}$ be a $C^k$ map. $\mathcal{H},~\mathcal{B}$ and $\mathcal{E}$ Banach manifolds with
$\mathcal{H}$ and $\mathcal{E}$ separable. If $0$ is a regular value of $F$ and $F_b=F(\cdot,b)$ is a Fredholm map of index $<k,$ then the set
$\{b\in \mathcal{B}:0~is~a~regular~value~of~F_b\}$ is residual in $\mathcal{B}.$
\end{theorem}

We say $y\in \mathcal{E}$ is a regular value if every point $x\in F^{-1}(y)$ is a regular point, where
$x\in \mathcal{H}\times \mathcal{B}$ is a regular point of $F$ if $D_xF:T_x(\mathcal{H}\times \mathcal{B})\rightarrow T_{F(x)}\mathcal{E}$ is onto.
We say a set $A$ is a residual set if $A$ is a countable intersection of open dense sets, see \cite{a}, which implies $A$ is dense in $B$ ($B$ is a
Banach space), see \cite{k}.\\

Following the notations in Theorem \ref{th2.1}, we denote
$$\mathcal{H}=\big(M_*^m\setminus\Gamma^m\big)\times\mathring{W}^{2,\mathfrak{p}}(M),~\mathcal{B}=C^{2,\alpha}(M)\times C^{2,\alpha}(M),
~\mathcal{E}=(\mathbb{R}^2)^m\times\mathring{W}^{0,\mathfrak{p}}(M),$$
where
$$M_*=M\setminus S_1,~\Gamma^m:=\{(x_1,x_2,\cdots,x_m)\mid x_i\in M_*, x_i=x_j~\mathrm{for~some}~i=j\},$$
$$\mathring{W}^{2,\mathfrak{p}}(M):=\{f\in W^{2,\mathfrak{p}}\mid \int_Mf=0\},~
\mathring{W}^{0,\mathfrak{p}}(M):=\{f\in L^{\mathfrak{p}}\mid \int_Mf=0\},$$
and
$$C^{2,\alpha}(M)=\{f\in C^{2,\alpha}(M)\}.$$
\\

\begin{remark}
\label{re1}
Clearly, Theorem \ref{th2.1} is local in nature. Even though $M_*^m\setminus \Gamma_m$ is not a complete manifold, we can follow
the proof of the Transversality Theorem in \cite{q} with minor modification to get Theorem \ref{th2.1}, see \cite{q,s}.
\end{remark}

\noindent We consider the map
\begin{equation}
\label{2.37}
T(w,P_w,h_1^*,h_2^*)=\left[\begin{array}{l}
\Delta w+2\rho_2(\frac{h_2^*e^{w-4\pi\sum_{i=1}^mG(x,p_i^0)-F_0(x)}}
{\int_Mh_2^*e^{w-4\pi\sum_{i=1}^mG(x,p_i^0)-F_0(x)}}-1)\\
\nabla\log\big(h_1^*e^{-\frac12w}+4\pi R(x,x)\big)(p_1^0)+\nabla F_{1}(p_1^0)\\
~\quad\quad\quad\quad\quad\quad\quad\quad\vdots\\
\nabla\log\big(h_1^*e^{-\frac12w}+4\pi R(x,x)\big)(p_m^0)+\nabla F_{m}(p_m^0)
\end{array}\right].
\end{equation}
Clearly, $T$ is $C^1.$ Next, we claim
\begin{enumerate}
\item [(i)] $T(\cdot,\cdot,h_1^*,h_2^*)$ is a Fredholm map of index $0$,
\item [(ii)] $0$ is a regular value of $T.$
\end{enumerate}

\noindent For the first claim, after computation, we get
\begin{equation}
\label{2.38}
T'_{w,P_w}(w,P_w,h_1^*,h_2^*)[\phi,\nu_1,\cdots,\nu_m]=\left[\begin{array}{l}
T_0(w,P_w,h_1^*,h_2^*)[\phi,\nu_1,\cdots,\nu_m]\\
T_1(w,P_w,h_1^*,h_2^*)[\phi,\nu_1,\cdots,\nu_m]\\
~\quad\quad\quad\quad\quad\quad\vdots\\
T_m(w,P_w,h_1^*,h_2^*)[\phi,\nu_1,\cdots,\nu_m]
\end{array}\right],
\end{equation}
where
\begin{align*}
T_0(w,P_w,h_1^*,h_2^*)[\phi,\nu_1,\cdots,\nu_m]=&\Delta\phi+2\rho_2\frac{\hat{h}_2e^w}
{\int_M\hat{h}_2e^w}\phi-2\rho_2\frac{\hat{h}_2e^w}
{(\int_M\hat{h}_2e^w)^2}\int_M\hat{h}_2e^w\phi\\
&-8\pi\rho_2\sum_{i=1}^m\frac{\hat{h}_2e^w}
{\int_M\hat{h}_2e^w}\nabla G(x,p_i^0)\cdot\nu_i\\
&+8\pi\rho_2\sum_{i=1}^m\frac{\hat{h}_2e^w}
{(\int_M\hat{h}_2e^w)^2}\int_M\hat{h}_2e^w\nabla G(x,p_i^0)\cdot\nu_i,
\end{align*}
\begin{align*}
T_i(w,P_w,h_1^*,h_2^*)[\phi,\nu_1,\cdots,\nu_m]=&\nabla^2_x(\log h_1^*e^{-\frac12w}+4\pi R(x,x)+F_{i})\mid_{x=p_i^0}\cdot \nu_i\\
&+\mathcal{F}_i-\frac12\nabla\phi(p_i^0)~\mathrm{for}~i=1,2,\cdots,m,
\end{align*}
where
$$\hat{h}_2=h_2^*e^{-4\pi\sum_{i=1}^mG(x,p_i^0)-F_0(x)}$$
and
$$\mathcal{F}_i=8\pi\sum_{j=1,j\neq i}^m\nabla_x^2G(p_i^0,x)\mid_{x=p_j^0}\cdot\nu_j.$$

We decompose
\begin{align}
\label{2.39}
T'_{w,P_w}[\phi,\nu_1,\cdots,\nu_m]=\left[\begin{array}{l}
T_{01}\\
T_{11}\\
~\vdots\\
T_{m1}
\end{array}\right][\phi,\nu_1,\cdots,\nu_m]+
\left[\begin{array}{l}
T_{02}\\
T_{12}\\
~\vdots\\
T_{m2}
\end{array}\right][\phi,\nu_1,\cdots,\nu_m],
\end{align}
where
\begin{align*}
&T_{01}(w,P_w,h_1^*,h_2^*)[\phi,\nu_1,\cdots,\nu_m]=\Delta\phi+2\rho_2\frac{\hat{h}_2e^w}
{\int_M\hat{h}_2e^w}\phi-2\rho_2\frac{\hat{h}_2e^w}
{(\int_M\hat{h}_2e^w)^2}\int_M\hat{h}_2e^w\phi,\\
&T_{02}(w,P_w,h_1^*,h_2^*)[\phi,\nu_1,\cdots,\nu_m]=-8\pi\rho_2\sum_{i=1}^m\frac{\hat{h}_2e^w}
{\int_M\hat{h}_2e^w}\nabla G(x,p_i^0)\nu_i\\
&\quad\quad\quad\quad\quad\quad\quad\quad\quad\quad\quad\quad\quad\quad\quad
+8\pi\rho_2\sum_{i=1}^m\frac{\hat{h}_2e^w}
{(\int_M\hat{h}_2e^w)^2}\int_M\hat{h}_2e^w\nabla G(x,p_i^0)\nu_i,\\
&T_{i1}=0,\quad T_{i2}=T_i,~\mathrm{for}~i=1,2,\cdots,m.
\end{align*}
We define $\mathfrak{T}_1=\left[\begin{array}{l}
T_{01}\\
T_{11}\\
~\vdots\\
T_{m1}
\end{array}\right]$ and $\mathfrak{T}_2=\left[\begin{array}{ll}
T_{02}\\
T_{12}\\
~\vdots\\
T_{m2}
\end{array}\right].$ We can easily see that $\mathfrak{T}_1$ is symmetric, it follows from the basic theory of elliptic operators that $\mathfrak{T}_1$
is a Fredholm operator of index $0.$ Combining the Sobolev inequality and $(\mathbb{R}^2)^m$ is a finite Euclidean space, we can show that
$\mathfrak{T}_2$ is a compact operator. Therefore, by the standard linear operator theory \cite{k}, we get $\mathfrak{T}_1+\mathfrak{T}_2$ is
also a Fredholm linear operator with index $0$. Hence, we prove the first claim that $T$ is a Fredholm map with index $0.$

It remains to show that $0$ is a regular value. We derive the differentiation of the operator $T$ with respect to $h_1^*$ and $h_2^*,$
\begin{align*}
T_{h_1^*}'(w,P_w,h_1^*,h_2^*)[H_1]=
\left[\begin{array}{l}
\quad\quad\quad\quad\quad0\\
\frac{\nabla H_1}{h_1^*}(p_1^0)-\frac{\nabla h_1^*}{(h_1^*)^2}H_1(p_1^0)\\
\quad\quad\quad\quad\quad\vdots\\
\frac{\nabla H_1}{h_1^*}(p_m^0)-\frac{\nabla h_1^*}{(h_1^*)^2}H_1(p_m^0)\\
\end{array}\right],
\end{align*}
and
\begin{align*}
T_{h_2^*}'(w,P_w,h_1^*,h_2^*)[H_2]=
\left[\begin{array}{l}
2\rho_2\frac{\hat{h}_2e^w}{\int_M\hat{h}_2e^w}\frac{H_2}{h_2^*}-
2\rho_2\frac{\hat{h}_2e^w}{(\int_M\hat{h}_2e^w)^2}\int_M\hat{h}_2e^w\frac{H_2}{h_2^*}\\
\quad\quad\quad\quad\quad\quad\quad\quad\quad0\\
\quad\quad\quad\quad\quad\quad\quad\quad\quad\vdots\\
\quad\quad\quad\quad\quad\quad\quad\quad\quad0\\
\end{array}\right].
\end{align*}
By choosing $\nu_1=\nu_2=\cdots=\nu_m=0,$ and $H_1$ such that $\frac{\nabla H_1}{h_1^*}-\frac{\nabla h_1^*}{(h_1^*)^2}H_1=\frac12\nabla\phi$ at
$p_i^0,i=1,2,\cdots,m.$ We get
\begin{align*}
&T'_{w,P_w}(w,P_w,h_1^*,h_2^*)[\phi,\nu_1,\cdots,\nu_m]+T_{h_1}'(w,P_w,h_1^*,h_2^*)[H_1]\\
&=\left[\begin{array}{l}
\Delta\phi+2\rho_2\frac{\hat{h}_2e^w}{\int_M\hat{h}_2e^w}\phi-
2\rho_2\frac{\hat{h}_2e^w}{(\int_M\hat{h}_2e^w)^2}\int_M\hat{h}_2e^w\phi\\
\quad\quad\quad\quad\quad\quad\quad\quad\quad\quad0\\
\quad\quad\quad\quad\quad\quad\quad\quad\quad\quad\vdots\\
\quad\quad\quad\quad\quad\quad\quad\quad\quad\quad0\\
\end{array}\right].
\end{align*}
Next, we claim that the vector space spanned by $T'_{w,P_w}(w,P_w,h_1^*,h_2^*)[\phi,\nu_1,\cdots,\nu_m]$, $T_{h_1^*}'(w,P_w,h_1^*,h_2^*)[H_1]$ and
$T_{h_2^*}'(w,P_w,h_1^*,h_2^*)[H_2]$ contains $\left[\begin{array}{l}f\\0\\ \vdots\\0\end{array}\right]$ for all $f\in \mathring{W}^{0,p}.$ It is enough
for us to prove that only $\phi=0$ can satisfy
$$\phi\in\mathrm{Ker}\Big\{\Delta\cdot+2\rho_2\frac{\hat{h}_2e^w}{\int_M\hat{h}_2e^w}\cdot
-2\rho_2\frac{\hat{h}_2e^w}{(\int_M\hat{h}_2e^w)^2}\int_M\hat{h}_2e^w\cdot\Big\}$$
and
$$\Big\l\phi,~2\rho_2\frac{\hat{h}_2e^w}{\int_M\hat{h}_2e^w}\frac{H_2}{h_2^*}-2\rho_2\frac{\hat{h}_2e^w}
{(\int_M\hat{h}_2e^w)^2}\int_M\hat{h}_2e^w\frac{H_2}{h_2^*}\Big\r=0,$$
for all $H_2\in C^{2,\alpha}(M).$ We set
$$L=\Delta\cdot+2\rho_2\frac{\hat{h}_2e^w}{\int_M\hat{h}_2e^w}\cdot
-2\rho_2\frac{\hat{h}_2e^w}{(\int_M\hat{h}_2e^w)^2}\int_M\hat{h}_2e^w\cdot.$$
Using $\phi\in\mathrm{Ker}(L),$ we obtain that for any
$\overline{H}_2\in W^{0,p}(M),$
\begin{align}
\label{2.40}
\int_{M}\Big(\Delta\phi+2\rho_2\frac{\hat{h}_2e^w}{\int_M\hat{h}_2e^w}\phi
-2\rho_2\frac{\hat{h}_2e^w}{(\int_M\hat{h}_2e^w)^2}\int_M\hat{h}_2e^w\phi\Big)\cdot\overline{H}_2=0,
\end{align}
Since $C^{2,\alpha}(M)$ is dense in $W^{0,p}(M)$ and
$$\Big\l\phi,2\rho_2\frac{\hat{h}_2e^w}{\int_M\hat{h}_2e^w}\overline{H}_2-2\rho_2\frac{\hat{h}_2e^w}
{(\int_M\hat{h}_2e^w)^2}\int_M\hat{h}_2e^w\overline{H}_2\Big\r=0,$$
we deduce
\begin{equation}
\label{2.41}
\int_M\Delta\phi\cdot\overline{H}_2=0,\quad \forall~\overline{H}_2\in W^{0,p}(M).
\end{equation}
Thus
\begin{equation}
\label{2.42}
\Delta\phi=0~\mathrm{in}~M,\quad\int_M\phi=0.
\end{equation}
So $\phi\equiv0.$ Therefore the claim is proved.
\vspace{0.25cm}

On the other hand, we choose two functions, $H_{1,1}$ and $H_{1,2}$ such that
$$H_{1,i}(p_j^0)=0,\nabla H_{1,i}(p_j^0)=0,~2\leq j\leq m,~i=1,2.$$
Based on this choice, we can further make such $H_{1,i},i=1,2$ that
$$\frac{\nabla H_{1,1}}{h_1^*}(p_1^0)-\frac{\nabla h_1^*}{(h_1^*)^2}H_{1,1}(p_1^0)=(1,0),$$
and
$$\frac{\nabla H_{1,2}}{h_1^*}(p_1^0)-\frac{\nabla h_1^*}{(h_1^*)^2}H_{1,2}(p_1^0)=(0,1).$$
Then it is not difficult to see that (By setting $\phi=0,\nu_1=\nu_2=\cdots=\nu_m=0$)
$$\left[\begin{array}{l}0\\c\\0\\ \vdots\\0\end{array}\right]\subset DT(w,P_w,h_1^*,h_2^*)[\phi,\nu_1,\cdots,\nu_m,H_1,H_2]$$
for all $c\in \mathbb{R}^2.$ Similarly, we can show
$$\left[\begin{array}{l}0\\c_1\\c_2\\\vdots\\c_m\end{array}\right]\subset DT(w,P_w,h_1^*,h_2^*)[\phi,\nu_1,\cdots,\nu_m,H_1,H_2]$$
for all $c_i\in \mathbb{R}^2,~i=1,2,\cdots,m.$ Therefore, we proved that the differential map is onto. As a consequence, $0$ is a regular point of $T.$
By Theorem \ref{th2.1}, we have
$$\Big\{(h_1^*,h_2^*)\in \mathcal{B}:0~\mathrm{is~a~regular~value~of}~T(\cdot,\cdot,h_1^*,h_2^*)\Big\}$$
is residual in $\mathcal{B}.$ Since $T(w,P_w,h_1^*,h_2^*)$ is a Fredholm map of index $0$ for fixed $h_1^*,h_2^*,$ we have
$$\Big\{(h_1^*,h_2^*)\in \mathcal{B}:~\mathrm{the~solution}~(w,P_w)~\mathrm{of}~T(\cdot,\cdot,h_1^*,h_2^*)=0~\mathrm{is~nondegenerate}\Big\}$$
is residual in $\mathcal{B}.$ Thus, we can choose $h_1^*,h_2^*>0$ such that the solution of (\ref{1.12}) is non-degenerate.

\vspace{1cm}
\section{Apriori estimate}
In this section, we shall prove that all the blow up solutions of (\ref{1.7}) must be contained in the set $S_{\rho_1}(p,w)\times S_{\rho_2}(p,w)$ when $\rho_1\rightarrow4\pi,$ $\rho_2\notin4\pi\mathbb{N}$, where the definition of $S_{\rho_i}(p,w),i=1,2$ is given in (\ref{3.14}) and
(\ref{3.15}) of this section.\\

To simplify our description, we may assume $M$ has a flat metric near a neighborhood of each blow up point. Of course we can modify our arguments without any difficulty for the general case, as in \cite{cl2}.\\

We start to define the set $S_{\rho_i}(p,w)$. For any given non-degenerate solution $(p,w)$ of (\ref{1.11}), we set
\begin{equation}
\label{3.1}
h=h_1e^{-\frac12w}.
\end{equation}
Note that
\begin{equation}
\label{3.2}
\nabla_x(\log h+4\pi R(x,x))\mid_{x=p}=\nabla_x\big(\log h(x)+8\pi R(x,p)\big)\mid_{x=p}=0,
\end{equation}
whenever $(p,w)$ is a solution of shadow system (\ref{1.11}). For $q$ such that $|q-p|\ll1$ and large $\lambda>0,$ we set
\begin{align}
\label{3.3}
U(x)=\lambda-2\log\big(1+\frac{\rho_1h(q)}{4}e^{\lambda}|x-q|^{2}\big),
\end{align}
and $U(x)$ satisfies the following equation
\begin{align}
\label{3.4}
\Delta U(x)+2\rho_1h(q)e^{U}=0~\mathrm{in}~\mathbb{R}^2,\quad U(q)=\max_{\mathbb{R}^2}U(x)=\lambda.
\end{align}
Let
\begin{align}
\label{3.5}
H(x)=\exp\Big\{\log\frac{h(x)}{h(q)}+8\pi R(x,q)-8\pi R(q,q)\Big\}-1,
\end{align}
and
\begin{equation}
\label{3.6}
s=\lambda+2\log\Big(\frac{\rho_1h(q)}{4}\Big)+8\pi R(q,q)+\frac{\Delta H(q)}{\rho_1h(q)}\frac{\lambda^2}{e^{\lambda}}.
\end{equation}

Let $\sigma_0(t)$ be a cut-off function:
\begin{equation*}
\sigma_0(t)=\left\{
\begin{array}{ll}
1,~&\mathrm{if}~|t|<r_0,\\
0,~&\mathrm{if}~|t|\geq 2r_0.
\end{array}\right.
\end{equation*}
Set $\sigma(x)=\sigma_0(|x-q|)$ and
\begin{align*}
J(x)=\left\{\begin{array}{ll}
\big(H(x)-\nabla H(p)\cdot(x-p)\big)\sigma,~&x\in B_{2r_0}(q),\\
0,&x\notin B_{2r_0}(q).
\end{array}\right.
\end{align*}
Let $\eta(x)$ satisfy
\begin{equation}
\label{3.7}
\left\{\begin{array}{l}
\Delta\eta+2\rho_1 h(q)e^{U}(\eta+J(x))=0\quad\mathrm{on}~\mathbb{R}^2,\\
\eta(q)=0,\nabla\eta(q)=0.
\end{array}\right.
\end{equation}
The existence of $\eta$ was proved in \cite{cl2}. Furthermore, we have the following lemma
\begin{lemma}
\label{le3.1}
Let $R=\sqrt{\frac{\rho_1h(q)}{4}e^{\lambda}}.$ For $h\in C^{2,\alpha}(M)$ and large $\lambda.$ there exists a solution $\eta$ satisfying (\ref{3.7}) and
the following
\begin{itemize}
  \item [(i)] $\eta(x)=-\frac{4\Delta H(q)}{\rho_1h(q)}e^{-\lambda}[\log(R|x-q|+2)]^2+O(\lambda e^{-\lambda})~\mathrm{on}~B_{2r_0}(q)$,
  \item [(ii)] $\eta,\nabla_x\eta,\partial_q\eta,\partial_{\lambda}\eta,\nabla_x\partial_q\eta,\nabla_x\partial_{\lambda}\eta=
              O(\lambda^2e^{-\lambda})~\mathrm{on}~B_{2r_0}(q)$.
\end{itemize}
\end{lemma}
\noindent The proof of Lemma \ref{le3.1} was given in \cite{cl2}. \\

We set
\begin{align}
\label{3.8}
\left\{\begin{array}{ll}
v_q(x)=\Big(U(x)+\eta(x)+8\pi(R(x,q)-R(q,q))+s\Big)\sigma(x)\\
\quad\quad\quad\quad+8\pi G(x,q)(1-\sigma(x)),\\
\overline{v}_q=\frac{1}{|M|}\int_Mv_q,\\
v_{q,\lambda,a}=a(v_q-\overline{v}_q).
\end{array}\right.
\end{align}

Note that $v_q(x)$ depends on $q$ and $\lambda.$ Next, we define $O^{(1)}_{q,\lambda}$
and $O^{(2)}_{q,\lambda}$:
\begin{align}
\label{3.9}
O^{(1)}_{q,\lambda}=&\Big\{\phi\in\mathring{H}^1(M)~\Big{|}~\int_M\nabla\phi\cdot\nabla v_{q}=\int_M\nabla\phi\cdot\nabla\partial_{q}v_{q}=
\int_M\nabla\phi\cdot\nabla\partial_{\lambda}v_{q}=0\Big\},
\end{align}
and
\begin{equation}
\label{3.10}
O^{(2)}_{q,\lambda}=\Big\{\psi\in W^{2,\mathfrak{p}}(M)~\Big{|}~\int\psi=0\Big\},\quad \mathfrak{p}>2.
\end{equation}
For each $(q,\lambda)$, we define
\begin{equation}
\label{3.11}
t=\lambda+8\pi R(q,q)+2\log \frac{\rho_1h(q)}{4}+\frac{\Delta H(q)}{\rho_1 h(q)}\lambda^2e^{-\lambda}-\overline{v}_q.
\end{equation}
For $\rho_1\neq 4\pi$, we define $\lambda(\rho_1)$ such that
\begin{align}
\label{3.12}
\rho_1-4\pi=\frac{\Delta\log h(p)+8\pi-2K(p)}{h(p)}\lambda(\rho_1)e^{-\lambda(\rho_1)},
\end{align}
where $(p,w)$ is the non-degenerate solution of (\ref{1.11}) and $K(p)$ denotes the Gaussian curvature of $p$. By using the equation (\ref{1.10}), we have $e^{-4\pi G(x,p)}\mid_{x=p}=0$ and $\Delta w(p)=2\rho_2.$ Thus
\begin{align}
\label{3.13}
\Delta\log h(p)+8\pi-2K(p)=&\Delta\log h_1(p)-\rho_2+8\pi-2K(p).
\end{align}
Obviously, $\lambda(\rho_1)$ can be well-defined only if
$$\Delta\log h_1(p)-\rho_2+8\pi-2K(p)\neq0.$$

Let $c$ be a positive constant, which will be chosen later. By using $\rho_1$, we set
\begin{align}
\label{3.14}
S_{\rho_1}(p,w)=\Big\{&v_1=\frac12v_{q,\lambda,a}+\phi~\Big{|}~|q-p|\leq c\lambda(\rho_1)e^{-\lambda(\rho_1)},
|\lambda-\lambda(\rho_1)|\leq c\lambda(\rho_1)^{-1},\nonumber\\
&|a-1|\leq c\lambda(\rho_1)^{-\frac12}e^{-\lambda(\rho_1)},\phi\in O_{q,\lambda}^{(1)}
~\mathrm{and}~\|\phi\|_{H^1(M)}\leq c\lambda(\rho_1)e^{-\lambda(\rho_1)}\Big\},
\end{align}
and
\begin{equation}
\label{3.15}
S_{\rho_2}(p,w)=\Big\{v_2=\frac12w+\psi~\Big{|}~\psi\in O_{q,\lambda}^{(2)}~\mathrm{and}~\|\psi\|_*\leq c\lambda(\rho_1)e^{-\lambda(\rho_1)}\Big\},
\end{equation}
where $\|\psi\|_*=\|\psi\|_{W^{2,\mathfrak{p}}(M)}.$\\

Now suppose $(v_{1k},v_{2k})$ is a sequence of bubbling solutions of (\ref{1.7}) such that $v_{1k}$ blows up at $p$ and weakly converges to
$4\pi G(x,p)$, while $v_{2k}\rightarrow\frac12w$ in $C^{2,\alpha}(M)$. Then we want to prove that
$$(v_{1k},v_{2k})\in S_{\rho_1}(p,w)\times S_{\rho_2}(p,w).$$
First of all, we prove the following lemma.

\begin{lemma}
\label{le4.1}
Let $(v_{1k},v_{2k})$ be a sequence of blow up solutions of (\ref{1.7}), which $v_{1k}$ blows up at $p$, weakly converges to $4\pi G(x,p)$ and
$v_{2k}\rightarrow\frac12w$ in $C^{2,\alpha}(M).$ Suppose $(p,w)$ is a non-degenerate solution of (\ref{1.11}) and
\begin{align}
\label{3.16}
\Delta\log h_1(p)-\rho_2+8\pi-2K(p)\neq 0.
\end{align}
Then there exist $c>0,$ $q_k$, $\lambda_{k},$ $a_k,$ $\phi_k,$ $\psi_k$ such that
\begin{equation}
\label{3.17}
v_{1k}=\frac12 v_{q_k,\lambda_k,a_k}+\phi_k,~\quad v_{2k}=\frac12w+\psi_k,
\end{equation}
and
$(v_{1k},v_{2k})\in S_{\rho_1}(p,w)\times S_{\rho_2}(p,w)$.
\end{lemma}

\begin{remark}
\label{re2}
Because the proof of this lemma is very long, we describe the process briefly. First of all, we have to obtain a good approximation of $v_{1k}$. Since $v_{2k}$ converges to $\frac12w$ in $C^{2,\alpha}(M),$ this fine estimate can be obtained by the same proof in Lemma \cite{cl1}. Next, we substitute $v_{1k}$ into the second equation of $v_{2k}.$ Then we use the non-degeneracy of (\ref{1.11}) to get the sharp estimates of $\psi_k$ and $|\tilde{q}_k-p|,$ where $\psi_k=v_{2k}-\frac12w$ and $\tilde{q}_k$ is the point where $v_{1k}$ obtains its maximal value. After that, we get the lemma. In the following proof, we use the same notation as the proof of Theorem \ref{th1.2}.
\end{remark}

\begin{proof} Let $v_{1k}$ and $v_{2k}$ be a sequence of blow up solutions of (\ref{1.7}).
\begin{equation}
\label{3.18}
\left\{\begin{array}{l}
\Delta v_{1k}+\rho_{1k}(\frac{h_1e^{2v_{1k}-v_{2k}}}{\int_{M}h_1e^{2v_{1k}-v_{2k}}}-1)=0,\\
\Delta v_{2k}+\rho_{2k}(\frac{h_2e^{2v_{2k}-v_{1k}}}{\int_{M}h_2e^{2v_{2k}-v_{1k}}}-1)=0.
\end{array}\right.
\end{equation}
For convenience, we write the first equation in (\ref{3.18}) as,
\begin{equation}
\label{3.19}
\Delta v_{1k}+\rho_{1k}(\frac{\tilde{h}_ke^{2v_{1k}}}{\int_{\Omega}\tilde{h}_ke^{2v_{1k}}}-1)=0,
\end{equation}
where
\begin{equation}
\label{3.20}
\tilde{h}_k=h_1e^{-v_{2k}}=he^{-\psi_k}~\mathrm{and}~\psi_k=v_{2k}-\frac12w.
\end{equation}
Since $\tilde{h}_k\rightarrow h$ in $C^{2,\alpha}(M)$, all the estimates in \cite{cl1} can be applied to our case here, although in \cite{cl1} the coefficient $\tilde{h}_k$ is independent of $k$. In the followings (up to (\ref{3.28}) below), we sketch the estimates in \cite{cl1, cl2} which will be used here.
We denote $\tilde{q}_k$ to be the maximal point of $\tilde{v}_{1k}$ near $p$, where
$\tilde{v}_{1k}=v_{1k}-\frac12\log\int_M\tilde{h}_ke^{2v_{1k}}.$ Let
\begin{align*}
\lambda_{k}=2\tilde{v}_{1k}(\tilde{q}_{k})-\log\int_M\tilde{h}_ke^{2v_{1k}}.
\end{align*}
In the local coordinate near $\tilde{q}_{k},$ we set
\begin{equation*}
\tilde{U}_{k}(x)=\log\frac{e^{\lambda_k}}{(1+\frac{\rho_{1k}\tilde{h}_k(q_{k})}{4}e^{\lambda_{k}}|x-q_{k}|^2)^2}.
\end{equation*}
where $q_k$ is chosen such that
$$\nabla\tilde{U}_k(\tilde{q}_k)=\nabla\log h(\tilde{q}_k),$$
clearly $|q_k-\tilde{q}_k|=O(e^{-\lambda_k}).$
Then the error term inside $B_{r_0}(q_k)$ is set by
\begin{equation}
\label{3.21}
\tilde{\eta}_{k}(x)=2\tilde{v}_{1k}-\tilde{U}_{k}(y)-(8\pi R(x,q_k)-8\pi R(q_k,q_{k})),
\end{equation}
and the error term outside $B_{r_0}(q_k)$ is set by
\begin{equation}
\label{3.22}
\xi_k(x)=2v_{1k}(x)-8\pi G(x,q_{k}).
\end{equation}
By Green's representation for $v_{1k}$, it is not difficult to obtain
\begin{equation}
\label{3.23}
\xi_k(x)=O(\lambda_ke^{-\lambda_k})~\mathrm{for}~x\in M\setminus B_{r_0}(q_k).
\end{equation}
By a straightforward computation, the error term $\tilde{\eta}_k$ satisfies
\begin{equation}
\label{3.24}
\Delta\tilde{\eta}_{k}+2\rho_{1k}\tilde{h}_k(q_{k})e^{\tilde{U}_{k}}\tilde{H}_{k}(x,\tilde{\eta}_{k})=0,
\end{equation}
where
\begin{align*}
\tilde{H}_{k}(x,t)=&\exp\{\log\frac{\tilde{h}_k(x)}{\tilde{h}_k(q_{k})}+8\pi\big(R(x,q_k)-R(q_k,q_{k})\big)+t\}-1\\
=&H_{k}(x)+t+O(|t|^2),
\end{align*}
and
\begin{align*}
H_k(x)=\exp\Big\{\log \frac{\tilde{h}_k(x)}{\tilde{h}_k(q_k)}+8\pi R(x,q_k)-8\pi R(q_k,q_k)\Big\}-1.
\end{align*}
We see that except for the higher-order term $O(|\tilde{\eta}_{k}|^2)$, equation (\ref{3.24}) is exactly like (\ref{3.7}). By Lemma \ref{le3.1}, we can prove
\begin{equation}
\label{3.25}
\tilde{\eta}_{k}(x)=-\frac{4}{\rho_{1k}\tilde{h}_k(q_{k})}\Delta\log H_k(q_{k})e^{-\lambda_{k}}[\log(R_{k}|x-q_{k}|+2)]^2
+O(\lambda_{k}e^{-\lambda_{k}})
\end{equation}
for $x\in B_{2r_0}(q_{k}),$ where $R_k=\sqrt{\frac{\rho_{1k}\tilde{h}_k(q_{k})}{4}e^{\lambda_{k}}}.$

From \cite[Theorem 1.1, Theorem 1.4 and Lemma 5.4]{cl1}, we have
\begin{equation}
\label{3.26}
\rho_{1k}-4\pi=\frac{\Delta\log\tilde{h}_k(q_k)+8\pi-2K(q_k)}{\tilde{h}_k(q_{k})}\lambda_{k}e^{-\lambda_{k}}+O(e^{-\lambda_{k}}),
\end{equation}
\begin{align}
\label{3.27}
2\overline{\tilde{v}}_{1k}+\lambda_{k}+2\log\frac{\rho_{1k}\tilde{h}_k(q_{k})}{4}+8\pi R(q_k,q_{k})+
\frac{\Delta H_k(q_{k})}{\rho_{1k}\tilde{h}_k(q_{k})}\lambda_k^2e^{-\lambda_{k}}=O(\lambda_{k}e^{-\lambda_{k}}),
\end{align}
and
\begin{equation}
\label{3.28}
|\nabla H_{k}(q_{k})|=O(\lambda_{k}e^{-\lambda_{k}}).
\end{equation}

Now we let $\eta_{k}$ be defined as in (\ref{3.7}), $v_{q_k}$ and $v_{q_k,\lambda_k,a_k}$ be defined as in (\ref{3.8}) with $q=q_k$, $\lambda=\lambda_k$ and $a=a_k=1$. By
Lemma \ref{le3.1}, (\ref{3.25}) and (\ref{3.28}), we have
\begin{align}
\label{3.29}
\eta_{k}(x)=\tilde{\eta}_{k}+O(\lambda_{k}e^{-\lambda_{k}})~\mathrm{for}~x\in B_{2r_0}(q_k).
\end{align}
Note that for $x\in B_{r_0}(q_k),$
\begin{align*}
v_{q_k,\lambda_k,a_k}=&\tilde{U}_{k}(x)+\eta_{k}(x)+\big(8\pi R(x,q_k)-8\pi R(q_k,q_{k})\big)+\lambda_{k}+2\log\frac{\rho_{1k}\tilde{h}_k(q_{k})}{4}\\
&+8\pi R(q_k,q_{k})+\frac{\Delta H_k(q_{k})}{\rho_{1k}\tilde{h}_k(q_{k})}\lambda_{k}^2e^{-\lambda_{k}}-\overline{v}_{q_k},
\end{align*}
where $\overline{v}_{q_k}$ denotes the average of $v_{q_k}$. From \cite[Lemma 2.2 and Lemma 2.3]{cl2}, we have
\begin{align}
\label{3.30}
v_{q_k}-4\pi G(x,q_k)=O(\lambda_ke^{-\lambda_k})~\mathrm{in}~M\setminus B_{2r_0}(q_k),~\mathrm{and}~\overline{v}_{q_k}=O(\lambda_ke^{-\lambda_k}).
\end{align}
By (\ref{3.21}), (\ref{3.27}), (\ref{3.29}) and (\ref{3.30}), we have
\begin{align}
\label{3.31}
2v_{1k}-v_{q_k,\lambda_k,a_k}=&2\tilde{v}_{1k}+\int_{M}\tilde{h}_ke^{2v_{1k}}-v_{q_k,\lambda_k,a_k}\nonumber\\
=&2\tilde{v}_{1k}-\tilde{U}_{k}-\big(8\pi R(x,x)-8\pi R(x,q_{k})\big)-\eta_{k}(x)+O(\lambda_{k}e^{-\lambda_{k}})\nonumber\\
=&\tilde{\eta}_{k}(x)-\eta_{k}(x)+O(\lambda_{k}e^{-\lambda_{k}})=O(\lambda_{k}e^{-\lambda_{k}})
\end{align}
for $x\in B_{r_0}(q_k)$. For $x\in M\setminus B_{2r_0}(q_k),$ by (\ref{3.22}) and (\ref{3.30}), we get
\begin{align*}
2v_{1k}-v_{q_k,\lambda_k,a_k}=2v_{1k}-8\pi G(x,q_k)-(v_{q_k}-8\pi G(x,q_k))+\overline{v}_{q_k}=O(\lambda_ke^{-\lambda_k}).
\end{align*}
For the intermediate domain $B_{2r_0}(q_k)\setminus B_{r_0}(q_k),$ following a similar way,
we can obtain that $2v_{1k}-v_{q_k,\lambda_k,a_k}=O(\lambda_{k}e^{-\lambda_{k}}).$ Thus, we find a good approximation
$\frac12v_{q_k,\lambda_k,a_k}$ for $v_{1k}.$ For convenience, we write
\begin{align}
\label{3.32}
v_{1k}=\frac12v_{q_k,\lambda_k,a_k}+\phi_{k},~\mathrm{where}~\|\phi_{k}\|_{L^{\infty}(M)}<\tilde{c}\lambda_ke^{-\lambda_k},
\end{align}
where $\tilde{c}$ is independent of $\psi_k.$

Next, we substitute (\ref{3.32}) and $v_{2k}=\frac12w+\psi_k$ into the second equation of (\ref{3.18}),
after computation, we obtain
\begin{align}
\label{3.33}
\mathcal{L}\psi_{k}=\mathbb{I}_1+\mathbb{I}_2+\mathbb{I}_3,\quad \int_M\psi_k=0,
\end{align}
where
\begin{align*}
L\psi_k=&\Delta\psi_{k}+2\rho_2\frac{h_2e^{w-4\pi G(x,p)}}{\int_{M}h_2e^{w-4\pi G(x,p)}}\psi_{k}\\
&-2\rho_2\frac{h_2e^{w-4\pi G(x,p)}}
{(\int_Mh_2e^{w-4\pi G(x,p)})^2}\int_{M}(h_2e^{w-4\pi G(x,p)}\psi_{k})\\
&-4\pi\rho_2\frac{h_2e^{w-4\pi G(x,p)}}{\int_{M}h_2e^{w-4\pi G(x,p)}}\big(\nabla G(x,p)(q_{k}-p)\big)\\
&+4\pi\rho_2\frac{h_2e^{w-4\pi G(x,p)}}{(\int_{M}h_2e^{w-4\pi G(x,p)})^2}
\int_{M}\big(h_2e^{w-4\pi G(x,p)}(\nabla G(x,p)(q_{k}-p))\big),
\end{align*}
\begin{align*}
\mathbb{I}_1=&-\rho_2\frac{h_2e^{w+2\psi_{k}-v_{1k}}}{\int_{M}h_2e^{w+2\psi_{k}-v_{1k}}}
+\rho_2\frac{h_2e^{w+2\psi_{k}-4\pi G(x,q_k)}}
{\int_{M}h_2e^{w+2\psi_{k}-4\pi G(x,q_k)}},
\end{align*}
\begin{align*}
\mathbb{I}_2=&\rho_2\frac{h_2e^{w-4\pi G(x,p)}}{\int_{M}h_2e^{w-4\pi G(x,p)}}
-\rho_2\frac{h_2e^{w+2\psi_{k}-4\pi G(x,p)}}{\int_{M}h_2e^{w+2\psi_{k}-4\pi G(x,p)}}\\
&+2\rho_2\frac{h_2e^{w-4\pi G(x,p)}}{\int_{M}h_2e^{w-4\pi G(x,p)}}\psi_{k}\\
&-2\rho_2\frac{h_2e^{w-4\pi G(x,p)}}
{(\int_{M}h_2e^{w-4\pi G(x,p)})^2}\int_{M}(h_2e^{w-4\pi G(x,p)}\psi_{k}),
\end{align*}
and
\begin{align*}
\mathbb{I}_3=&-\rho_2\frac{h_2e^{w+2\psi_{k}-4\pi G(x,q_k)}}{\int_{M}h_2e^{w+2\psi_{k}-4\pi G(x,q_k)}}
+\rho_2\frac{h_2e^{w+2\psi_{k}-4\pi G(x,p)}}{\int_{M}h_2e^{w+2\psi_{k}-4\pi G(x,p)}}\\
&-4\pi\rho_2\frac{h_2e^{w-4\pi G(x,p)}}{\int_{M}h_2e^{w-4\pi G(x,p)}}(\nabla G(x,p)(q_{k}-p))\\
&+4\pi\rho_2\frac{h_2e^{w-4\pi G(x,p)}}{(\int_{M}h_2e^{w-4\pi G(x,p)})^2}
\int_{M}\big(h_2e^{w-4\pi G(x,p)}(\nabla G(x,p)(q_{k}-p))\big).
\end{align*}
We shall analyze the right hand side of (\ref{3.33}) term by term in the following. For $\mathbb{I}_1,$ we set
\begin{align*}
\mathfrak{E}_1=\exp\big(w+2\psi_k-4\pi G(x,q_{k})\big)-\exp\big(w+2\psi_k-\frac12 v_{q_k,\lambda_k,a_k}-\phi_k\big).
\end{align*}

For $x\in M\setminus B_{r_0}(q_{k}).$ We see that the difference between $4\pi G(x,q_{k})$ and $v_{q_k,\lambda_k,a_k}$ is of order
 $\lambda_ke^{-\lambda_k}.$ As a consequence, $\mathfrak{E}_1=O(\lambda_ke^{-\lambda_k}).$

For $x\in B_{r_0}(q_{k}),$
\begin{align*}
4\pi G(x,q_{k})-\frac12v_{q_k,\lambda_k,a_k}=&4\pi G(x,q_{k})-4\pi R(x,q_{k})-\log(\frac{\rho_1\tilde{h}_k(q_{k})e^{\lambda_k}}{4})\\
&+\log(1+\frac{\rho_1\tilde{h}_k(q_{k})e^{\lambda_{k}}}{4}|x-q_{k}|^2)-\lambda_k\\
&-\frac12(\eta_{k}+\frac{\Delta H_k(q_{k})}{\rho_1\tilde{h}_k(q_{k})}\frac{\lambda_{k}^2}{e^{\lambda_{k}}})+O(\lambda_{k}e^{-\lambda_{k}})\\
=&\log(\frac{4}{\rho_1\tilde{h}_k(q_{k})e^{\hat{\lambda}_{k}}|x-q_{k}|^2}+1)\\
&-\frac12(\eta_{k}+\frac{\Delta H_k(q_{k})}{\rho_1\tilde{h}_k(q_{k})}\frac{\lambda_{k}^2}{e^{\lambda_{k}}})+O(\lambda_{k}e^{-\lambda_{k}}).
\end{align*}
Since $\phi_k=O(\lambda_{k}e^{-\lambda_{k}}),$
\begin{align*}
\exp\big(w+2\psi_k-\frac12 v_{q_k,\lambda_k,a_k}-\phi_k\big)=\exp\big(w+2\psi_k-\frac12 v_{q_k,\lambda_k,a_k}\big)+O(\lambda_{k}e^{-\lambda_{k}}).
\end{align*}
Then, we have
\begin{align*}
&\exp\big(w+2\psi_k-4\pi G(x,q_{k})\big)-\exp\big(w+2\psi_k-\frac12 v_{q_k,\lambda_k,a_k}-\phi_k\big)\\
&~=\exp\big(w+2\psi_k-4\pi G(x,q_{k})\big)-\exp\big(w+2\psi_k-\frac12 v_{q_k,\lambda_k,a_k}\big)+O(\lambda_{k}e^{-\lambda_{k}})\\
&~=\exp\big(w+2\psi_{k}-4\pi G(x,q_{k})\big)\big(1-\exp\big(4\pi G(x,q_{k})-\frac12 v_{q_k,\lambda_k,a_k}\big)\big)+O(\lambda_{k}e^{-\lambda_{k}})\\
&~=\exp\big(w+2\psi_{k}-4\pi G(x,q_{k})\big)\Big(1-\exp\Big[\log(1+\frac{4}{\rho_1\tilde{h}_ke^{\lambda_k}|x-q_k|^2})\\
&\quad\quad+O(\eta_k+
\frac{\Delta H_k(q_k)}{\rho_1\tilde{h}_k(q_k)}\frac{\lambda_k^2}{e^{\lambda_k}})+O(\lambda_{k}e^{-\lambda_{k}})\Big]\Big)
\end{align*}
When $|x-q_k|=O(e^{-\frac{\lambda_k}{2}}),$ we have
\begin{align*}
\exp\big(w+2\psi_k-4\pi G(x,q_k)\big)=O(e^{-\lambda_k}),
\end{align*}
and
\begin{align*}
\log(1+\frac{4}{\rho_1\tilde{h}_ke^{\lambda_k}|x-q_k|^2})+O(\eta_k+
\frac{\Delta H_k(q_k)}{\rho_1\tilde{h}_k(q_k)}\frac{\lambda_k^2}{e^{\lambda_k}})=O\big(\log(e^{-\lambda_k}|x-q_k|^{-2})\big),
\end{align*}
hence
\begin{align*}
\mathfrak{E}_1=O(\lambda_{k}e^{-\lambda_{k}})~\mathrm{for}~|x-q_k|=O(e^{-\frac{\lambda_k}{2}}).
\end{align*}
When $|x-q_k|\gg e^{-\frac{\lambda_k}{2}},$ then
\begin{align*}
1-\exp\big(\log(1+\frac{4}{\rho_1\tilde{h}_ke^{\lambda_k}|x-q_k|^2})+&O(\eta_k+
\frac{\Delta H_k(q_k)}{\rho_1\tilde{h}_k(q_k)}\frac{\lambda_k^2}{e^{\lambda_k}})\big)+O(\lambda_{k}e^{-\lambda_{k}})\\
&=O(\frac{4}{\rho_1\tilde{h}_ke^{\lambda_k}|x-q_k|^2}+\lambda_ke^{-\lambda_k}),
\end{align*}
as a result, we have
$$\mathfrak{E}_1=O(\lambda_{k}e^{-\lambda_{k}})~\mathrm{for}~r_0\geq|x-q_k|\gg e^{-\frac{\lambda_k}{2}}.$$
Thus, $\|\mathfrak{E}_1\|_{L^{\infty}(M)}=O(\lambda_{k}e^{-\lambda_{k}}).$
This implies $\mathbb{I}_1=O(\lambda_{k}e^{-\lambda_{k}}).$

For the second term, it is easy to see that $\mathbb{I}_2=O(\|\psi_k\|_{*}).$ It remains to estimate $\mathbb{I}_3.$ We divide it into three parts.
\begin{align*}
\mathbb{I}_3=\mathbb{I}_{31}+\mathbb{I}_{32}+\mathbb{I}_{33},
\end{align*}
where
\begin{align*}
\mathbb{I}_{31}=&-\rho_2\frac{h_2e^{w+2\psi_{k}-4\pi G(x,q_{k})}}
{\int_{M}h_2e^{w+2\psi_{k}-4\pi G(x,q_{k})}}+\rho_2\frac{h_2e^{w+2\psi_{k}-4\pi G(x,p)}}
{\int_{M}h_2e^{w+2\psi_{k}-4\pi G(x,p)}}\\
&-4\pi\rho_2\frac{h_2e^{w+2\psi_k-4\pi G(x,p)}}{\int_{M}h_2e^{w+2\psi_k-4\pi G(x,p)}}\big(\nabla G(x,p)(q_{k}-p)\big),\\
&+4\pi\rho_2\frac{h_2e^{w+2\psi_k-4\pi G(x,p)}}{(\int_{M}h_2e^{w+2\psi_k-4\pi G(x,p)})^2}
\int_{M}\Big(h_2e^{w+2\psi_k-4\pi G(x,p)}(\nabla G(x,p)(q_{k}-p))\Big),
\end{align*}
\begin{align*}
\mathbb{I}_{32}=&4\pi\rho_2\frac{h_2e^{w-4\pi G(x,p)}}{(\int_{M}h_2e^{w-4\pi G(x,p)})^2}
\int_{M}\Big(h_2e^{w-4\pi G(x,p)}(\nabla G(x,p)(q_{k}-p))\Big)\\
&-4\pi\rho_2\frac{h_2e^{w+2\psi_k-4\pi G(x,p)}}{(\int_{M}h_2e^{w+2\psi_k-4\pi G(x,p)})^2}
\int_{M}\Big(h_2e^{w+2\psi_k-4\pi G(x,p)}(\nabla G(x,p)(q_{k}-p))\Big),
\end{align*}
and
\begin{align*}
\mathbb{I}_{33}=&4\pi\rho_2\frac{h_2e^{w+2\psi_k-4\pi G(x,p)}}{\int_{M}h_2e^{w+2\psi_k-4\pi G(x,p)}}\big(\nabla G(x,p)(q_{k}-p)\big)\\
&-4\pi\rho_2\frac{h_2e^{w-4\pi G(x,p)}}{\int_{M}h_2e^{w-4\pi G(x,p)}}\big(\nabla G(x,p)(q_{k}-p)\big).
\end{align*}
It is not difficult to see
$$\mathbb{I}_{31}=O(|q_k-p|^2),~\mathbb{I}_{32}=O(1)\|\psi\|_*|q_k-p|,~\mathbb{I}_{33}=O(1)\|\psi\|_*|q_k-p|.$$
Then (\ref{3.33}) can be written as
\begin{align}
\label{3.34}
\mathcal{L}(\psi_k)=o(1)\|\psi_k\|_*+O(\|\psi_k\|^2_{*}+\lambda_ke^{-\lambda_k})+O(|p-q_k|^2).
\end{align}
By the definition of $H_k$ and (\ref{3.28}), we have
\begin{equation}
\label{3.35}
\nabla H_k(q_k)=\nabla\log h(q_k)-\nabla\psi(q_k)+8\pi\nabla R(q_k,q_k)=O(\lambda_ke^{-\lambda_k}).
\end{equation}
By (\ref{3.2}) and (\ref{3.35}), we have
\begin{align}
\label{3.36}
\nabla^2\big{(}\log h(p)+8\pi R(p,p)\big{)}(q_{k}-p)-\nabla\psi_k(p)=&\nabla\log h(q_k)-\nabla\psi(q_k)+8\pi\nabla R(q_k,q_k)\nonumber\\
&-\big(\nabla\log h(p)+8\pi\nabla R(p,p)\big)\nonumber\\
&+\nabla\psi(q_k)-\nabla\psi_k(p)\nonumber\\
&+O(|p-q|^2)\nonumber\\
=&\nabla H_k(q_k)-\nabla H(p)+O(|p-q|^{\gamma}\|\psi_k\|_*)\nonumber\\
&+O(|p-q|^2),
\end{align}
where $\gamma$ depends on $\mathfrak{p}$. We note that $\nabla H(p)=0$. From (\ref{3.34})-(\ref{3.36}) and the non-degeneracy of $p,w$, we obtain
\begin{equation}
\label{3.37}
\|\psi_k\|_*+|p-q_k|\leq C(\lambda_ke^{-\lambda_k}+o(1)\|\psi_k\|_*+\|\psi_k\|^2_{*}+|p-q_k|^2),
\end{equation}
where $C$ is a generic constant, independent of $k$ and $\psi_k.$ Therefore, we have
\begin{equation}
\label{3.38}
\psi_{k}=O(\lambda_ke^{-\lambda_k}),~|p-q_k|=O(\lambda_ke^{-\lambda_k}).
\end{equation}
As a conclusion of (\ref{3.12}), (\ref{3.26}) and (\ref{3.38}), we have
\begin{equation}
\label{3.39}
\lambda_k-\lambda(\rho_1)=O(\lambda(\rho_1)^{-1}),~\tilde{h}_k=h+O(\lambda(\rho_1)e^{-\lambda(\rho_1)}),~|q_{k}-p|=O(\lambda(\rho_1)e^{-\lambda(\rho_1)})
\end{equation}
and
\begin{equation}
\label{3.40}
v_{2k}-\frac12w=O(\lambda(\rho_1)e^{-\lambda(\rho_1)}).
\end{equation}
We replace $\tilde{h}_k$ by $h$ in the definition of $v_{q}$, we denote the new terms by $v_{q}.$
By (\ref{3.38}), we have
\begin{align*}
v_{q_k}-v_{q}=O(\lambda(\rho_1)e^{-\lambda(\rho_1)}).
\end{align*}
We set
\begin{align}
\label{3.41}
v_{q,\lambda,a}=v_{q}-\overline{v}_{q}.
\end{align}
By (\ref{3.32}) and (\ref{3.41}), we gain
\begin{align}
\label{3.42}
v_{1k}-\frac12v_{q,\lambda,a}=O(\lambda(\rho_1)e^{-\lambda(\rho_1)}).
\end{align}
By \cite[Lemma 3.2]{cl2}, if we choose $c$ in $S_{\rho_1}(p,w)$ big enough, there exists triplet $(q^*,\lambda^*,a^*)$ and
$\phi^*\in O^{(1)}_{q^*,\lambda^*}$ such that
\begin{equation}
\label{3.43}
v_{1k}=\frac12v_{q^*,\lambda^*,a^*}+\phi_k^*,
\end{equation}
where $q^*,~\lambda^*,~a^*$ satisfy the condition in $S_{\rho_1}(p,w)$.  Therefore, we proved
$$(v_{1k},v_{2k})\in S_{\rho_1}(p,w)\times S_{\rho_2}(p,w).$$
\end{proof}
\vspace{0.25cm}

In conclusion, we have the
following Theorem,
\begin{theorem}
\label{th3.1}
Suppose $h_1,h_2$ are two positive $C^{2,\alpha}$ function on $M$ such that any solution $(p,w)$ of (\ref{1.11}) is non-degenerate and
$\Delta\log h_1(p)-\rho_2+8\pi\neq0.$ Then there exists $\varepsilon_0>0$ and $C>0$ such that
for any solution of (\ref{1.7}) with $\rho_1\in(4\pi-\varepsilon_0,4\pi+\varepsilon_0),\rho_2\notin4\pi\mathbb{N}$, either
$|v_1|,|v_2|\leq C,\forall x\in M$ or $(v_1,v_2)\in S_{\rho_1}(p,w)\times S_{\rho_2}(p,w)$ for some solution $(p,w)$ of (\ref{1.11}).
\end{theorem}
\vspace{1cm}

\section{Approximate blow up solution}
In the following two sections. We shall construct the blow up solutions of (\ref{1.7}) in general case when
\begin{equation}
\label{4-0}
\rho_1\rightarrow\rho_*=4m\pi+4\pi\sum_{q\in S}\alpha_q~\mathrm{and}~\rho_2\notin\Sigma_2.
\end{equation}
The construction of such bubbling solution is based on a
non-degenerate solution of (\ref{1.12}). For a given non-degenerate solution $(P_w,w)$ of (\ref{1.12}). We define the space
$S_{\rho_1}(Q,w)$, $S_{\rho_2}(Q,w)$ for $v_{1}$ and $v_{2}$ respectively, where the definition of $S_{\rho_i}(Q,w)$ is given in
(\ref{4.19})-(\ref{4.20}). These two sets are generalization of the one defined in the previous section.
Our aim is to compute the degree of the following nonlinear operator
$$\left(\begin{array}{l}v_1\\v_2\end{array}\right)=
(-\Delta)^{-1}\left(\begin{array}{l}
2\rho_1\big(\frac{h_1e^{2v_1-v_2}}{\int_{M}h_1e^{2v_1-v_2}}-1\big)\\
2\rho_2\big(\frac{h_2e^{2v_2-v_1}}{\int_Mh_2e^{2v_2-v_1}}-1\big)
\end{array}\right)$$
in the space $S_{\rho_1}(Q,w)\times S_{\rho_2}(Q,w)$.

For a give non-degenerate solution $(p_w,w)$ of (\ref{1.12}). We define
$$Q=P_w\cup S=\{p_1^0,p_2^0,\cdots,p_m^0\}\cup S,~P_w\cap S_1=\emptyset,~S\subseteq S_1.$$
In order to simplify our notation, we relabel the index in $Q$ and set
$$Q=\{p_{1}^0,p_{2}^0,\cdots,p_{n}^0\},~n=m+|S|$$
and
$$\alpha_i=0~\mathrm{for}~1\leq i\leq m,~\alpha_{m+i}=\alpha_{p_{m+i}^0},~1\leq i\leq n-m.$$

For each point $p_j^0$ in $Q,$ we set
\begin{align}
\label{4.1}
G_j^*(x)=8\pi\Big((1+\alpha_j)R(x,p_j^0)+\sum_{1\leq i\leq n,i\neq j}(1+\alpha_i)G(x,p_i^0)\Big),~j=1,2,\cdots,n.
\end{align}
We only consider the case $S\neq\emptyset,$ because the construction for the case $S=\emptyset$ is similar. we may assume
(relabeling the index if necessary)
\begin{align*}
\alpha_{m+1}=\cdots=\alpha_{m+l}>\alpha_{m+l+1}\geq\cdots\geq\alpha_{m+n},
\end{align*}
We use $G_j^*$ associate with $Q$ to define $l(Q)$ as follows.

\begin{align}
\label{4.2}
l(Q)=\sum_{j={m+1}}^{m+l}\big(\frac{h_{p_j}(p_{j}^0)\rho_*}{4(1+\alpha_{j})^2}\big)^{\frac{1}{1+\alpha_{j}}}
e^{\frac{G_j^*(p_j^0)}{1+\alpha_{j}}}\big(\Delta\log(h_1^*e^{-\frac12w})(p_j^0)+2\rho_*-N^*-2K(p_j^0)\big),
\end{align}
where $K(p)$ is the Gaussian curvature at $p$, $\rho_*=4\pi\sum_{j=1}^n(1+\alpha_j)$, $N^*=4\pi\sum_{q\in S_1}\alpha_q,$ and
\begin{align*}
h_{p_j}(x)=\left\{\begin{array}{ll}
h(x),~&\mathrm{if}~1\leq j\leq m,\\
|x-p_j^0|^{-2\alpha_j}h(x),~&\mathrm{if}~m+1\leq j\leq n.
\end{array}\right.
\end{align*}

Let $P=(p_1,p_2,\cdots,p_n)$ with $|p_i-p_i^0|\ll1$ for $1\leq i\leq m$ and $p_i=p_i^0$ for $m< i\leq n.$ For large
$\lambda_{j}>0,~j=1,2,\cdots,n,$ we set
\begin{align}
\label{4.3}
U_{j}(x)=\lambda_{j}-2\log\big(1+\frac{\rho_1h_{p_j}(p_j)}{4(1+\alpha_j)^2}e^{\lambda_{j}}|x-p_j|^{2(1+\alpha_j)}\big).
\end{align}

These $U_{j}(x)$ satisfy the following equation
\begin{align}
\label{4.4}
\Delta U_{j}(x)+2\rho_1h_{p_j}(p_j)e^{U_j}=0~\mathrm{in}~\mathbb{R}^2,\quad U_j(p_j)=\max_{\mathbb{R}^2}U_{j}(x)=\lambda_{j},
\end{align}
By using $h_{p_j}(x)$, we define $H_j(x,t)$,
\begin{align}
\label{4.5}
H_{j}(x,t)=\exp\big\{\log\frac{h_{p_j}(x+p_j)}{h_{p_j}(p_j)}+(G_{j}^*(x+p_j)-G_{j}^*(p_j))+t\big\}-1,
\end{align}
For convenience, we set
\begin{align*}
J=\{1,2,\cdots,n\},\quad J_1=\{1,2,\cdots,m\},\quad J_2=\{m+1,m+2,\cdots,m+l\}.
\end{align*}

Next, we construct the error terms near each point $p_j$ for $j\in J_2.$ With out loss of generality, we may assume
$\nabla H_{j}(0,0)=(e_j,0)$. Let $Q(x)=\frac12\big(\nabla^2\log[H_j(0,0)+1]x,x\big).$ Then the Taylor expansion gives
\begin{equation}
\label{4.6}
\log\frac{h_{p_j}(x+p_j)}{h_{p_j}(p_j)}+G_{j}^*(x+p_j)-G_{j}^*(p_j)=e_jx_1+Q(x)+\mathrm{higher~order},
\end{equation}
and
\begin{equation}
\label{4.7}
H_{j}(x,t)=e_jx_1+t+Q(x)+\frac12(e_jx_1+t)^2+O(|x|^3+t^3).
\end{equation}
For $j\in J_2$, we let $\zeta_{1,j}(y)$ and $\zeta_{2,j}(y)$ be the solutions of
\begin{equation}
\label{4.8}
\left\{\begin{array}{l}
\Delta\zeta_{1,j}(y)+2\rho_1h_{p_j}(p_j)|y|^{2\alpha_j}e^{U'_j(y)}(\zeta_{1,j}(y)+e_jy_1)=0~\mathrm{in}~\mathbb{R}^2,\\
\zeta_{1,j}(0)=0,\quad|\zeta_{1,j}(y)|=O(|y|^{-2\alpha-1})~\mathrm{at}~\infty,
\end{array}\right.
\end{equation}
and
\begin{equation}
\label{4.9}
\left\{\begin{array}{l}
\Delta\zeta_{2,j}(y)+2\rho_1h_{p_j}(p_j)|y|^{2\alpha_j}e^{U_j'(y)}
\Big(\zeta_{2,j}(y)+Q(y)+\frac12(e_jy_1+\zeta_{1,j}(y))^2\Big)=0~\mathrm{in}~\mathbb{R}^2,\\
\zeta_{2,j}(0)=0,\quad|\zeta_{2,j}(y)|=O(\log|y|)~\mathrm{at}~\infty,
\end{array}\right.
\end{equation}
where
$$U_j'(y)=-2\log\Big(1+\frac{\rho_1h_{p_j}(p_j)}{4(1+\alpha_j)^2}|y|^{2(1+\alpha_j)}\Big).$$
The existence of $\zeta_{1,j}$ and $\zeta_{2,j}$ has been proved in section 3 of \cite{cl3}. Set $\epsilon_{j}=e^{-\frac{\lambda_{j}}{2(1+\alpha_j)}}.$ For $j\in J\setminus J_1$, we define
\begin{equation}
\label{4.10}
\eta_{j}(x)=\epsilon_{j}\zeta_{1,j}(\epsilon_{j}^{-1}x)+\epsilon_{j}^2\zeta_{2,j}(\epsilon_{j}^{-1}x)~\mathrm{for}~|x|\leq2r_0.
\end{equation}
By (\ref{4.7}), (\ref{4.8}) and (\ref{4.9}), $\eta_{j}$ satisfies,
\begin{equation}
\label{4.11}
\Delta\eta_{j}+2\rho_1h_{p_j}(p_j)|x-p_j|^{2\alpha_j}e^{U_{j}(x)}\tilde{H}_{j}(x,\eta_{j})=0,
\end{equation}
where
\begin{align*}
\tilde{H}_{j}(x,\eta_{j})=&e_jx_1+\eta_{j}+Q(x)+\frac12(e_jx_1+\eta_{j})^2
-\frac12(\epsilon_{j}^2\zeta_{2,j}(\epsilon_{j}^{-1}x))^2\\
&-\epsilon_{j}^2\zeta_{2,j}(\epsilon_{j}^{-1}x)\big(e_jx_1+\epsilon_j\zeta_{1,j}(\epsilon_j^{-1}x)\big).
\end{align*}
If $j\in J_1,$ we set $\eta_j\equiv 0.$\\

For $j\in J$, we let
\begin{equation}
\label{4.12}
s_{j}=\lambda_{j}+2\log\Big(\frac{\rho_1h_{p_j}(p_j)}{4(1+\alpha_j)^2}\Big)+8\pi(1+\alpha_j)R(p_j,p_j)+\frac{d_j}{2(1+\alpha_j)}
\lambda_je^{-\frac{\lambda_j}{1+\alpha_j}},
\end{equation}
where
$$d_{j}=\frac{\pi}{(1+\alpha_j)\sin(\frac{\pi}{1+\alpha_j})}\Big(\frac{4(1+\alpha_j)^2}{\rho_1h_{p_j}(p_j)}\Big)^{\frac{1}{1+\alpha_j}}
\times\big(\Delta\log(h_1^*e^{-\frac12w})(p_j)+2\rho_*-N^*-2K(p_j)\big),$$
if $j\in J\setminus J_1$, and $d_j=0$ for $j\in J_1.$\\

Let $\sigma(x)$ be a cut-off function:
\begin{equation*}
\sigma(x)=\left\{
\begin{array}{ll}
1,~&\mathrm{if}~|x|<r_0,\\
0,~&\mathrm{if}~|x|\geq 2r_0,
\end{array}\right.
\end{equation*}
and $\sigma_{j}(x)=\sigma(x-p_j)$. We set
\begin{align}
\label{4.13}
v_{p_j}(x)=&\Big(U_{j}(x)+\eta_{j}(x)+8\pi(1+\alpha_j)(R(x,p_j)-R(p_j,p_j))+s_{j}\Big)\sigma_{j}(x)\nonumber\\
&+8\pi(1+\alpha_j)G(x,p_j)(1-\sigma_{j}).
\end{align}
We note that $\eta_j(x)+\frac{d_j}{2(1+\alpha_j)}\lambda_je^{-\frac{\lambda_j}{1+\alpha_j}}=O(e^{-\frac{\lambda_j}{1+\alpha_j}})$ if $m+1\leq j\leq n$ and
$|x|\geq\delta>0$ for some $\delta>0.$

Next, we state two lemmas that shall be used later. One can see \cite{cl3} for a proof
\begin{lemma}
\label{le4.1}
Let $\xi_{j}(x)=v_{p_j}(x)-8\pi(1+\alpha_j)G(x,p_j).$ Then for $r_0\leq|x-p_j|\leq2r_0,$ the followings hold,
\begin{enumerate}
  \item For $m+1\leq j\leq n,$ $\xi_j(x),\partial_{\lambda_{j}}\xi_{j}(x),\nabla_x\xi_{j}(x),\Delta_x\xi_{j}(x)$ are
$O(e^{-\frac{\lambda_{j}}{1+\alpha_j}}).$
  \item For $1\leq j\leq m,$ $\xi_j(x),\partial_{p_j}\xi_j(x),\partial_{\lambda_{j}}\xi_{j}(x),\nabla_x\xi_{j}(x),\Delta_x\xi_{j}(x)$ are $O(\lambda_je^{-\lambda_j}).$
\end{enumerate}
\end{lemma}
\noindent and
\begin{lemma}
\label{le4.2} For $v_{p_j},$ we have
\begin{enumerate}
  \item For $m+1\leq j\leq n,$ $\int_Mv_{p_j},~\int_M\partial_{\lambda_j}v_{p_j}=O(e^{-\frac{\lambda_{j}}{1+\alpha_j}}),$
  \item For $1\leq j\leq m,$ $\int_Mv_{p_j},~\int_M\partial_{\lambda_j}v_{p_j},~\int_M\partial_{p_j}v_{p_j}(x)=O(\lambda_je^{-\lambda_{j}}).$
\end{enumerate}
\end{lemma}

Now we are going to construct a good approximation of the blow up solution to (\ref{1.7}). For each
$$P=(p_1,p_2,\cdots,p_n),\Lambda=(\lambda_1,\lambda_{2},\cdots,\lambda_{n})\in\mathbb{R}^{n},A=(a_{1},a_{2},\cdots,a_{n})\in\mathbb{R}^{n},$$
we set
$$v_{P,\Lambda,A}=\sum_{j=1}^{n}a_{j}\big(v_{p_j}-\overline{v}_{p_j}\big),$$
where $v_{p_j}$ is constructed in (\ref{4.13}) and $\overline{v}_{p_j}$ denotes the average of $v_{p_j}$. We define
\begin{equation}
\label{4.14}
t_j=\lambda_j+G_j^*(p_j)+2\log\Big(\frac{\rho_*h_{p_j}(p_j)}{4(1+\alpha_j)^2}\Big)+\frac{d_j}{2(1+\alpha_j)}\lambda_je^{-\frac{\lambda_j}{1+\alpha_j}}
-\sum_{j=1}^{n}\overline{v}_{p_j}~\mathrm{for}~1\leq j\leq n.
\end{equation}
We recall the term $l(Q):$
\begin{align}
\label{4.15}
l(Q)=&\sum_{j\in J_2}\big(\frac{h_{p_j}(p_{j}^0)\rho_*}{4(1+\alpha_{j})^2}\big)^{\frac{1}{1+\alpha_{j}}}
e^{\frac{G_j^*(p_j^0)}{1+\alpha_{j}}}\big(\Delta\log(h_1^*e^{-\frac12w})(p_j^0)+2\rho_*-N^*-2K(p_j^0)\big).
\end{align}
In the proof, we assume
\begin{align}
\label{4-Q}
l(Q)\neq0,
\end{align}
and then define $\lambda(P)$ by
\begin{align}
\label{4.16}
\rho_1-\rho_*
=\frac{\pi^2}{(1+\alpha_{m+1})\sin\frac{\pi}{1+\alpha_{m+1}}}\big(\frac{4(1+\alpha_{m+1})^2}{\rho_*h_{p_{m+1}}(p_{m+1}^0)}\big)^
{\frac{2}{1+\alpha_{m+1}}}e^{-\frac{G_{m+1}^*(p_{m+1}^0)}{1+\alpha_{m+1}}}l(Q)e^{-\frac{\lambda(P)}{1+\alpha_{m+1}}}.
\end{align}

\begin{remark}
\label{re3}
We note that the main components in $l(Q)$ are
\begin{equation}
\label{4-1}
\Delta\log(h_1^*e^{-\frac12w})(p_j^0)+2\rho_*-N^*-2K(p_j^0)=\Delta\log h_1^*-\rho_2+2\rho_*-N^*-2K(p_j^0).
\end{equation}
For $h_1^*,h_2^*$, using Theorem \ref{th2.1}, we can find a dense set in $C^{2,\alpha}(M)\times C^{2,\alpha}(M)$ to make all the solutions $(P_w,w)$ of (\ref{1.12}) non-degenerate. Based on this choice, we can choose such $h_1^*$ that
\begin{align}
\label{4-2}
&\|\Delta\log h_1^*-\rho_2-2K\|_{L^{\infty}(M)}<|2\rho_*-N^*|,~&\mathrm{if}~2\rho_*-N^*\neq0,\nonumber\\
&\|\Delta\log h_1^*-2K\|_{L^{\infty}(M)}<\rho_2,~&\mathrm{if}~2\rho_*-N^*=0,
\end{align}
where $\rho_*$ and $N^*$ are given in (\ref{4.2}). In conclusion, we can always choose $h_1^*,h_2^*$ such that the solutions of (\ref{1.12}) are
non-degenerate and the term $l(Q)\neq0.$
\end{remark}

For each $P,\Lambda$, we define the following function space
\begin{align}
\label{4.17}
O_{P,\Lambda}^{(1)}=&\Big\{\phi\in\mathring{H}^1(M)~\Big{|}~\int_M\nabla\phi\cdot\nabla v_{p_j}=\int_M\nabla\phi\cdot\nabla\partial_{p_j}v_{p_j}
=\int_M\nabla\phi\cdot\nabla\partial_{\lambda_j}v_{p_j}=0\nonumber\\
&\mathrm{for}~1\leq j\leq m,~\mathrm{and}~\int_M\nabla\phi\cdot\nabla v_{p_j}=\int_M\nabla\phi\cdot\nabla\partial_{\lambda_{j}}v_{p_j}=0\nonumber\\
&\mathrm{for}~m+1\leq j\leq n\Big\}
\end{align}
and
\begin{equation}
\label{4.18}
O_{P,\Lambda}^{(2)}=\Big\{\psi\in W^{2,\mathfrak{p}}(M)~\Big{|}~\int_M\psi=0\Big\}.
\end{equation}
\vspace{0.20cm}

\noindent Let $c_0$ be a positive constant, which will be chosen later. By using (\ref{4.13}), (\ref{4.14}), (\ref{4.16}), (\ref{4.17}) and (\ref{4.18}),
we define
\begin{align}
\label{4.19}
S_{\rho_1}(Q,w)=&\Big\{v_1=\frac12v_{P,\Lambda,A}+\phi~\Big{|}~|p_j-p_j^0|\leq c_0e^{-\frac{\lambda(P)}{1+\alpha_{m+1}}}~
\mathrm{for}~1\leq j\leq m,
\nonumber\\
&|\lambda_{m+1}-\lambda(P)|\leq c_0\lambda(P)^{-1},~|t_j-t_1|\leq c_0e^{-\frac{\lambda(P)}{1+\alpha_{m+1}}}~\mathrm{for}~2\leq j\leq n,\nonumber\\
&|a_j-1|\leq c_0\lambda(P)^{-\frac12}e^{-\frac{\lambda(P)}{1+\alpha_{m+1}}}~\mathrm{for}~1\leq j\leq n,\nonumber\\
&\phi\in O_{P,\Lambda}^{(1)}~\mathrm{and}~\|\phi\|_{H^1(M)}\leq c_0e^{-\frac{\lambda(P)}{1+\alpha_{m+1}}}\Big\}
\end{align}
and
\begin{equation}
\label{4.20}
S_{\rho_2}(Q,w)=\Big\{v_2=\frac12w+\psi~\Big{|}~\psi\in O_{P,\Lambda}^{(2)}~\mathrm{and}~\|\psi\|_{*}\leq c_0
e^{-\frac{\lambda(P)}{1+\alpha_{m+1}}}\Big\},
\end{equation}
where $\|\psi\|_*=\|\psi\|_{W^{2,\mathfrak{p}}(M)},~\mathfrak{p}> 2.$\\

Next, we want to reduce the computation of the topological degree contributed from
$\big(S_{\rho_1}(Q,w)\times S_{\rho_2}(Q,w)\big)$ to a easier problem.

Set
\begin{align*}
T(v_1,v_2)=\left(\begin{array}{l}T_1(v_1,v_2)\\T_2(v_1,v_2)\end{array}\right)
=\Delta^{-1}\left(\begin{array}{l}
2\rho_1\big(\frac{h_1e^{2v_1-v_2}}{\int_{M}h_1e^{2v_1-v_2}}-1\big)\\
2\rho_2\big(\frac{h_2e^{2v_2-v_1}}{\int_Mh_2e^{2v_2-v_1}}-1\big)
\end{array}\right).
\end{align*}
Since each solution in $v_1$ in $S_{\rho_1}(Q,w)$ can be represented by $(P,\Lambda,A,\phi),$ each solution in $v_2$ in $S_{\rho_2}(Q,w)$ can be
represented by $w$ and $\psi.$ Therefore the nonlinear operator $2v_1+T_1(v_1,v_2)$ can be split according to this representation.

Let $v_1=\frac12v_{P,\Lambda,A}+\phi\in S_{\rho_1}(Q,w).$ Recall
$$t_j=s_j+8\pi\sum_{i\neq j}(1+\alpha_i)G(p_j,p_i)-\sum_{i=1}^{n}\overline{v}_{p_i}$$
and for $x\in B_{r_0}(p_j)$
\begin{align*}
v_{P,\Lambda,A}(x)+\log\frac{h_{p_j}(x)}{h_{p_j}(p_j)}=&U_j+t_j+\log(H_j(x-p_j,\eta_j)+1)+(a_j-1)(U_j+s_j)\\
&+8\pi\sum_{i\neq j}(1+\alpha_i)(a_i-1)G(p_j,p_i)+O(|a_j-1|(|y|+|\eta_j|)),
\end{align*}
where $y=x-p_j.$ The above together with (\ref{4.6}) implies

\begin{align}
\label{4.21}
\rho_1h_1e^{2v_1-v_2-2\phi+\psi}=&\rho_1he^{v_{P,\Lambda,A}}=\rho_1h_{p_j}(p_j)|y|^{2\alpha_j}e^{U_j+t_j}\Big[1+(a_j-1)(U_j+s_j)\nonumber\\
&+\sum_{i\neq j,i=1}^{n}8\pi(1+\alpha_i)(a_i-1)G(p_j,p_i)+\eta_j+\nabla_yH_j(0,0)\cdot y\nonumber\\
&+Q_j(y)+(a_j-1)(U_j+s_j)(\nabla_yH_j(0,0)\cdot y+\eta_j)\nonumber\\
&+\frac12(\eta_j+\nabla_yH_j\cdot y)^2+O(\tilde{\beta}_j)\Big],
\end{align}
where
$$\tilde{\beta}_j=\lambda_{m+1}^2\sum_{i=1}^{n}\big(|a_i-1|^2+|a_i-1|(|\eta_j|+|y|)\big)+|\eta_j|^3+|y|^3.$$
Therefore we have on $B_{r_0}(p_j)$
\begin{align}
\label{4.22}
\rho_1h_1e^{2v_1-v_2}=&(1+\varphi)\rho_1he^{v_{P,\Lambda,A}}+(e^{\varphi}-1-\varphi)\rho_1he^{v_{P,\Lambda,A}}
\nonumber\\
=&\rho_1h_{p_j}(p_j)|y|^{2\alpha_j}e^{U_j+t_j}\Big[1+(a_j-1)(U_j+s_j)+\eta_j+\nabla_yH_j\cdot y+Q_j(y)\nonumber\\
&+(a_j-1)(U_j+s_j)(\nabla_yH_j(0,0)\cdot y+\eta_j)+\frac12(\eta_j+\nabla_yH_j\cdot y)^2\nonumber\\
&+\sum_{i=1,i\neq j}^{n}8\pi(a_i-1)(1+\alpha_i)G(p_j,p_i)+\varphi\Big]+\tilde{E}_j,
\end{align}
where
\begin{equation}
\label{4.23}
\tilde{E}_j=(e^{\varphi}-1-\varphi)\rho_1he^{v_{P,\Lambda,A}}+\rho_1h_{p_j}(p_j)|y|^{2\alpha_j}e^{U_j+t_j}O(\varphi^2+\tilde{\beta}_j),
\end{equation}
and $\varphi=2\phi-\psi.$

Let $\epsilon_2>0$ be small. $\tilde{E}_j$ can be written into two parts
$$\tilde{E}_j=\tilde{E}^+_j+\tilde{E}^-_j,$$
where
\begin{align*}
\tilde{E}_j^+=\left\{\begin{array}{lll}
\tilde{E}_j~&\mathrm{if}~&|\varphi|\geq\epsilon_2\\
0~&\mathrm{if}~&|\varphi|<\epsilon_2,
\end{array}\right.\quad
\tilde{E}_j^-=\left\{\begin{array}{lll}
0&\mathrm{if}~&|\varphi|\geq\epsilon_2\\
\tilde{E}_j~&\mathrm{if}~&|\varphi|<\epsilon_2.
\end{array}\right.
\end{align*}
Then
\begin{equation}
\label{4.24}
\tilde{E}^+_j=O(e^{|\varphi|+2\lambda_j})~\mathrm{if}~|\varphi|\geq\epsilon_2,
\end{equation}
and
\begin{equation}
\label{4.25}
\tilde{E}^-_j=\rho_1he^{U_j+\lambda_j}|y|^{2\alpha_j}O(\varphi^2+\tilde{\beta}_j).
\end{equation}
Using the expression for $\rho_1h_1e^{2v_1-v_2}$ above, we obtain the following estimate for $\int_M\rho_1h_1e^{2v_1-v_2}.$
\vspace{0.25cm}

\begin{lemma}
\label{le4.3}
Let $\rho_*=\sum_{j=1}^{n}4\pi(1+\alpha_j)$, $v_1=\frac12v_{P,\Lambda,A}+\phi\in S_{\rho_1}(Q,w)~\mathrm{and}~v_2=\frac{1}{2}w+\psi\in S_{\rho_2}(Q,w).$
Then as
$\rho_1\rightarrow \rho_*,\rho_2\notin \Sigma_2,$ we have
\begin{align}
\label{4.26}
\int_{M}\rho_1h_1e^{2v_1-v_2}=&\sum_{j=1}^{n}4\pi(1+\alpha_j)e^{t_j}(1-\psi(p_j))
+\sum_{j\in J_2}\pi d_je^{t_j}e^{-\frac{\lambda_j}{1+\alpha_j}}\nonumber\\
&+\sum_{j=1}^{n}8\pi(1+\alpha_j)\lambda_j(a_j-1)e^{t_j}
+O(\sum_{j=1}^{n}|a_j-1|e^{\lambda_1})\nonumber\\
&+O(e^{\lambda_{m+1}-\frac{\lambda_{m+1}}{1+\alpha_{m+1}}-\varepsilon\lambda_{m+1}})
\end{align}
for some $\varepsilon>0.$
\end{lemma}
\begin{proof}
We note that $\lambda_{m+1}=\lambda_j+O(1)$ and $t_j=\lambda_j+O(1).$ By (\ref{4.22})
\begin{align*}
\int_{M}\rho_1h_1e^{2v_1-v_2}=&\sum_j\int_{B_{r_0}(p_j)}\{\rho_1h_{p_j}(p_j)|y|^{2\alpha_j}e^{U_j+t_j}[1+\cdots]+\tilde{E}_j\}\mathrm{d}y\nonumber\\
&+\int_{M\setminus\bigcup_jB_{r_0}(p_j)}\rho_1h_1e^{2v_1-v_2}.
\end{align*}
By the explicit expression of $U_j,$ we have
\begin{equation}
\label{4.27}
\int_{M\setminus\bigcup_jB_{r_0}(p_j)}\rho_1h_1e^{2v_1-v_2}=O(1),
\end{equation}
\begin{equation}
\label{4.28}
\int_{B_{r_0}(p_j)}\rho_1h_{p_j}(p_j)|y|^{2\alpha_j}e^{U_j+t_j}\mathrm{d}y=4\pi(1+\alpha_j)e^{t_j}+O(1),
\end{equation}
\begin{align}
\label{4.29}
\int_{B_{r_0}(p_j)}\rho_1h_{p_j}(p_j)|y|^{2\alpha_j}e^{U_j+t_j}(a_j-1)(U_j+s_j)\mathrm{d}y=&8\pi(1+\alpha_j)(a_j-1)\lambda_je^{t_j}\nonumber\\
&+O(|a_j-1|e^{\lambda_j}),
\end{align}
where $U_j+s_j=2\lambda_j-2\log\big(1+\frac{\rho_1h_{p_j}(p_j)}{4(1+\alpha_j)^2}e^{\lambda_j}|y|^{2(1+\alpha_j)}\big)+O(1)$ is used.
By equation (\ref{4.11}) of $\eta_j$ for $j\in J\setminus J_1$, and the fact $\zeta_{2,j}(y)y_1$ and $\zeta_{2,j}(y)\zeta_{1,j}(y)$ are odd functions, we have
\begin{align}
\label{4.30}
\int_{B_{r_0}(p_j)}&\rho_1h_j(p_j)|y|^{2\alpha_j}e^{U_j+t_j}\Big(\eta_j+\nabla_yH_j\cdot y+Q_j+\frac12(\eta_j+\nabla_yH_j\cdot y)^2\Big)dy\nonumber\\
&=-\frac12e^{t_j}\int_{B_{r_0}(p_j)}\Delta\eta_j\mathrm{d}y+O(e^{t_j}e^{-\frac{2\lambda_j}{1+\alpha_j}})\nonumber\\
&=-\frac12e^{t_j}\int_{\partial B_{r_0}(p_j)}\frac{\partial\eta_j}{\partial\nu}+O(e^{t_j}e^{-\frac{2\lambda_j}{1+\alpha_j}})\nonumber\\
&=\pi d_je^{t_j}e^{-\frac{\lambda_j}{1+\alpha_j}}+O(e^{t_j}e^{-\frac{2\lambda_j}{1+\alpha_j}}).
\end{align}
It is not difficult to see
\begin{equation}
\label{4.31}
\int_{B_{r_0}(p_j)}|y|^{2\alpha_j}e^{U_j+t_j}(a_i-1)G(p_j,p_i)\mathrm{d}y=O(\sum_{i=1}^{n}|a_i-1|e^{\lambda_j}).
\end{equation}
\begin{align}
\label{4.32}
\int_{B_{r_0}(p_j)}|y|^{2\alpha_j}e^{U_j+t_j}(a_j-1)&[(U_j+s_j)+1+\eta_j]\nabla_yH_j(0,0)\cdot y\mathrm{d}y\nonumber\\
&=O(|a_j-1|e^{\lambda_{m+1}-\frac{\lambda_{m+1}}{1+\alpha_{m+1}}}),
\end{align}
where the cancelation occurs due to the oddness of $\nabla_yH_j(0,0)\cdot y.$\\

To estimate the terms involving $\phi$ and $\psi$, we use (\ref{4.13}) to obtain
\begin{align*}
\Delta v_{p_j}=\Delta U_j+\Delta\eta_j+8\pi(1+\alpha_j)~\mathrm{for}~x\in B_{r_0}(p_j),
\end{align*}
and
\begin{align*}
\Delta v_{p_j}=\Delta(v_{p_j}-8\pi(1+\alpha_j)G(x,p_j))+8\pi(1+\alpha_j)~\mathrm{for}~x\notin B_{r_0}(p_j).
\end{align*}
This together with Lemma \ref{le4.1} implies
\begin{align}
\label{4.33}
\int_{B_{r_0}(p_j)}2\rho_1h_{p_j}(p_j)|y|^{2\alpha_j}e^{U_j}\phi=
&-\int_M\phi\Delta v_{p_j}+\int_{B_{r_0}(p_j)}\phi\Delta\eta_j+8\pi(1+\alpha_j)\int_M\phi\nonumber\\
&+\int_{M\setminus B_{r_0}(p_j)}\phi\Delta(v_{p_j}-8\pi(1+\alpha_j)G(x,p_j))\nonumber\\
=&\int_M\nabla\phi\nabla v_{p_j}+8\pi(1+\alpha_j)\int_M\phi+\int_{B_{r_0}(p_j)}\phi\Delta\eta_j\nonumber\\
&+\int_{M\setminus B_{r_0}(p_j)}\phi\Delta(v_{p_j}-8\pi(1+\alpha_j)G(x,p_j)).
\end{align}
To estimate $\int_{B_{r_0}(p_j)}\phi\Delta\eta_j,$ we choose $r_0'\in(\frac{r_0}{2},r_0)$ such that
\begin{align*}
\int_{\partial B_{r_0'}(p_j)}|\phi|\mathrm{d}\sigma\leq \frac{2}{r_0}\int_{B_{r_0}(p_j)}|\phi|\mathrm{d}x.
\end{align*}
Hence
\begin{align*}
\big|\int_{\partial B_{r_0'}(p_j)}\phi\frac{\partial\eta_j}{\partial\nu}\mathrm{d}\sigma\big|\leq C\max_{\partial B_{r_0'}(p_j)}
|\frac{\partial\eta_j}{\partial\nu}|\|\phi\|_{H^1}=O(e^{-\frac{\lambda_{m+1}}{1+\alpha_{m+1}}})\|\phi\|_{H^1},
\end{align*}
where by (\ref{4.11}),
\begin{align*}
-\int_{\partial B_{r_0'}(p_j)}\frac{\partial\eta_j}{\partial\nu}=
\int_{B_{r_0'}(p_j)}2\rho_1h_{p_j}(p_j)|y|^{2\alpha_j}e^{U_j}\tilde{H}(x,\eta_j)\mathrm{d}x
=O(e^{-\frac{\lambda_{m+1}}{1+\alpha_{m+1}}}).
\end{align*}
Thus
\begin{align*}
\int_{B_{r_0}(p_j)}\phi\Delta\eta_j=&\int_{B_{r_0'}(p_j)}\phi\Delta\eta_j
+O(e^{-\frac{\lambda_{m+1}}{1+\alpha_{m+1}}})\int_M|\phi|\\
=&-\int_{B_{r_0'}(p_j)}\nabla\phi\nabla\eta_j+O(e^{-\frac{\lambda_{m+1}}{1+\alpha_{m+1}}})\|\phi\|_{H^1}\\
\leq&~(\int_{B_{r_0'}(p_j)}|\nabla\phi|^2)^{\frac12}(\int_{B_{r_0'}(p_j)}|\nabla\eta_j|^2)^{\frac12}+O(e^{-\lambda_{m+1}})\int_M|\phi|\\
=&O(e^{-\frac{\lambda_{m+1}}{2(1+\alpha_{m+1})}})\|\phi\|_{H^1}.
\end{align*}
By (\ref{7.2}), we have for some $\epsilon>0$
\begin{align}
\label{4.34}
\Big|\int_{B_{r_0}(p_j)}\rho_1h_{p_j}(p_j)|y|^{2\alpha_j}e^{U_j}\phi\Big|=O(e^{-\epsilon\lambda_{m+1}})\|\phi\|_{H^1}.
\end{align}
While, for the terms involving $\psi,$ we have
\begin{align}
\label{4.35}
\int_{B_{r_0}(p_j)}\rho_1h_j(p_j)|y|^{2\alpha_j}e^{U_j}\psi=&\int_{B_{r_0}(p_j)}\rho_1h_{p_j}(p_j)|y|^{2\alpha_j}e^{U_j}\psi(p_j)\nonumber\\
&+\int_{B_{r_0}(p_j)}\rho_1h_{p_j}(p_j)|y|^{2\alpha_j}e^{U_j}(\psi-\psi(p_j))\nonumber\\
=&4\pi(1+\alpha_j)\psi(p_j)+O(e^{-\frac{3}{2(1+\alpha_{m+1})}\lambda_j}).
\end{align}

For $\tilde{E}_j^+,$ we have
\begin{align*}
\int_{B_{r_0}(p_j)}|\tilde{E}_j^+|=&O(1)\int_{B_{r_0}(p_j)\cap\{|\varphi-\overline{\varphi}|\geq\epsilon_2\}}e^{|\varphi|+2\lambda_j}\\
=&O(1)\int_{B_{r_0}(p_j)\cap\{|\varphi-\overline{\varphi}|\geq\epsilon_2\}}e^{|\varphi-\overline{\varphi}|+2\lambda_j},
\end{align*}
where $\overline{\varphi}=\frac{\int_{B_{r_0}(p_j)}\varphi}{\mathrm{vol}(B_{r_0}(p_j))}=O(\|\varphi\|_{H^1})=O(e^{-\frac{\lambda_{m+1}}{1+\alpha_{m+1}}})$
if $\lambda_{m+1}$ is large. Write
$$e^{|\varphi-\overline{\varphi}|}=e^{|\varphi-\overline{\varphi}|(1-\frac{4\pi|\varphi-\overline{\varphi}|}{\|\varphi-\overline{\varphi}\|^2})}
e^{\frac{4\pi|\varphi-\overline{\varphi}|^2}{\|\varphi-\overline{\varphi}\|^2}}.$$
Since $\|\varphi-\overline{\varphi}\|^{-2}=\|\varphi-\overline{\varphi}\|_{H^1}^{-2}\gg2\lambda_j,$ we have
\begin{align*}
e^{|\varphi-\overline{\varphi}|(1-\frac{4\pi|\varphi-\overline{\varphi}|}{\|\varphi-\overline{\varphi}\|^2})}\leq
e^{\frac{\epsilon_2}{2}(1-\frac{2\pi\epsilon_2}{\|\varphi-\overline{\varphi}\|^2})}\ll e^{-2\lambda_j}~\mathrm{for}~\|\varphi-\overline{\varphi}\|
\geq\frac{\epsilon_2}{2}.
\end{align*}
Hence, by Moser-Trudinger inequality
\begin{align}
\label{4.36}
\int_{B_{r_0}(p_j)\cap\{|\varphi-\overline{\varphi}|\geq\frac{\epsilon_2}{2}\}}e^{|\varphi-\overline{\varphi}|}\leq e^{-2\lambda_j}
\int_{B_{r_0}(p_j)}\exp\big(\frac{4\pi|\varphi-\overline{\varphi}|^2}{\|\varphi-\overline{\varphi}\|^2}\big)
\leq O(1)e^{-2\lambda_j},
\end{align}
which implies $\int_{B_{r_0}(p_j)}|\tilde{E}_j^+|\leq O(1).$

For $\tilde{E}_j^-,$ (\ref{4.25}) gives
\begin{equation}
\label{4.37}
\int_{B_{r_0}(p_j)}|\tilde{E}_j^-|\leq O(1)\int_{B_{r_0}(p_j)}(|\varphi|^2+\tilde{\beta}_j)\rho_1h_{p_j}(p_j)|y|^{2\alpha_j}e^{U_j+t_j}.
\end{equation}

\noindent By (\ref{7.3}) in section 8, we can estimate the first term on the right hand side of (\ref{4.37}) by
\begin{align*}
\int_{B_{r_0}(p_j)}\rho_1h_j(p_j)|y|^{2\alpha_j}e^{U_j+t_j}|\varphi|^2\leq& O(1)e^{t_j}
\Big{(}\int_{M}|\nabla\phi|^2+\|\psi\|_{*}^2\int_{B_{r_0}(p_j)}|y|^{2\alpha_j}e^{U_j}\Big{)}\\
=&O(1)e^{t_j}e^{-2\frac{\lambda_{m+1}}{1+\alpha_{m+1}}}=O(1)e^{\lambda_{m+1}-\frac{\lambda_{m+1}}{1+\alpha_{m+1}}-\epsilon_1\lambda_{m+1}}.
\end{align*}

Since $v_1\in S_{\rho_1}(Q,w)$, the term related to $\tilde{\beta}_j$ can be estimated as follows
\begin{align*}
&\int_{B_{r_0}(p_j)}\rho_1h_{p_j}(p_j)|y|^{2\alpha_j}e^{U_j+t_j}\tilde{\beta}_j\\
&\quad=O(1)e^{t_j}\Big{(}\sum_{i=1}^{n}\lambda_{m+1}^2(|a_i-1|)^2+\int_{B_{r_0}(p_j)}[|\eta_j|^3+|y|^3+|a_i-1|(|\eta_i|+|y|)]e^{U_j}\mathrm{d}y\Big{)}\\
&\quad=O(e^{t_j}e^{-\frac{\lambda_{m+1}}{1+\alpha_{m+1}}-\epsilon\lambda_{m+1}})
\end{align*}
for some $\epsilon>0$ and large $\lambda_{m+1}.$ Therefore, we have
\begin{equation}
\label{4.38}
\int_{B_{r_0}(p_j)}|\tilde{E}_j^-|=O(1)e^{\lambda_{m+1}-\frac{\lambda_{m+1}}{1+\alpha_{m+1}}-\epsilon\lambda_{m+1}}.
\end{equation}
By (\ref{4.22}) and (\ref{4.27})-(\ref{4.38}), we obtain (\ref{4.26}). Hence we finish the proof of Lemma \ref{le4.3}.
\end{proof}
\vspace{0.5cm}

Now we want to express $2v_1+T_1(v_1,v_2)$ in a formula similar to (\ref{4.22}). By Lemma \ref{le4.3},
we expect that $\frac{e^{t_j}}{\int_Mh_1e^{2v_1-v_2}}-1$ is small. Indeed, by definition of
$S_{\rho_1}(Q,w),$
$$|t_j-t_1|=O(1)e^{-\frac{\lambda_{m+1}}{1+\alpha_{m+1}}}.$$

By Lemma \ref{le4.3} and the Taylor expansion of the exponential function,
\begin{align}
\label{4.39}
e^{-t_j}\int_M\rho_1h_1e^{2v_1-v_2}
=&\rho_*+\sum_{i=1}^{n}4\pi(1+\alpha_i)(t_i-t_j-\psi(p_i))\nonumber\\
&+\sum_{i=1}^{n}\pi d_1e^{-\frac{\lambda_i}{1+\alpha_i}}+\sum_{i=1}^{n}8\pi(1+\alpha_i)\lambda_i(a_i-1)\nonumber\\
&+O(|a_i-1|)+O(e^{-\frac{\lambda_{m+1}}{1+\alpha_{m+1}}-\epsilon\lambda_{m+1}}).
\end{align}
Hence
\begin{align}
\label{4.40}
\frac{e^{t_j}}{\int_Mh_1e^{2v_1-v_2}}-1=&\frac{1}{e^{-t_j}\int_M\rho_1h_1e^{2v_1-v_2}}\big(\rho_1-\frac{\int_M\rho_1h_1e^{2v_1-v_2}}{e^{t_j}}\big)
\nonumber\\
=&\theta_j+O(|a_i-1|)+O(e^{-\frac{\lambda_{m+1}}{1+\alpha_{m+1}}-\epsilon\lambda_{m+1}}),
\end{align}
where $\theta_j$ is defined by
\begin{align}
\label{4.41}
\theta_j=&\frac{1}{\rho_*}\Big[(\rho_1-\rho_*)-\sum_{i=1}^{n}\pi d_ie^{-\frac{\lambda_i}{1+\alpha_i}}-\sum_{i=1}^{n}4\pi(1+\alpha_i)(t_i-t_j-\psi(p_i))
\nonumber\\
&-\sum_{i=1}^{n}8\pi(1+\alpha_i)\lambda_i(a_i-1)\Big].
\end{align}
Let
\begin{align}
\label{4.42}
\beta_j=\big|\frac{e^{t_j}}{\int_Mh_1e^{2v_1-v_2}}-1\big|^2+\tilde{\beta}_j,
\end{align}
and
\begin{align}
\label{4.43}
E_j=(e^{\varphi}-1-\varphi)\frac{2\rho_1h_1e^{2v_1-v_2}}{\int_Mh_1e^{2v_1-v_2}}+2\rho_1h_{p_j}(p_j)|y|^{2\alpha_j}e^{U_j}(O(\varphi^2)+O(\beta_j)).
\end{align}
Then in $B_{r_0}(p_j),$ we have by (\ref{4.22}),
\begin{align}
\label{4.44}
\frac{\rho_1h_1e^{2v_1-v_2}}{\int_Mh_1e^{2v_1-v_2}}=&
(1+\varphi)\frac{\rho_1he^{v_{P,\Lambda,A}}}{\int_Mh_1e^{2v_1-v_2}}
+(e^{\varphi}-1-\varphi)\frac{\rho_1he^{v_{P,\Lambda,A}}}{\int_Mh_1e^{2v_1-v_2}}\nonumber\\
=&\rho_1h_{p_j}(p_j)|y|^{2\alpha_j}e^{U_j}\Big[1+\big(\frac{e^{t_j}}{\int_Mh_1e^{2v_1-v_2}}-1\big)+(a_j-1)(U_j+s_j)\nonumber\\
&+\sum_{i=1,i\neq j}^n8\pi(1+\alpha_i)(a_i-1)G(p_j,p_i)+\eta_j+\nabla_yH_j\cdot y\nonumber\\
&+(a_j-1)(U_j+s_j)(\nabla_yH_j\cdot y+\eta_j)+\frac12(\eta_j+\nabla_yH_j\cdot y)^2\nonumber\\
&+Q_j(y)+\varphi\Big]+E_j.
\end{align}
Thus, we have in $B_{r_0}(p_j),$
\begin{align}
\label{4.45}
\Delta(2v_1+T_1(v_1,v_2))=&2\Delta v_1+\frac{2\rho_1h_1e^{2v_1-v_2}}{\int_Mh_1e^{2v_1-v_2}}-2\rho_1\nonumber\\
=&a_j(\Delta U_j+\Delta\eta_j)+2\Delta\phi+\sum_{j=1}^{n}8\pi(1+\alpha_i)a_i
+\frac{2\rho_1h_1e^{2v_1-v_2}}{\int_Mh_1e^{2v_1-v_2}}-2\rho_1\nonumber\\
=&-2a_j\rho_1h_{p_j}(p_j)|y|^{2\alpha_j}e^{U_j}\Big[1+\eta_j+\nabla_yH_j\cdot y+Q_j(y)\nonumber\\
&-\epsilon_0^2\zeta_{2,j}(\epsilon_0^{-1}y)(\epsilon_0\zeta_{1,j}(\epsilon_0^{-1}y)+\nabla_yH_j\cdot y)-\frac{\epsilon_0^4}{2}\zeta_{2,j}^2(\epsilon_0^{-1}y)
\nonumber\\
&+\frac12(\eta_j+\nabla_yH_j\cdot y)^2\Big]+2\Delta\phi+[\sum_{j=1}^{n}8\pi(1+\alpha_j)-2\rho_1]\nonumber\\
&+\sum_{j=1}^{n}8\pi(1+\alpha_j)(a_j-1)+\frac{2\rho_1h_1e^{2v_1-v_2}}{\int_Mh_1e^{2v_1-v_2}}\nonumber\\
=&2\Delta\phi+(2\rho_*-2\rho_1)+\sum_{j=1}^{n}8\pi(1+\alpha_j)(a_j-1)\nonumber\\
&+2\rho_1h_{p_j}(p_j)|y|^{2\alpha_j}e^{U_j}\Big[(U_j+s_j)(a_j-1)(\nabla_yH_j\cdot y+\eta_j)\nonumber\\
&+(a_j-1)(U_j+s_j-1)+\sum_{i\neq j}8\pi(1+\alpha_i)(a_i-1)G(p_j,p_i)\nonumber\\
&+(\frac{e^{t_j}}{\int_Mh_1e^{2v_1-v_2}}-1)+\varphi\Big]+\hat{E}_j,
\end{align}
where $\epsilon_0=e^{-\frac{\lambda_{m+1}}{2(1+\alpha_{m+1})}}$ and
$$\hat{E}_j=E_j+2a_j\rho_1h_{p_j}(p_j)|y|^{2\alpha_j}e^{U_j}
\Big[\epsilon_0^2\zeta_{2,j}(\epsilon_0^{-1}y)(\eta_j+\nabla_yH_j\cdot y)+\frac{\epsilon_0^4}{2}
\zeta_{2,j}^2(\epsilon_0^{-1}y)\Big].$$

\noindent On $B_{2r_0}(p_j)\setminus B_{r_0}(p_j),$ since $v_{p_j}-\overline{v}_{p_j}-8\pi(1+\alpha_j)G(x,p_j)$ is small, we write
$\Delta(2v_1+T_1(v_1,v_2))$ as
\begin{align}
\label{4.46}
&\Delta(2v_1+T_1(v_1,v_2))\nonumber\\
=&2\Delta\phi+a_j\Delta(v_{p_j}-8\pi(1+\alpha_j)G(x,p_j))\nonumber\\
&+2\rho_*-2\rho_1+\sum_{j=1}^{n}8\pi(1+\alpha_j)(a_j-1)\nonumber\\
&+\frac{2\rho_1h}{\int_Mh_1e^{2v_1-v_2}}
e^{a_j(v_{p_j}-\overline{v}_{p_j}-8\pi(1+\alpha_j)G(x,p_j))+\sum_{i=1}^{n}8\pi(1+\alpha_i)a_i G(x,p_i)+\varphi}.
\end{align}

\noindent On $M\setminus\bigcup_jB_{2r_0}(p_j),$ we have
\begin{align}
\label{4.47}
\Delta(2v_1+T_1(v_1,v_2))=&2\Delta\phi+2\rho_*-2\rho_1+\sum_{j=1}^{n}8\pi(1+\alpha_j)(a_j-1)\nonumber\\
&+\frac{2\rho_1h}{\int_{M}h_1e^{2v_1-v_2}}e^{\sum_{i=1}^{n}8\pi(1+\alpha_i)a_iG(x,p_i)+\varphi}.
\end{align}

From (\ref{4.45})-(\ref{4.47}), we have the following
\begin{lemma}
\label{le4.4}
Let $v_1=\frac12v_{P,\Lambda,A}+\phi\in S_{\rho_1}(Q,w),$ $v_2=\frac12w+\phi$. Then as $\rho_1\rightarrow \rho_*=\sum_{j=1}^{n}4\pi(1+\alpha_j),$
\begin{enumerate}
      \item
      \begin{align}
      \label{4.48}
      \l\nabla(2v_1+T_1(v_1,v_2)),\nabla\phi_1\r=2\mathfrak{B}(\phi,\phi_1)+O(e^{-\frac{\lambda_{m+1}}{1+\alpha_{m+1}}})\|\phi_1\|_{H^1_0(M)},
      \end{align}
      where
      $$\mathfrak{B}(\phi,\phi_1):=\int_{M}\nabla\phi\cdot\nabla\phi_1-\sum_j\int_{B_{r_0}(p_j)}2\rho_1h_{p_j}(p_j)|y|^{2\alpha_j}e^{U_j}\phi\phi_1,$$
      \item for $1\leq j\leq m,$
      \begin{align}
      \label{4.49}
      \l\nabla(&2v_1+T_1(v_1,v_2)),\nabla\partial_{p_j}v_{p_j}\r\nonumber\\
      =&-8\pi\nabla_yH_j(0,0)+8\pi\nabla\psi(p_j)+O\Big(|\frac{e^{t_j}}{\int_{M} h_1e^{2v_1-v_2}}-1-\psi(p_j)|\nonumber\\
      &+|a_j-1|\lambda_j+e^{-\frac{\lambda_{m+1}}{1+\alpha_{m+1}}}\Big),
      \end{align}
      \item for $1\leq j\leq n$,
      \begin{align}
      \label{4.50}
      \l\nabla(2v_1+&T_1(v_1,v_2)),\nabla\partial_{\lambda_j}v_{p_j}\r\nonumber\\
      =&-16\pi(1+\alpha_j)(a_j-1)\lambda_j-8\pi(1+\alpha_j)(\theta_j-\psi(p_j))\nonumber\\
        &+O(\max_i|a_i-1|+e^{-\frac{\lambda_{m+1}}{1+\alpha_{m+1}}-\epsilon\lambda_{m+1}}),
      \end{align}
      \item for $1\leq j\leq n$,
      \begin{align}
      \label{4.51}
      \l\nabla(2v_1+&T_1(v_1,v_2)),\nabla v_{p_j}\r\nonumber\\
      =&\Big(2\lambda_j-1+8\pi(1+\alpha_j)R(p_j,p_j)+2\log\frac{\rho_1h_{p_j}(p_j)}{4(1+\alpha_j)^2}\Big)\nonumber\\
        &\times\l\nabla(2v_1+T_1(v_1,v_2)),\nabla\partial_{\lambda_j}v_{p_j}\r+16\pi(1+\alpha_j)(a_j-1)\lambda_j\nonumber\\
        &+8\pi(1+\alpha_j)\sum_{i\neq j}G(p_i,p_j)\l\nabla(2v_1+T_1(v_1,v_2)),\nabla\partial_{\lambda_i}v_{p_i}\r\nonumber\\
      &+O(1)\|\phi\|_{H^1(M)}+O(e^{-\frac{\lambda_{m+1}}{1+\alpha_{m+1}}}),
      \end{align}
\end{enumerate}
\end{lemma}

We will prove Lemma \ref{le4.4} in the section 8 because
the proof contains a lot of computations.

In order to know the Morse index for the solutions in $S_{\rho_1}(Q,w),$ we have to compute the Morse index of the bilinear form
$\mathfrak{B}(\phi,\phi_1)$ in Lemma \ref{le4.4}. For such bilinear form $\mathfrak{B}(\phi,\phi_1)$, we have the following lemma, due to Chen-Lin
\cite[Lemma 3.3]{cl4}.
\begin{lemma}
\label{le4.5}
Assume that all $\alpha_j\geq0$, are not inters for $j\in J\setminus J_1$. Let $P=(p_1,p_2,\cdots,p_{n})$ and
$\Lambda=(\lambda_1,\lambda_2,\cdots,\lambda_{n})$, where
$\mathrm{dist}(p_i,p_j)>2r_0$ for $i\neq j$. If $\lambda(P)$ is large, then the symmetric bilinear form
$$\mathfrak{B}(\phi,\phi_1):=\int_{M}\nabla\phi\cdot\nabla\phi_1-\sum_{j=1}^n\int_{B_{r_0}(p_j)}2\rho_1h_{p_j}(p_j)|y|^{2\alpha_j}e^{U_j}\phi\phi_1$$
is non-degenerate and has Morse index $\sum_{j=m+1}^{n}2(1+[\alpha_j])$ in $O_{P,\Lambda}^{(1)}$, where $[\alpha_j]$ denotes the greatest integer less than
or equal to $\alpha_j.$
\end{lemma}
\vspace{1cm}

\section{Deformation and Degree counting formula}
In this section, we want to deform $2v_i+T_i(v_1,v_2)$ into a simple form which can be solvable. Obviously, $v_1=\frac12v_{P,\Lambda,A}+\phi,
~v_2=\frac12w+\psi$ is a solution of $2v_1+T_1(v_1,v_2)=0,$ iff the left hand sides of
(\ref{4.48})-(\ref{4.51}) vanish. To solve the system (\ref{4.48})-(\ref{4.51}) and $2v_2+T_2(v_1,v_2)=0$, we recall
$$\mathring{H}^1=O_{P,\Lambda}^{(1)}\bigoplus~\mathrm{the~linear~subspace~spanned~by}~v_{p_j},
\partial_{\lambda_j}v_{p_j}~\mathrm{and}~\partial_{p_j}v_{p_j}$$
and deform $2v_i+T_i(v_1,v_2)$ to a simpler operator $2v_i+T_i^0(v_1,v_2)$ by defining the operator
$2I+T_i^t,~0\leq t\leq1,~i=1,2$ through the following relations.

\begin{align}
\label{5.1}
\l\nabla(2v_1+T_1^t(v_1,v_2)),\nabla\phi_1\r=&t\l\nabla(2v_1+T_1(v_1,v_2)),\nabla\phi_1\r\nonumber\\
&+2(1-t)\mathfrak{B}(\phi,\phi_1)~\mathrm{for}~\phi_1\in O_{P,\Lambda}^{(1)};
\end{align}
\begin{align}
\label{5.2}
&\l\nabla(2v_1+T_1^t(v_1,v_2)),\nabla\partial_{p_j}v_{p_j}\r=t\l\nabla(2v_1+T_1(v_1,v_2)),\nabla\partial_{p_j}v_{p_j}\r\nonumber\\
&\quad\quad\quad\quad+(1-t)\big(-8\pi\nabla_yH_j(0,0)+8\pi\nabla\psi(p_j)\big)~\mathrm{for}~1\leq j\leq m;
\end{align}
\begin{align}
\label{5.3}
&\l\nabla(2v_1+T_1^t(v_1,v_2)),\nabla\partial_{\lambda_j}v_{p_j}\r=t\l\nabla(2v_1+T_1(v_1,v_2)),\nabla\partial_{\lambda_j}v_{p_j}\r\nonumber\\
&\quad\quad\quad\quad-(1-t)8\pi(1+\alpha_j)\Big[2(a_j-1)\lambda_j+(\theta_j-\psi(p_j))\Big],~\mathrm{for}~1\leq j\leq n;
\end{align}
\begin{align}
\label{5.4}
&\l\nabla(2v_1+T_1^t(v_1,v_2)),\nabla v_{p_j}\r
=t\Big{[}\big(2\lambda_j+O(1)\big)\l\nabla(2v_1+T_1^t(v_1,v_2)),\nabla\partial_{\lambda_j}v_{p_j}\r\nonumber\\
&\quad\quad\quad\quad+\sum_{i\neq j}O(1)\l\nabla(2v_1+T_1^t(v_1,v_2)),\nabla\partial_{\lambda_i}v_{p_i}\r+O(1)\|\phi\|_{H^1}\nonumber\\
&\quad\quad\quad\quad+O(e^{-\frac{\lambda_{m+1}}{1+\alpha_{m+1}}})\Big]+16\pi(1+\alpha_j)(a_j-1)\lambda_j,~\mathrm{for}~1\leq j\leq n;
\end{align}
\begin{align}
\label{5.5}
&2v_2+T_2^t(v_1,v_2)=t(2v_2+T_2(v_1,v_2))\nonumber\\
&\quad+(1-t)\Big(w+2\psi-2\rho_2(-\Delta)^{-1}\big(\frac{h_2e^{w+2\psi-\sum_{j=1}^{n}4\pi(1+\alpha_j)G(x,p_j)}}
      {\int_Mh_2e^{w+2\psi-\sum_{j=1}^{n}4\pi(1+\alpha_j)G(x,p_j)}}-1\big)\Big),
\end{align}
where those coefficients $O(1)$ are those terms appeared in (\ref{5.4}) so that $T_1^1(v_1,v_2)=T_1(v_1,v_2)$. From the construction above, we have
$$2v_i+T_i(v_1,v_2)=2v_i+T_i^1(v_1,v_2),~i=1,2.$$
When $t=0$, the operator $T_i^0$ is simpler than $T_i,~i=1,2.$ During the deformation from $T_i^1$ to $T_i^0,~i=1,2$ we have

\begin{lemma}
\label{le5.1}
Let $\rho_*=\sum_{j=1}^{n}4(1+\alpha_j)\pi.$ Assume $(\rho_1-\rho_*)\neq0$, and $\rho_2\notin\Sigma_2$. Then there is $\varepsilon_1>0$
such that $(2v_1+T_1^t(v_1,v_2),2v_2+T_2^t(v_1,v_2))\neq0$ for $(v_1,v_2)\in \partial\big(S_{\rho_1}(Q,w)\times S_{\rho_2}(Q,w)\big)$ and $0\leq t\leq1$
if $|\rho_1-\rho_*|<\varepsilon_1$ and $\rho_2$ is fixed.
\end{lemma}
\begin{proof}
Assume $(v_1,v_2)\in\overline{S}_{\rho_1}(Q,w)\times\overline{S}_{\rho_2}(Q,w)$, where $\overline{S}_{\rho_i}(Q,w)$ denotes the closure of
$S_{\rho_i}(Q,w),$ and $2v_i+T_i^t(v_1,v_2)=0,~i=1,2$ for some $0\leq t\leq1.$ We will show that
$(v_1,v_2)\notin\partial\big(S_{\rho_1}(Q,w)\times S_{\rho_2}(Q,w)\big).$

From $\l\nabla(2v_1+T_1^t(v_1,v_2)),\nabla\phi\r=0,$ we have by Lemma \ref{le4.4}
$$\|\phi\|_{H^1}^2\leq O(e^{-\frac{\lambda_{m+1}}{1+\alpha_{m+1}}})\|\phi\|_{H^1}.$$
This implies
\begin{equation}
\label{5.6}
\|\phi\|_{H^1}=O(e^{-\frac{\lambda_{m+1}}{1+\alpha_{m+1}}})\leq c_1e^{-\frac{\lambda_{m+1}}{1+\alpha_{m+1}}},
\end{equation}
for some constant $c_1$ \footnote{Here $c_1$ is independent of $\psi$, it can be shown in the proof of Lemma \ref{le4.4}.} independent
of $c_0$.

Using $\l\nabla(2v_1+T_1^t(v_1,v_2)),\nabla\partial_{\lambda_j}v_{p_j}\r=0$ and $\l\nabla(2v_1+T_1^t(v_1,v_2)),\nabla v_{p_j}\r=0,$ (\ref{5.4}) and
(\ref{5.6}) implies
\begin{equation}
\label{5.7}
16\pi\lambda_j(1+\alpha_j)(a_j-1)=O(e^{-\frac{\lambda_{m+1}}{1+\alpha_{m+1}}})~\mathrm{for}~j=1,\cdots,n,
\end{equation}
that is, when $\rho_1$ is close to $\rho_*,$
\begin{equation}
\label{5.8}
|a_j-1|=O(\lambda_{m+1}^{-1}e^{-\frac{\lambda_{m+1}}{1+\alpha_{m+1}}})<c_0\lambda_{m+1}^{-\frac12}(P)e^{-\frac{\lambda(P)}{1+\alpha_{m+1}}}
~\mathrm{for}~1\leq j\leq n.
\end{equation}

By $\l\nabla(2v_1+T_1^t(v_1,v_2)),\nabla_{\lambda_j}v_{p_j}\r=0,$ we conclude
from (\ref{4.50}) and (\ref{5.8}) that
\begin{equation}
\label{5.9}
\theta_j-\psi(p_j)+2\lambda_j(a_j-1)=O(\max|a_i-1|+e^{-\frac{\lambda_{m+1}}{1+\alpha_{m+1}}-\epsilon\lambda_{m+1}})=
O(\lambda_{m+1}^{-1}e^{-\frac{\lambda_{m+1}}{1+\alpha_{m+1}}}).
\end{equation}
\begin{equation}
\label{5.10}
e^{t_j}\big(\int_Mh_1e^{2v_1-v_2}\big)^{-1}-1-\theta_j=O(\max|a_i-1|+e^{-\frac{\lambda_{m+1}}{1+\alpha_{m+1}}-\epsilon\lambda_{m+1}})=
O(e^{-\frac{\lambda_{m+1}}{1+\alpha_{m+1}}}).
\end{equation}

Together with $\l\nabla(2v_1+T_1^t(v_1,v_2)),\nabla\partial_{p_j}v_{p_j}\r=0$ for $j\leq m$ and (\ref{5.8}), (\ref{5.10}) and part (2) of Lemma
\ref{le4.4}, we have
\begin{align*}
|\nabla_yH_j(0,0)-\nabla\psi(p_j)|=&~O\Big(\lambda_j|a_j-1|+|\frac{e^{t_j}}{\int_{M}h_1e^{2v_1-v_2}}-1-\psi(p_j)|
+e^{-\frac{\lambda_j}{1+\alpha_{m+1}}}\Big)
\nonumber\\
\leq&~ O(1)e^{-\frac{\lambda_{m+1}}{1+\alpha_{m+1}}},
\end{align*}
which implies
\begin{align}
\label{5.11}
&\Big|\nabla_y^2H_j\cdot(p_j-p_j^0)+8\pi\sum^m_{i=1,i\neq j}\nabla_x^2G(x,p_j^0)\mid_{x=p_i^0}\cdot(p_i-p_i^0)-\nabla\psi(p_j^0)\Big|\nonumber\\
&\leq O(1)e^{-\frac{\lambda_{m+1}}{1+\alpha_{m+1}}}+O(1)\|\psi\|_*|p_j-p_j^0|^{\gamma}+O(1)\sum_{i=1}^m|p_i-p_i^0|^2,
\end{align}
where we used $\nabla_yH_j(p_j^0-p_j,0)=0$ and $\mathfrak{p}>2.$

For the second component, by (\ref{5.5}), we have
\begin{align}
\label{5.12}
0=&(1-t)\Big(\Delta w+2\Delta\psi+2\rho_2\Big(\frac{h_2e^{w+2\psi-\sum_{j=1}^{n}4\pi(1+\alpha_j)G(x,p_j)}}
{\int_Mh_2e^{w+2\psi-\sum_{j=1}^{n}4\pi(1+\alpha_j)G(x,p_j)}}-1\Big)\Big)\nonumber\\
&+t\Big(\Delta w+2\Delta\psi+2\rho_2\Big(\frac{h_2e^{w+2\psi-\frac12v_{P,\Lambda,A}-\phi}}
{\int_Mh_2e^{w+2\psi-\frac12v_{P,\Lambda,A}-\phi}}-1\Big)\Big).
\end{align}
We set
\begin{equation*}
\Theta=2\rho_2\frac{h_2e^{w+2\psi-\frac12v_{P,\Lambda,A}-\phi}}
{\int_Mh_2e^{w+2\psi-\frac12v_{P,\Lambda,A}-\phi}}-2\rho_2\frac{h_2e^{w+2\psi-\sum_{j=1}^{n}4\pi(1+\alpha_j)G(x,p_j)}}
{\int_Mh_2e^{w+2\psi-\sum_{j=1}^{n}4\pi(1+\alpha_j)G(x,p_j)}},
\end{equation*}
and claim
\begin{equation}
\label{5.13}
\|\Theta\|_{L^{\mathfrak{p}}(M)}\leq c_2e^{-\frac{\lambda(P)}{1+\alpha_{m+1}}},
\end{equation}
where $c_2$ is a constant that independent of $c_0$ and $\mathfrak{p}$ is defined in $O_{P,\Lambda}^{(2)}.$
By (\ref{5.6}), it is not difficult to get
\begin{equation*}
\exp\big(w+2\psi-\frac12v_{P,\Lambda,A}-\phi\big)=\exp(w+2\psi-\frac12v_{P,\Lambda,A}\big)+\Theta_1,
\end{equation*}
where $\|\Theta_1\|_{\mathfrak{p}}\leq c_3e^{-\frac{\lambda(P)}{1+\alpha_{m+1}}}.$ By noting (\ref{5.8}),
it is enough for us to prove the following one.
\begin{equation}
\label{5.14}
\Big\|\exp(w+2\psi-\frac12 v_{P,\Lambda,A})-\exp\big(w+2\psi-\sum_{j=1}^n4\pi(1+\alpha_j)a_jG(x,p_j)\big)\Big\|_{L^{\infty}(M)}\leq
c_4e^{-\frac{\lambda(P)}{1+\alpha_{m+1}}}.
\end{equation}
Since the proof is long, we leave it in section 8. By (\ref{5.13}), (\ref{5.12}) can be written as
\begin{align}
\label{5.15}
\Delta w+2\Delta\psi+2\rho_2\Big(\frac{h_2e^{w+2\psi-\sum_{j=1}^{n}4\pi(1+\alpha_j)G(x,p_j)}}
{\int_Mh_2e^{w+2\psi-\sum_{j=1}^{n}4\pi(1+\alpha_j)G(x,p_j)}}-1\Big)+t\Theta=0.
\end{align}
We expand the above equation,
\begin{align}
\label{5.16}
\mathfrak{R}=&\Delta\psi+2\rho_2\frac{\overline{h}_2e^{w-4\pi\sum_{j=1}^{m}G(x,p_j^0)}}
{\int_M\overline{h}_2e^{w-4\pi\sum_{j=1}^{m}G(x,p_j^0)}}\psi\nonumber\\
&-2\rho_2\frac{\overline{h}_2e^{w-4\pi\sum_{j=1}^{m}G(x,p_j^0)}}
{\Big(\int_M\overline{h}_2e^{w-4\pi\sum_{j=1}^{m}G(x,p_j^0)}\Big)^2}\int_M\big(\overline{h}_2e^{w-4\pi\sum_{j=1}^{m}G(x,p_j^0)}\psi\big)\nonumber\\
&-4\pi\rho_2\frac{\overline{h}_2e^{w-4\pi\sum_{j=1}^{m}G(x,p_j^0)}}
{\int_M\overline{h}_2e^{w-4\pi\sum_{j=1}^{m}G(x,p_j^0)}}\big(\sum_{j=1}^m\nabla G(x,p_j^0)(p_j-p_j^0)\big)\nonumber\\
&+4\pi\rho_2\frac{\overline{h}_2e^{w-4\pi\sum_{j=1}^{m}G(x,p_j^0)}}
{\Big(\int_M\overline{h}_2e^{w-4\pi\sum_{j=1}^{m}G(x,p_j^0)}\Big)^2}
\int_M\Big(\overline{h}_2e^{w-4\pi\sum_{j=1}^{m}G(x,p_j^0)}\big(\sum_{j=1}^m\nabla G(x,p_j^0)(p_j-p_j^0)\big)\Big),
\end{align}
where $\mathfrak{R}=t\Theta+o(1)\|\psi\|_{*}+|p_j-p_j^0|^2.$ By the non-degeneracy of $(P_w,w)$ to (\ref{1.12}), (\ref{5.11}) and (\ref{5.16}), we can get
\begin{align}
\label{5.17}
\|\psi\|_{*}\leq c_5e^{-\frac{\lambda_{m+1}}{1+\alpha_{m+1}}}
~\mathrm{and}~
|p_j-p_j^0|\leq c_6e^{-\frac{\lambda_{m+1}}{1+\alpha_{m+1}}}.
\end{align}

Recall that
\begin{align*}
\theta_j=&\frac{1}{\rho_*}\Big((\rho_1-\rho_*)-\sum_{j=1}^{n}\pi d_je^{-\frac{\lambda_j}{1+\alpha_j}}-\sum_{i=1}^{n}4\pi(1+\alpha_i)(t_i-t_j-\psi(p_i))\\
&-\sum_{i=1}^{n}8\pi(1+\alpha_i)\lambda_i(a_i-1)\Big).
\end{align*}
From (\ref{5.9}), we obtain
\begin{align}
\label{5.18}
O(1)\lambda_{m+1}^{-1}e^{-\frac{\lambda_{m+1}}{1+\alpha_{m+1}}}=&\sum_j8\pi(1+\alpha_j)[\theta_j-\psi(p_j)+2\lambda_j(a_j-1)]\nonumber\\
=&2\Big(\rho_1-\rho_*-\sum_{j=1}^{n}\pi d_je^{-\frac{\lambda_j}{1+\alpha_j}}\Big),
\end{align}
where all $t_j$ cancel out with each other. Here $\rho_*=4\pi\sum_j(1+\alpha_j)$ is used. By the definition of $d_i$ in section 4, and the assumption
$$|t_j-t_1|\leq c_0e^{-\frac{\lambda(P)}{1+\alpha_{m+1}}},$$
we have
\begin{align}
\label{5.19}
\sum_{j=1}^{n}\pi d_je^{-\frac{\lambda_j}{1+\alpha_j}}=&
\frac{\pi^2}{(1+\alpha_{m+1})\sin\frac{\pi}{1+\alpha_{m+1}}}\big(\frac{4(1+\alpha_{m+1})^2}
{\rho_*h_{p_{m+1}}^*(p_{m+1})}\big)^{\frac{2}{1+\alpha_{m+1}}}e^{-\frac{G_{m+1}^*(p_{m+1})}{1+\alpha_{m+1}}}\nonumber\\
&\times l(Q)e^{-\frac{\lambda_{m+1}}{1+\alpha_{m+1}}}+O(e^{-\frac{\lambda_{m+1}}{1+\alpha_{m+1}}-\epsilon\lambda_{m+1}}).
\end{align}
Also, we have
$$\rho_1-\rho_*=\frac{\pi^2}{(1+\alpha_{m+1})\sin\frac{\pi}{1+\alpha_{m+1}}}
\big(\frac{4(1+\alpha_{m+1})^2}{\rho_*h_{p_{m+1}}^*(p_{m+1}^0)}\big)^{\frac{2}{1+\alpha_{m+1}}}
e^{-\frac{G_{m+1}^*(p_{m+1}^0)}{1+\alpha_{m+1}}}l(Q)e^{-\frac{\lambda(P)}{1+\alpha_{m+1}}}.$$
Therefore, (\ref{5.18}) and (\ref{5.19}) imply
$$O(1)\lambda_{m+1}^{-1}e^{-\frac{\lambda_{m+1}}{1+\alpha_{m+1}}}=c(e^{-\frac{\lambda(P)}{1+\alpha_{m+1}}}-e^{-\frac{\lambda_{m+1}}{1+\alpha_{m+1}}})$$
for some $c\neq0$. This in turn gives
\begin{align}
\label{5.20}
|\lambda_{m+1}-\lambda(P)|\leq c_7\lambda(P)^{-1}.
\end{align}
for some $c_7$ independent of $c_0.$

Using (\ref{5.17}) and (\ref{5.20}), we have
\begin{align}
\label{5.21}
\|\psi\|_{*}\leq c_8e^{-\frac{\lambda(P)}{1+\alpha_{m+1}}}
~\mathrm{and}~
|p_j-p_j^0|\leq c_9e^{-\frac{\lambda(P)}{1+\alpha_{m+1}}}.
\end{align}
To obtain estimates for $t_j-t_1,j\geq2,$ we note $\theta_j=O(e^{-\frac{\lambda_{m+1}}{1+\alpha_{m+1}}})$ by (\ref{5.9}). Combined with
(\ref{4.41}), we have
\begin{align*}
\Big|t_j-\frac{1}{\rho_*}\sum_j4\pi(1+\alpha_j)t_j\Big|=O(e^{-\frac{\lambda_{m+1}}{1+\alpha_{m+1}}})
\end{align*}
and
\begin{align}
\label{5.22}
|t_j-t_1|\leq&\Big|t_j-\frac{1}{\rho_*}\sum_j4\pi(1+\alpha_j)t_j\Big|+\Big|\frac{1}{\rho_*}\sum_j4\pi(1+\alpha_j)t_j-t_1\Big|\nonumber\\
=&O(e^{-\frac{\lambda_{m+1}}{1+\alpha_{m+1}}})\leq c_{10}e^{-\frac{\lambda_{m+1}}{1+\alpha_{m+1}}}
\end{align}
for $j\geq2$, where $c_{10}$ depends on $c_8,~c_9$ and is independent of $c_0.$ By choosing $$c_0>c_1,c_7,c_8,c_9,c_{10},$$ we can get
$v_2\notin\partial S_{\rho_2}(Q,w).$
From (\ref{5.6}), (\ref{5.8}), (\ref{5.20}), (\ref{5.21}), and
(\ref{5.22}), we obtain $v_1\not\in\partial S_{\rho_1}(Q,w)$. Therefore,
$$(v_1,v_2)\notin\partial\Big(S_{\rho_1}(Q,w),S_{\rho_2}(Q,w)\Big).$$
The proof is completed.
\end{proof}

Then, we want to apply Lemma \ref{le4.5} and Lemma \ref{le5.1} to get the degree of the linear operator in $S_{\rho_1}(Q,w)
\times S_{\rho_2}(Q,w)$ when $\rho_1$ crosses $\rho_*$.

To compute the term
\begin{align*}
\mathrm{deg}\Big(\big(2v_1+T_1(v_1,v_2),2v_2+T_2(v_1,v_2)\big);S_{\rho_1}(Q,w)\times S_{\rho_2}(Q,w),0\Big).
\end{align*}
We set
$$S_1^*(Q,w)=\Big\{(P,\Lambda,A):\frac12v_{P,\Lambda,A}+\phi\in S_{\rho_1}(Q,w), \phi\in O_{P,\Lambda}^1\Big\}$$
and define the map
$$\Phi_{Q}=(\Phi_{Q,1},\Phi_{Q,2},\Phi_{Q,3},\Phi_{Q,4}):$$
\begin{align*}
&\Phi_{Q,1}^{j}=\l\nabla(2v_1+T^0_1(v_1,v_2)),\nabla\partial_{p_j}v_{p_j}\r+\l\nabla2v_2+T^0_2(v_1,v_2),0\r,~\mathrm{for}~1\leq j\leq m,\\
&\Phi_{Q,2}^{j}=\l\nabla(2v_1+T^0_1(v_1,v_2)),\nabla\partial_{\lambda_j}v_{p_j}\r+\l\nabla(2v_2+T^0_2(v_1,v_2)),0\r,~\mathrm{for}~1\leq j\leq n,\\
&\Phi_{Q,3}^{j}=\l\nabla(2v_1+T^0_1(v_1,v_2)),\nabla v_{p_j}\r+\l\nabla(2v_2+T^0_2(v_1,v_2)),0\r,~\mathrm{for}~1\leq j\leq n,\\
&\Phi_{Q,4}=\l\nabla(2v_1+T^0_1(v_1,v_2)),0\r+(2v_2+T^0_2(v_1,v_2)).
\end{align*}
Clearly, by Lemma \ref{le4.5} and Lemma \ref{le5.1}, we have
\begin{align}
\label{5.23}
\mathrm{deg}\Big(\big(2v_1+&T_1(v_1,v_2),2v_2+T_2(v_1,v_2)\big);S_{\rho_1}(Q,w)\times S_{\rho_2}(Q,w),0\Big)\nonumber\\
&=\mathrm{deg}\Big(\Phi_{Q};S_1^*(Q,w)\times S_{\rho_2}(Q,w),0\Big).
\end{align}

Next, we study the right hand side of (\ref{5.23}) and prove Theorem \ref{th1.4}.\\

\noindent {\em Proof of Theorem \ref{th1.4}.}
To compute the degree, we can simplify the problem by replacing $\Phi_{Q}$ by a new map $\hat{\Phi}_{Q}$ defined as follows:
$\hat{\Phi}_{Q,1}=\Phi_{Q,1}$, $\hat{\Phi}_{Q,3}=\Phi_{Q,3},$ $\hat{\Phi}_{Q,4}=\Phi_{Q,4}$,
\begin{align}
\label{5.24}
\hat{\Phi}_{Q,2}^j=&\Phi_{Q,2}^j-\frac{8\pi(1+\alpha_j)}{2\rho_*}\sum_{i=1}^n\Phi_{Q,3}^i+\Phi_{Q,3}^j\nonumber\\
&=-\frac{8\pi(1+\alpha_j)}{\rho_*}\Big[\rho-\rho_*-4\pi\sum_i[(1+\alpha_i)(t_i-t_j)]-\pi\sum_id_ie^{-\frac{\lambda_i}{1+\alpha_i}}\Big].
\end{align}
Clearly, we have
\begin{align}
\label{5.25}
\frac{\partial\hat{\Phi}_{Q,1}}{\partial\Lambda}=\frac{\partial\hat{\Phi}_{Q,1}}{\partial A}
=\frac{\partial\hat{\Phi}_{Q,2}}{\partial A}=\frac{\partial\hat{\Phi}_{Q,2}}{\partial \psi}=\frac{\partial\hat{\Phi}_{Q,3}}{\partial \psi}
=\frac{\partial\hat{\Phi}_{Q,4}}{\partial A}=\frac{\partial\hat{\Phi}_{Q,4}}{\partial\Lambda}=0,
\end{align}
\begin{align}
\label{5.26}
\Phi_{Q}(P,\Lambda,A,\psi)=0~\mathrm{if~and~only~in}~
\hat{\Phi}_{Q}(P,\Lambda,A,\psi)=0,
\end{align}
and
\begin{align}
\label{5.27}
\mathrm{deg}\Big(\Phi_{Q};S_1^*(Q,w)\times S_{\rho_2}(Q,w),0\Big)=\mathrm{deg}\Big(\hat{\Phi}_{Q};S_1^*(Q,w)\times S_{\rho_2}(Q,w),0\Big).
\end{align}
Moreover if $\hat{\Phi}_{Q,1}=0$, $\hat{\Phi}_{Q,3}=0$ and $\hat{\Phi}_{Q,4}=0$ if and only if
\begin{align}
\label{5.28}
(p_1,p_2,\cdots,p_m)=(p_1^0,p_2^0,\cdots,p_m^0),~ A=(1,1,\cdots,1),~\psi=0,
\end{align}
and $\hat{\Phi}_{Q,2}=0$ if and only if
\begin{align}
\label{5.29}
\left\{\begin{array}{l}
t_1=t_2=\cdots=t_{n},\\
\rho_1-\rho_*=\pi\sum_jd_je^{-\frac{\lambda_j}{1+\alpha_j}}.
\end{array}\right.
\end{align}
It is not difficult to see that if $|\rho_1-\rho_*|$ is sufficiently small, equation (\ref{5.29}) possesses a unique solution
$$\Lambda(P)=(\lambda_1,\lambda_2,\cdots,\lambda_{n})$$
up to permutation. Hence $(P,\Lambda(P), A,0)$ is the solution of $\hat{\Phi}_{Q},$ where $ A=(1,1,\cdots,1).$ By (\ref{5.25}),
the degree of $\hat{\Phi}_{Q}$
at $(P,\Lambda(P), A,0)$ depends on the number of negative eigenvalue for the following matrix
\begin{align*}
\mathcal{M}=\left[\begin{array}{llll}
\frac{\partial\hat{\Phi}_{Q,1}}{\partial P},&\frac{\partial\hat{\Phi}_{Q,1}}{\partial \Lambda}, &\frac{\partial\hat{\Phi}_{Q,1}}{\partial A},
&\frac{\partial\hat{\Phi}_{Q,1}}{\partial \psi}\\
\frac{\partial\hat{\Phi}_{Q,2}}{\partial P},&\frac{\partial\hat{\Phi}_{Q,2}}{\partial \Lambda}, &\frac{\partial\hat{\Phi}_{Q,2}}{\partial A},
&\frac{\partial\hat{\Phi}_{Q,2}}{\partial \psi}\\
\frac{\partial\hat{\Phi}_{Q,3}}{\partial P},&\frac{\partial\hat{\Phi}_{Q,3}}{\partial \Lambda}, &\frac{\partial\hat{\Phi}_{Q,3}}{\partial A},
&\frac{\partial\hat{\Phi}_{Q,3}}{\partial \psi}\\
\frac{\partial\hat{\Phi}_{Q,4}}{\partial P},&\frac{\partial\hat{\Phi}_{Q,4}}{\partial \Lambda}, &\frac{\partial\hat{\Phi}_{Q,4}}{\partial A},
&\frac{\partial\hat{\Phi}_{Q,4}}{\partial \psi}\\
\end{array}\right]
\end{align*}
Here we say $\mu_{\mathcal{M}}$ is an eigenvalue of $\mathcal{M}$, if there exists $(\nu_1,\nu_2,\cdots,\nu_m)\in(\mathbb{R}^{2})^{m},$
$(\lambda_1,\lambda_2,\cdots,\lambda_{n})\in\mathbb{R}^{n},$ $(a_1,a_2,\cdots,a_{n})\in\mathbb{R}^{n},$ and $\Psi$ such that
\begin{align*}
\mathcal{M}\left[\begin{array}{l}\nu_1\\~\vdots\\\nu_m\\a_1\\~\vdots\\a_{n}\\ \lambda_1\\~\vdots\\ \lambda_{n}\\ ~\Psi\end{array}\right]=
\mu_{\mathcal{M}}\left[\begin{array}{l}~\quad \nu_1\\ \quad\quad\vdots\\ ~\quad \nu_m\\ ~\quad a_1\\ \quad\quad \vdots\\ ~\quad a_{n}\\
~\quad \lambda_1\\\quad\quad\vdots\\ ~\quad \lambda_{n}\\ (-\Delta)^{-1}\Psi\end{array}\right],
\end{align*}
where $\frac{\partial\hat{\Phi}_{Q,1}^j}{\partial \psi}[\Psi]=8\pi\nabla\Psi(p_j^0),$ and
\begin{align*}
\frac{\partial\hat{\Phi}_{Q,4}}{\partial \psi}[\Psi]=&\Psi-(-\Delta)^{-1}\Big(2\rho_2\frac{\overline{h}_2e^{w-4\pi\sum_{j=1}^{m}G(x,p_j^0)}}
{\int_M\overline{h}_2e^{w-4\pi\sum_{j=1}^{m}G(x,p_j^0)}}\Psi\\
&-2\rho_2\frac{\overline{h}_2e^{w-4\pi\sum_{j=1}^{m}G(x,p_j^0)}}
{\big(\int_M\overline{h}_2e^{w-4\pi\sum_{j=1}^{m}G(x,p_j^0)}\big)^2}
\int_M\big(\overline{h}_2e^{w-4\pi\sum_{j=1}^{m}G(x,p_j^0)}\Psi\big)\Big).
\end{align*}

We set $N(T)$ as the number of the negative eigenvalue of matrix $T$,
\begin{align*}
\mathcal{M}_1=\left[\begin{array}{ll}\frac{\partial\hat{\Phi}_{Q,1}}{\partial P},&\frac{\partial\hat{\Phi}_{Q,1}}{\partial \psi}\\
\frac{\partial\hat{\Phi}_{Q,4}}{\partial P},&\frac{\partial\hat{\Phi}_{Q,4}}{\partial \psi}\\
\end{array}\right]~
\mathrm{and}~
\mathcal{M}_2=\left[\begin{array}{ll}\frac{\partial\hat{\Phi}_{Q,2}}{\partial \Lambda},&\frac{\partial\hat{\Phi}_{Q,2}}{\partial A}\\
\frac{\partial\hat{\Phi}_{Q,3}}{\partial \Lambda},&\frac{\partial\hat{\Phi}_{Q,3}}{\partial A}\\
\end{array}\right].
\end{align*}
By using (\ref{5.25}),
\begin{align*}
N(\mathcal{M})=N(\mathcal{M}_1)+N(\mathcal{M}_2)=N(\mathcal{M}_1)+N\Big[\frac{\partial \hat{\Phi}_{Q,2}}{\partial\Lambda}\Big]+N\Big[\frac{\partial
\hat{\Phi}_{Q,3}}{\partial A}\Big],
\end{align*}
Therefore,
\begin{align*}
\deg\Big(\Phi_Q;S_1^*(Q,w)\times &S_{\rho_2}(Q,w),0\Big)=(-1)^{N(\mathcal{M})}=(-1)^{N(\mathcal{M}_1)}\times(-1)^{N(\mathcal{M}_2)}\\
&=(-1)^{N(\mathcal{M})}\times\mathrm{sgn}\Big(\det(\frac{\partial \hat{\Phi}_{Q,2}}{\partial\Lambda})\Big)\times\mathrm{sgn}
\Big(\det(\frac{\partial \hat{\Phi}_{Q,3}}{\partial A})\Big).
\end{align*}
We first consider the last two terms on the
right hand side of above equality. For $\mathrm{det}(\frac{\partial\hat{\Phi}_{Q,3}}{\partial A}),$ it is easy to see that the sign of this value is
positive, since it is a diagonal matrix with every term positive on diagonal. Therefore
$$\mathrm{sgn}~\mathrm{det}(\frac{\partial\hat{\Phi}_{Q,3}}{\partial A})=1.$$
To compute $\mathrm{det}(\frac{\partial\hat{\Phi}_{Q,2}}{\partial\Lambda}),$ we recall that
$$t_j=\lambda_j+\frac{d_j}{2(1+\alpha_j)^2}\lambda_je^{-\frac{\lambda_j}{1+\alpha_j}}-\sum_{j=1}^{n}\overline{v}_{p_j}+\mathrm{constant}.$$
Thus
$$\frac{\partial t_j}{\partial \lambda_i}=\Big[1+(\frac{d_j}{2(1+\alpha_j)}-\frac{d_j}{2(1+\alpha_j)}\lambda_j)e^{-\frac{\lambda_j}{1+\alpha_j}}\Big]
\delta_{ij}-\frac{\partial\overline{v}_{p_i}}{\partial\lambda_i}.$$
By (\ref{5.24}), we have
\begin{align*}
\frac{\partial\hat{\Phi}_{Q,2}^j}{\partial\lambda_j}=-\sum_{i\neq j}(1+\alpha_i)\frac{\partial t_j}{\partial\lambda_j}
+O(e^{-\frac{\lambda_j}{1+\alpha_j}})
=&-\sum_{i\neq j}(1+\alpha_i)[1-\frac{d_j}{2(1+\alpha_j)^2}\lambda_je^{-\frac{\lambda_j}{1+\alpha_j}}]\\
&+O(e^{-\frac{\lambda_j}{1+\alpha_j}})
\end{align*}
and
\begin{align*}
\frac{\partial\hat{\Phi}_{Q,2}^j}{\partial\lambda_i}=(1+\alpha_i)\frac{\partial t_i}{\partial\lambda_i}-\frac{d_i}{4(1+\alpha_i)}
e^{-\frac{\lambda_i}{1+\alpha_i}}
=&(1+\alpha_i)[1-\frac{d_i}{2(1+\alpha_i)^2}\lambda_ie^{-\frac{\lambda_i}{1+\alpha_i}}]\\
&+O(e^{-\frac{\lambda_i}{1+\alpha_i}})
\end{align*}
for $i\neq j$, here we replace $\hat{\Phi}_{Q,2}$ by $\frac{\rho^*}{32\pi^2(1+\alpha_j)}\hat{\Phi}_{Q,2}$ (still denoted by $\hat{\Phi}_{Q,2}$). Denote
\begin{align}
\label{5.30}
B=\sum_i(1+\alpha_i),~E_i=1-\frac{d_i}{2(1+\alpha_i)^2}\lambda_ie^{-\frac{\lambda_i}{1+\alpha_i}},~\delta_j=
\sum_i\frac{\partial\hat{\Phi}_{Q,2}^j}{\partial\lambda_i}.
\end{align}
Thus, we have
\begin{align*}
&\mathrm{det}\Big[(\frac{\partial\hat{\Phi}_{Q,2}^j}{\partial\Lambda})\Big]\\
=&\mathrm{det}\left[\begin{array}{lllll}
(1+\alpha_{1}-B)E_1+(*)  &(1+\alpha_2)E_2+(*)     &\cdots    &(1+\alpha_{n})E_{n}+(*)\\
(1+\alpha_{1})E_1+(*)    &(1+\alpha_2-B)E_2+(*)   &\cdots    &(1+\alpha_{n})E_{n}+(*)\\
\quad\quad\quad\vdots  &\quad\quad\quad\vdots   &~\vdots   &\quad\quad\quad\vdots\\
(1+\alpha_{1})E_1+(*)    &(1+\alpha_2)E_2+(*)     &\cdots    &(1+\alpha_{n}-B)E_{n}+(*)\\\end{array}\right]\\
=&\mathrm{det}\left[\begin{array}{lllll}
\delta_1      &(1+\alpha_2)E_2+(*)   &\cdots    &(1+\alpha_{n})E_{n}+(*)\\
\delta_2      &(1+\alpha_2-B)E_2+(*) &\cdots    &(1+\alpha_{n})E_{n}+(*)\\
~\vdots       &\quad\quad\quad\vdots &~\vdots   &\quad\quad\quad\vdots\\
\delta_{n}  &(1+\alpha_2)E_2+(*)   &\cdots    &(1+\alpha_{n}-B)E_{n}+(*)\\\end{array}\right]\\
=&\mathrm{det}\left[\begin{array}{llllll}
~\quad\delta_1        &(1+\alpha_2)E_2+(*)          &(1+\alpha_3)E_3+(*)    &\cdots     &(1+\alpha_{n})E_{n}+(*)\\
\delta_2-\delta_1     &\quad-BE_2+(*)               &\quad\quad\quad(*)     &\cdots     &\quad\quad\quad\quad(*)\\
\delta_3-\delta_1     &\quad\quad\quad(*)           &\quad-BE_3+(*)         &\cdots     &\quad\quad\quad\quad(*)\\
~\quad\vdots          &\quad\quad\quad\quad~\vdots  &~\quad\quad\quad\vdots &~\vdots    &~\quad\quad\quad\quad\vdots\\
\delta_{n}-\delta_1 &\quad\quad\quad(*)           &\quad\quad\quad(*)     &\cdots     &~\quad-BE_{n}+(*)
\end{array}\right]\\
=&\mathrm{det}\left[\begin{array}{llllll}
\sum_j\frac{1+\alpha_j}{B}\delta_j  &\quad\quad\quad(*)       &\quad\quad\quad(*)      &\quad\cdots   &\quad\quad\quad(*)\\
\quad\delta_2-\delta_1              &~-BE_2+(*)               &\quad\quad\quad(*)      &\quad\cdots   &\quad\quad\quad(*)\\
\quad\delta_3-\delta_1              &\quad\quad\quad(*)       &~-BE_3+(*)              &\quad\cdots   &\quad\quad\quad(*)\\
~\quad\quad\vdots                   &~\quad\quad\quad\vdots   &~\quad\quad\quad\vdots  &~\quad\vdots  &~\quad\quad\quad\vdots\\
\quad\delta_{n}-\delta_1          &\quad\quad\quad(*)       &\quad\quad\quad(*)      &\quad\cdots   &~-BE_{n}+(*)
\end{array}\right],
\end{align*}
where all the terms $(*)$ is bounded by $O(e^{-\frac{\lambda_{m+1}}{1+\alpha_{m+1}}})$. Next, we consider $\sum_j(1+\alpha_j)\delta_j,$
\begin{align*}
&\sum_j(1+\alpha_j)\delta_j\\
=&\sum_j(1+\alpha_j)\Big[-\sum_{i\neq j}(1+\alpha_i)\Big(1-\big(\frac{d_j}{2(1+\alpha_j)^2}\lambda_j-\frac{d_j}{2(1+\alpha_j)}\big)
e^{-\frac{\lambda_j}{1+\alpha_j}}-\frac{\partial\overline{v}_{p_j}}
{\partial\lambda_j}\Big)\\
&-\frac{d_j}{4(1+\alpha_j)}e^{-\frac{\lambda_j}{1+\alpha_j}}+\Big(\sum_{i\neq
j}(1+\alpha_i)\Big(1-\big(\frac{d_i}{2(1+\alpha_i)^2}\lambda_i-\frac{d_i}{2(1+\alpha_i)}\big)
e^{-\frac{\lambda_i}{1+\alpha_i}}-
\frac{\partial\overline{v}_{p_i}}{\partial\lambda_i}\Big)\\
&-\frac{d_i}{4(1+\alpha_i)}e^{-\frac{\lambda_i}{1+\alpha_i}}\Big)\Big]\\
=&-\sum_j(1+\alpha_j)\sum_j\frac{d_j}{4(1+\alpha_j)}e^{-\frac{\lambda_j}{1+\alpha_j}}
+O(e^{-\frac{\lambda_{m+1}}{1+\alpha_{m+1}}-\epsilon\lambda_{m+1}})\\
=&-\frac{B}{4(1+\alpha_{m+1})}\sum_jd_je^{-\frac{\lambda_j}{1+\alpha_j}}
+O(e^{-\frac{\lambda_{m+1}}{1+\alpha_{m+1}}-\epsilon\lambda_{m+1}})\\
=&-\frac{B}{4\pi(1+\alpha_{m+1})}(\rho-\rho_*)+O(e^{-\frac{\lambda_{m+1}}{1+\alpha_{m+1}}-\epsilon\lambda_{m+1}})
\end{align*}
for some $\epsilon>0$. Thus
\begin{equation}
\label{5.31}
\mathrm{det}\Big[(\frac{\partial\hat{\Phi}_{Q,2}^j}{\partial\Lambda})\Big]=(-1)^{n}\frac{B^{n-1}}{4\pi(1+\alpha_{m+1})}(\rho-\rho_*)
+O(e^{-\frac{\lambda_{m+1}}{1+\alpha_{m+1}}-\epsilon\lambda_{m+1}}).
\end{equation}
It remains to compute $N(\mathcal{M}_1).$ According to the definition, we have
\begin{align}
\label{5.32}
\Big[\frac{\partial(\hat{\Phi}_{Q,1},\hat{\Phi}_{Q,4})}{\partial(P,\psi)}\Big]
\left(\begin{array}{l}
\nu_1\\
\nu_2\\
\vdots\\
\nu_m\\
\Psi
\end{array}
\right)
=\left(\begin{array}{ll}
&-\mathcal{I}_1+\nabla\Psi(p_1^0)\\
&-\mathcal{I}_2+\nabla\Psi(p_2^0)\\
&~\quad\quad\vdots\\
&-\mathcal{I}_m+\nabla\Psi(p_m^0)\\
&\quad\quad-\mathcal{I}_0
\end{array}\right),
\end{align}
where
\begin{align*}
\mathcal{I}_i=\nabla^2_xH_i(0,0)\cdot\nu_i+8\pi\sum_{j=1,j\neq i}^m\nabla^2_xG(x,p_j^0)\mid_{x=p_i^0}\cdot\nu_j,~i=1,2,\cdots,m,
\end{align*}
and
\begin{align*}
\mathcal{I}_0=&-\Psi+(-\Delta)^{-1}\Big(2\rho_2\frac{\overline{h}_2e^{w-4\pi\sum_{j=1}^{m}G(x,p_j^0)}}
{\int_M\overline{h}_2e^{w-4\pi\sum_{j=1}^{m}G(x,p_j^0)}}\Psi\\
&-2\rho_2\frac{\overline{h}_2e^{w-4\pi\sum_{j=1}^{m}G(x,p_j^0)}}
{\big(\int_M\overline{h}_2e^{w-4\pi\sum_{j=1}^{m}G(x,p_j^0)}\big)^2}\int_M\big(\overline{h}_2e^{w-4\pi\sum_{j=1}^{m}G(x,p_j^0)}\Psi\big)\\
&-4\pi\rho_2\frac{\overline{h}_2e^{w-4\pi\sum_{j=1}^{m}G(x,p_j^0)}}
{\int_M\overline{h}_2e^{w-4\pi\sum_{j=1}^{m}G(x,p_j^0)}}\big(\sum_{j=1}^m\nabla G(x,p_j^0)\cdot \nu_j\big)\\
&+4\pi\rho_2\frac{\overline{h}_2e^{w-4\pi\sum_{j=1}^{m}G(x,p_j^0)}}
{\big(\int_M\overline{h}_2e^{w-4\pi\sum_{j=1}^{m}G(x,p_j^0)}\big)^2}
\int_M\Big[\overline{h}_2e^{w-4\pi\sum_{j=1}^{m}G(x,p_j^0)}\big(\sum_{j=1}^m\nabla G(x,p_j^0)\cdot \nu_j\big)\Big]\Big).
\end{align*}
According to the definition of the topological degree for the solution to the shadow system (\ref{1.12}), we can get $(-1)^{N(\mathcal{M}_1)}$ is exactly
the
Leray-Schauder topological degree contributed by $(P_w,w).$ Therefore, we proved Theorem \ref{th1.4}.\quad\quad\quad\quad\quad\quad\quad $\square$\\

\noindent{\em Proof of Theorem \ref{th1.3}.} Theorem \ref{th1.3} is a consequence of Theorem \ref{th1.4}.\quad\quad\quad\quad\quad$\square$

\vspace{1cm}
\section{Proof of Theorem \ref{th1.5} and Theorem \ref{1.6}}

This section is devoted to prove Theorem \ref{1.5}. We first introduce a deformation to decouple the system (\ref{1.11}).

\begin{equation}
\label{6.1}
(S_t)
\left\{\begin{array}{l}
\Delta w+2\rho_2(\frac{h_2e^{w-4\pi G(x,p)}}
{\int_Mh_2e^{w-4\pi G(x,p)}}-1)=0,\\
\nabla\big(\log(h_1e^{-\frac12w\cdot(1-t)})+4\pi R(x,x)\big)\mid_{x=p}=0.
\end{array}\right.
\end{equation}

We can easily see that the system (\ref{6.1}) is exactly (\ref{1.11}) when $t=0,$ and will be a decoupled system when $t=1.$ During the deformation
from $(S_1)$ to $(S_0),$ we have
\begin{lemma}
\label{le6.1}
Let $\rho_2\notin4\pi\mathbb{N}$. Then there is uniform constant $C_{\rho_2}$ such that for all solutions to (\ref{6.1}),
we have $|w|_{L^{\infty}(M)}<C_{\rho_2}.$
\end{lemma}

\begin{proof}
Since $\rho_2\notin4\pi\mathbb{N},$ then we can see any solution for the following equation
\begin{equation}
\label{6.2}
\Delta w+2\rho_2(\frac{h_2e^{w-4\pi G(x,p)}}{\int_Mh_2e^{w-4\pi G(x,p)}}-1)=0
\end{equation}
is uniformly bounded above. By using the classical elliptic estimate, we have $|w|_{C^1(M)}<C.$ This constant $C$ depends on $\rho_2.$
\end{proof}

\noindent{\em Proof of Theorem \ref{th1.5}.} It is known that the topological degree is independent of $h_1$ and $h_2$ as long as they are positive $C^1$ functions. By Remark 3 in section 4, we always can choose $h_1$ and $h_2$ such that the hypothesis of Theorem \ref{th3.1} holds.

Let $d_S$ denote the Leray-Schauder degree for (\ref{1.1}). By Lemma \ref{le6.1}, computing the topological degree for (\ref{1.11}) is reduced to compute the
topological
degree for system (\ref{6.1}) when $t=1,$
\begin{align}
\label{6.3}
\left\{\begin{array}{l}
\Delta w+2\rho_2(\frac{h_2e^{w-4\pi G(x,p)}}
{\int_Mh_2e^{w-4\pi G(x,p)}}-1)=0,\\
\nabla[\log h_1+4\pi R(x,x)]\mid_{x=p}=0.
\end{array}\right.
\end{align}
Since this is a decoupled system, the topological degree of (\ref{6.3})
equals the product of the degree of first equation and degree contributed by the second equation. By Poincare-Hopf Theorem, the degree of the second
equation is $\chi(M).$ On the other hand, by Theorem A, the topological degree for the first equation is $b_k+b_{k-1},$ where $b_{k}$ is given (\ref{1.17}). Therefore,
\begin{align}
\label{6.4}
d_S=\chi(M)\cdot(b_k+b_{k-1}).
\end{align}
Combined with Lemma \ref{le6.1}, we get Theorem \ref{th1.5}.         \quad\quad\quad\quad\quad\quad\quad\quad\quad\quad\quad\quad\quad$\square$
\\

\noindent {\em Proof of Theorem \ref{th1.6}.} Theorem \ref{th1.6} is a consequence of Theorem \ref{th1.4}, Theorem \ref{th1.5} and Theorem A.   $\square$

\vspace{1cm}
\section{The Dirichlet Problem}
In this section we consider the Dirichlet problem of $SU(3)$ Toda system. Let $\Omega$ be a bounded smooth domain in $\mathbb{R}^2$ and
$h_1,~h_2$ are two positive $C^{2,\alpha}$ function in $\Omega.$ We consider
\begin{equation}
\label{9.1}
\left\{\begin{array}{l}
\Delta u_1+2\rho_1\frac{h_1e^{u_1}}{\int_Mh_1e^{u_1}}-\rho_2\frac{h_2e^{u_2}}{\int_{M}h_2e^{u_2}}=0,~u_1=0~\mathrm{on}~\partial\Omega,\\
\Delta u_2-\rho_1\frac{h_1e^{u_1}}{\int_Mh_1e^{u_1}}+2\rho_2\frac{h_2e^{u_2}}{\int_{M}h_2e^{u_2}}=0,~u_2=0~\mathrm{on}~\partial\Omega.
\end{array}
\right.
\end{equation}

By \cite[Theorem 1.1 ]{lwz1}, we know that the blow up never occurs on the boundary. Therefore, we can use all the arguments for (\ref{1.6}) with minor modification to get the corresponding result for Dirichlet boundary problem (\ref{9.1}).

\begin{theorem}
\label{th9.1}
Suppose $h_1,h_2$ are two positive $C^{2,\alpha}$ function in $\Omega$ and the assumption $(i),(ii)$. Then there exists a positive constant $c$
such that for any solution of equation (\ref{9.1}), there holds:
$$|u_1(x)|,|u_2(x)|\leq c,~\forall x\in M,~i=1,2.$$
\end{theorem}

%

In order to state our degree formula for (\ref{9.1}), we introduce the following generating function
\begin{align*}
\Xi_{\Omega}(x)=(1+x+x^2+x^3\cdots)^{-\chi(\Omega)+1}=\mathfrak{b}_0+\mathfrak{b}_1x^1+\mathfrak{b}_2x^2+\cdots+\mathfrak{b}_kx^k+\cdots,
\end{align*}
where $\chi(\Omega)$ denotes the Euler characteristic number for $\Omega$, then we have the following theorem

\begin{theorem}
\label{th9.2}
Suppose $d_{\rho_1,\rho_2}^{(2)}$ denotes the topological degree for (\ref{9.1}) when $\rho_2\in(4k\pi,4(k+1)\pi)$, then
\begin{align*}
d_{\rho_1,\rho_2}^{(2)}=
\left\{\begin{array}{ll}
\mathfrak{b}_k,&\rho_1\in(0,4\pi),\\
\mathfrak{b}_k-\chi(\Omega)(\mathfrak{b}_k+\mathfrak{b}_{k-1}),&\rho_1\in(4\pi,8\pi),
\end{array}\right.
\end{align*}
where $\mathfrak{b}_{-1}=0.$
\end{theorem}

\begin{corollary}
\label{co9.1}
If $\Omega$ is not simply connected, then (\ref{9.1}) has a solution for $\rho_1\in(0,4\pi)\cup(4\pi,8\pi)$ and $\rho_2\notin 4\pi\mathbb{N}.$
\end{corollary}

\vspace{1cm}
\section{Proof of Lemma \ref{le4.4} and (\ref{5.14})}
This section is devoted to prove Lemma \ref{le4.4} and (\ref{5.14}). Let
\begin{align*}
\overline{v}_{\alpha}:=&\frac{\int_{B_{r_0}(0)}\frac{|y|^{2\alpha}e^{\lambda}}{(1+e^{\lambda}|y|^{2+2\alpha})^2}v(y)\mathrm{d}y}
{\int_{\mathbb{R}^2}\frac{|y|^{2\alpha}e^{\lambda}}{(1+e^{\lambda}|y|^{2+2\alpha})^2}\mathrm{d}y}
=\frac{1+\alpha}{\pi}\int_{B_{r_0}(0)}\frac{|y|^{2\alpha}e^{\lambda}}{(1+e^{\lambda}|y|^{2+2\alpha})^2}v(y)\mathrm{d}y.
\end{align*}
Then we have the following Poincare-type inequality:
\begin{align}
\label{7.1}
\int_{B_{r_0}(0)}\frac{|y|^{2\alpha}e^{\lambda}}{(1+e^{\lambda}|y|^{2+2\alpha})^2}\phi^2(y)\mathrm{d}y
\leq c(\|\phi\|_{H^1(B_{r_0}(0))}^2+\overline{\phi}_{\alpha}^2)
\end{align}
for some constant $c$ independent of $\lambda$ (see \cite[Lemma 6.2]{cl3} for example). Using (\ref{7.1}) we can prove the following result.

\begin{lemma}
\label{le7.1}
Let $ P=(p_1,p_2,\cdots,p_{n})$ and $\Lambda=(\lambda_1,\cdots,\lambda_{n})$. Assume $\phi\in O_{P,\Lambda}^{(1)}$. Then there is a constant $c$ and
$\epsilon>0$ such that for large $\lambda_j$
\begin{equation}
\label{7.2}
\int_{B_{r_0}(p_j)}|y|^{2\alpha_j}e^{U_j}\phi\mathrm{d}y\leq ce^{-\epsilon\lambda_j}\|\phi\|_{H^1},
\end{equation}
and
\begin{equation}
\label{7.3}
\int_{B_{r_0}(p_j)}|y|^{2\alpha_j}e^{U_j}\phi^2\mathrm{d}y=O(1)\|\phi\|_{H^1}^2.
\end{equation}
\end{lemma}
\noindent For a proof, see \cite{cl3}.

\vspace{0.5cm}

\noindent{\em Proof of Lemma \ref{le4.4}}. We start with part (1). Let $\phi\in O_{P,\Lambda}^{(1)}$ and $\psi\in O_{P,\Lambda}^{(2)}.$ Recall
$2v_1=v_{P,\Lambda,A}+2\phi,$ $\phi\in O_{P,\Lambda}^{(1)}$ and $v_2=\frac{1}{2}w+\psi,$ $\psi\in O_{P,\Lambda}^{(2)}.$ We compute
$$\l\nabla(2v_1+T_1(v_1,v_2)),\nabla \phi_1\rangle=-\l\Delta(2v_1+T_1(v_1,v_2)),\phi_1\rangle.$$
Here we will use the decomposition of $\Delta(2v_1+T_1(v_1,v_2))$ in (\ref{4.45})-(\ref{4.47}).
\begin{align*}
\l\nabla(2v_1+T_1(v_1,v_2)),\nabla\phi_1\r=&\int2\nabla\phi\cdot\nabla\phi_1-\sum_j\int_{B_{r_0}(p_j)}4\rho_1h_{p_j}(p_j)e^{I_j}\phi\phi_1
\nonumber\\&+\mathrm{remainders}
\nonumber\\
:=&2\mathfrak{B}(\phi,\phi_1)+\mathrm{remainder~terms}.
\end{align*}
Clearly, $\mathfrak{B}$ is a symmetric bilinear form in $O_{P,\Lambda}^{(1)}$. For the remainder terms, by (\ref{7.2}) and $\phi_1\in O_{P,\Lambda}^{(1)}$, we have for large $\lambda_{m+1},$
\begin{align}
\label{7.4}
\Big|\Big(\frac{e^{t_j}}{\int_Mh_1e^{2v_1-v_2}}-1\Big)\int_{B_{r_0}(0)}\phi_1\rho_1h_{p_j}(p_j)|y|^{2\alpha_j}e^{U_j}\mathrm{d}y\Big|
=O(e^{-(\frac{1}{1+\alpha_{m+1}}+\epsilon)\lambda_{m+1}})\|\phi_1\|_{H^1},
\end{align}
\begin{align}
\label{7.5}
&\lambda_j(a_j-1)\int_{B_{r_0}(p_j)}\nabla H_j\cdot y\rho_1h_{p_j}(p_j)|y|^{2\alpha_j}e^{U_j}\phi_1\mathrm{d}y\nonumber\\
\leq& O(1)\lambda_j|a_j-1|e^{-\frac{\lambda_{m+1}}{2(1+\alpha_{m+1})}}\Big[\int_{B_{r_0}(p_j)}|y|^{2+2\alpha_j}e^{U_j}\mathrm{d}y\Big]^{\frac12}
\|\phi_1\|_{H^1}\nonumber\\
\leq& O(1)\lambda_j|a_j-1|e^{-\frac{\lambda_{m+1}}{2(1+\alpha_{m+1})}}\|\phi_1\|_{H^1}=o(1)e^{-\frac{\lambda_{m+1}}{1+\alpha_{m+1}}}\|\phi_1\|_{H^1}.
\end{align}
Similarly, we have
\begin{align}
\label{7.6}
\lambda_j(a_j-1)\int_{B_{r_0}(p_j)}|y|^{2\alpha_j}e^{U_j}\eta_j\phi_1\mathrm{d}y=O(1)e^{-\frac{\lambda_{m+1}}{1+\alpha_{m+1}}}\|\phi_1\|_{H^1}.
\end{align}
By Lemma \ref{le7.1}, we have for large $\lambda_{m+1}$
\begin{align}
\label{7.7}
&\int_{B_{r_0}(p_j)}\rho_1h_{p_j}(p_j)|y|^{2\alpha_j}e^{U_j}(a_j-1)(U_j+s_j-1)\phi_1\mathrm{d}y\nonumber\\
&\quad=2\lambda_j\int_{B_{r_0}(p_j)}\rho_1h_{p_j}(p_j)|y|^{2\alpha_j}e^{U_j}(a_j-1)\phi_1\mathrm{d}y\nonumber\\
&\quad\quad+\int_{B_{r_0}(p_j)}\rho_1h_{p_j}(p_j)|y|^{2\alpha_j}e^{U_j}(a_j-1)(U_j-\lambda_j+O(1))\phi_1\mathrm{d}y\nonumber\\
&\quad=2\lambda_j(a_j-1)O(e^{-\epsilon\lambda_j})\|\phi_1\|_{H^1}\nonumber\\
&\quad\quad+O(a_j-1)\Big(\int_{B_{r_0}(p_j)}|y|^{2\alpha_j}e^{U_j}(U_j-\lambda_j+O(1))^2\mathrm{d}y\Big)^{\frac12}
\Big(\int_{B_{r_0}(p_j)}|y|^{2\alpha_j}e^{U_j}\phi_1^2\Big)^{\frac12}\nonumber\\
&\quad=O(1)|a_j-1|\|\phi_1\|_{H^1}=O(1)e^{-\frac{\lambda_{m+1}}{1+\alpha_{m+1}}}\|\phi_1\|_{H^1}.
\end{align}
As for $\hat{E}_j,$ we define $E^+$ and $E^-$ as before:
\begin{equation*}
E^+=\left\{\begin{array}{ll}
\hat{E}_j~&\mathrm{if}~|\varphi|\geq\varepsilon_2\\
0~&\mathrm{if}~|\varphi|<\varepsilon_2,
\end{array}\right.
\quad
E^-=\left\{\begin{array}{ll}
\hat{E}_j~&\mathrm{if}~|\varphi|<\varepsilon_2\\
0~&\mathrm{if}~|\varphi|\geq\varepsilon_2,
\end{array}\right.
\end{equation*}
where $\varepsilon_2$ is a small number. Then we use (\ref{4.36}) and similar argument there to obtain
\begin{align}
\label{7.8}
\int_{B_{r_0}(p_j)}|E^+\phi_1|\mathrm{d}y\leq& \Big(\int_{B_{r_0}(p_j)}|E^+|^2\mathrm{d}y\Big)^{\frac12}\Big(\int_{B_{r_0}(p_j)}\phi_1^2\Big)^{\frac12}
\nonumber\\
=&O(e^{-b\lambda_j})\|\phi_1\|_{H^1}
\end{align}
for any fixed $b>0.$ For $E^-$, we use (\ref{4.25}) and Lemma \ref{le7.1} to obtain
\begin{align*}
\int_{B_{r_0}(p_j)}|E^-\phi_1|\mathrm{d}y
\leq& O(1)\int_{B_{r_0}(p_j)}h_{p_j}(p_j)|y|^{2\alpha_j}e^{U_j}(O(\phi^2)+O(\beta_j))\phi_1\mathrm{d}y\\
=&O(\varepsilon_2)\Big(\int_{B_{r_0}(0)}|y|^{2\alpha_j}e^{U_j}\phi^2\Big)^{\frac12}
\Big(\int_{B_{r_0}(0)}|y|^{2\alpha_j}e^{U_j}\phi_1^2\mathrm{d}y\Big)^{\frac12}\\
&+O(e^{-\frac{2\lambda_{m+1}}{1+\alpha_{m+1}}})\Big(\int_{B_{r_0}(0)}|y|^{2\alpha_j}e^{U_j}\phi_1^2\mathrm{d}y\Big)^{\frac12}\\
&+\int_{B_{r_0}(0)}|y|^{2\alpha_j}e^{U_j}|y|^3|\phi_1|\\
=&O(\varepsilon_2)\|\phi\|_{H^1}\|\phi_1\|_{H^1}+O(e^{-\frac{2\lambda_{m+1}}{1+\alpha_{m+1}}})\|\phi_1\|_{H^1}\\
&+\Big(\int_{B_{r_0}(0)}|y|^{2b\alpha_j}e^{bU_j}|y|^6\Big)^{\frac12}
\Big(\int_{B_{r_0}(0)}e^{(2-b)U_j}|y|^{2(2-b)\alpha_j}\phi_1^2\mathrm{d}y\Big)^{\frac12}\\
=&O(\varepsilon_2e^{-\frac{\lambda_{m+1}}{1+\alpha_{m+1}}})\|\phi_1\|_{H^1}+O(e^{-(2-\frac{b}{2})\frac{\lambda_{m+1}}{1+\alpha_{m+1}}})\|\phi_1\|_{H^1}
\end{align*}
for $2>b>\frac{4}{2+\alpha_j}$.
\begin{equation}
\label{7.9}
\int_{B_{r_0}(p_j)}|E^-\phi_1|\mathrm{d}y=O(e^{-\frac{\lambda_{m+1}}{1+\alpha_{m+1}}})\|\phi_1\|_{H^1},
\end{equation}
provided that $\varepsilon_2c_1<1$. By Lemma \ref{le4.1},
\begin{equation}
\label{7.10}
\int_{B_{2r_0}(p_j)\setminus B_{r_0}(p_j)}\Delta(v_{p_j}-4\pi(1+\alpha_j)G(x,p_j))\phi_1=O(e^{-\frac{\lambda_{m+1}}{1+\alpha_{m+1}}})\|\phi_1\|_{H^1}.
\end{equation}
For the nonlinear term in $\Delta T_1(v_1,v_2)$ on $M\setminus\bigcup_{j=1}^{n}B_{r_0}(p_j)$, we first note that
\begin{align*}
\int_{M\setminus\cup_{j=1}^{n}B_{r_0}(p_j)}|e^{\varphi}\phi_1|=&
O(1)\Big(\int_{|\varphi|\geq\varepsilon_2}|e^{\varphi}\phi_1|+\int_{|\varphi|<\varepsilon_2}|e^{\varepsilon_2}\phi_1|\Big)\\
=&O(1)\|\phi_1\|_{H^1}
\end{align*}
by (\ref{4.36}). Using $\int_Mh_1e^{2v_1-v_2}=O(e^{\lambda_{m+1}})$, we have
\begin{align}
\label{7.11}
\int_{M\setminus\cup_{j=1}^{n}B_{r_0}(p_j)}\frac{\rho_1h_1e^{2v_1-v_2}}{\int_Mh_1e^{2v_1-v_2}}|\phi_1|
&=O(e^{-\lambda_{m+1}})\int_{M\setminus\cup_{j=1}^{n}B_{r_0}(p_j)}e^{\varphi}|\phi_1|\nonumber\\
&=O(e^{-\lambda_{m+1}})\|\phi_1\|_{H^1}.
\end{align}
In the end, we need to consider the term $\int_{B_{r_0}(p_j)}2\rho_1h_{p_j}(p_j)e^{U_j}\psi\phi_1.$ We have
\begin{align}
\label{7.12}
\int_{B_{r_0(p_j)}}2\rho_1h_{p_j}(p_j)e^{U_j}\psi\phi_1=&\int_{B_{r_0(p_j)}}2\rho_1h_{p_j}(p_j)e^{U_j}\psi(p_j)\phi_1\nonumber\\
&+\int_{B_{r_0(p_j)}}2\rho_1h_{p_j}(p_j)e^{U_j}(\psi-\psi(p_j))\phi_1\nonumber\\
=&o(e^{-\frac{\lambda_{m+1}}{1+\alpha_{m+1}}})\|\phi_1\|_{H^1},
\end{align}
where we used (\ref{7.2}). Combining (\ref{7.4})-(\ref{7.12}), we obtain
\begin{align*}
\l\nabla(2v_1+T_1(v_1,v_2)),\nabla\phi_1\r=\mathfrak{B}(\phi,\phi_1)+O(e^{-\frac{\lambda_{m+1}}{1+\alpha_{m+1}}})\|\phi_1\|_{H^1}.
\end{align*}
This proves part (1).

Next, we prove part (3). On $B_{2r_0}(p_j),$ we have
\begin{align}
\label{7.13}
\partial_{\lambda_j}v_{p_j}=&\Big{(}2-\frac{\frac{\rho_1h_{p_j}(p_j)}{2(1+\alpha_j)^2}e^{\lambda_j}|x-p_j|^{2(1+\alpha_j)}}
{1+\frac{\rho_1h_{p_j}(p_j)}{4(1+\alpha_j)^2}e^{\lambda_j}|x-p_j|^{2(1+\alpha_j)}}+O(\lambda_je^{-\frac{\lambda_j}{1+\alpha_j}})\Big)\sigma_j\nonumber\\
=&\Big[(1+\partial_{\lambda_j}U_j)+O(\lambda_je^{-\frac{\lambda_j}{1+\alpha_j}})\Big]\sigma_j
\end{align}
by the setting of $v_{p_j}$. On $M\setminus\bigcup_jB_{2r_0}(p_j),$ $\partial_{\lambda_j}v_{p_j}=0.$
We compute $\l\nabla(2v_1+T_1(v_1,v_2)),\nabla\partial_{\lambda_j}v_{p_j}\r=-\l\Delta(2v_1+T_1(v_1,v_2)),\partial_{\lambda_j}v_{p_j}\r$ by using
(\ref{4.45})-(\ref{4.47}).

Since $\phi\in O_{P,\Lambda}^{(1)},$ we have
\begin{align*}
\int_{M}\nabla\phi\cdot\nabla\partial_{\lambda_j}v_{p_j}=0.
\end{align*}
Direct computation yields,
\begin{align}
\label{7.14}
\int_{B_{r_0}(p_j)}\partial_{\lambda_j}v_{p_j}=\int_{B_{r_0}(p_j)}(1+\partial_{\lambda_j}U_j)+O(\lambda_je^{-\frac{\lambda_j}{1+\alpha_j}})=
O(\lambda_je^{-\frac{\lambda_j}{1+\alpha_j}}).
\end{align}
Hence,
\begin{align}
\label{7.15}
(|\rho_*-\rho_1|+|a_i-1|)\int_{B_{r_0}(p_j)}\partial_{\lambda_j}v_{p_j}\mathrm{d}y=O(e^{-\frac{\lambda_{m+1}}{1+\alpha_{m+1}}-\epsilon\lambda_{m+1}})
\end{align}
for some $\epsilon>0$. By (\ref{7.13}),
\begin{align}
\label{7.16}
&\int_{B_{r_0}(p_j)}\rho_1h_{p_j}(p_j)|x-p_j|^{2\alpha_j}e^{U_j}\partial_{\lambda_j}v_{p_j}\mathrm{d}y\nonumber\\
&\quad\quad=\int_{B_{r_0}(p_j)}\rho_1h_{p_j}(p_j)|x-p_j|^{2\alpha_j}e^{U_j}\Big(1+\partial_{\lambda_j}U_j
+O(\lambda_je^{-\frac{\lambda_j}{1+\alpha_j}})\Big)\mathrm{d}y
\nonumber\\
&\quad\quad=\int_{\mathbb{R}^2}\rho_1h_{p_j}(p_j)|x-p_j|^{2\alpha_j}e^{U_j}(1+\partial_{\lambda_j}U_j)\mathrm{d}z
+O(\lambda_je^{-\frac{\lambda_j}{1+\alpha_j}})\nonumber\\
&\quad\quad=4\pi(1+\alpha_j)+O(\lambda_je^{-\frac{\lambda_j}{1+\alpha_j}}),
\end{align}
and
\begin{align}
\label{7.17}
&\int_{B_{r_0}(p_j)}2\rho_1h_{p_j}(p_j)|x-p_j|^{2\alpha_j}e^{U_j}\Big[-2\log\Big(1+\frac{\rho_1h_{p_j}(p_j)
e^{\lambda_j}}{4(1+\alpha_j)^2}|x-p_j|^{2(1+\alpha_j)}\Big)
\Big]\partial_{\lambda_j}v_{p_j}\nonumber\\
&=\int_{\mathbb{R}^2}\frac{8(1+\alpha_j)^2r^{2\alpha_j}}{(1+r^{2(1+\alpha_j)})^2}[-2\log(1+r^{2(1+\alpha_j)})]
\Big(\frac{2}{1+r^{2(1+\alpha_j)}}+O(\lambda_je^{-\frac{\lambda_j}{1+\alpha_j}})\Big)\mathrm{d}z\nonumber\\&\quad+O(e^{-\lambda_j})\nonumber\\
&=-8\pi(1+\alpha_j)+O(e^{-\frac{\lambda_j}{1+\alpha_j}}).
\end{align}
(\ref{7.16}) and (\ref{7.17}) together give
\begin{align}
\label{7.18}
&\int_{B_{r_0}(p_j)}\Big(2\rho_1h_{p_j}(p_j)|x-p_j|^{2\alpha_j}e^{U_j}\big[(a_j-1)(U_j+s_j-1)\nonumber\\
&\quad+\sum8\pi(1+\alpha_i)(a_i-1)G(p_j,p_i)+
(\frac{e^{t_j}}{\int_Mh_1e^{2v_1-v_2}}-1)\big]\Big)\mathrm{d}y\nonumber\\
&\quad=8\pi(1+\alpha_j)(2\lambda_j(a_j-1)+\frac{e^{t_j}}{\int_Mh_1e^{2v_1-v_2}}-1)+O(1)\big(\max_i|a_1-1|\big).
\end{align}

To estimate the term with $\phi\partial_{\lambda_j}v_{p_j}$ and $\psi\partial_{\lambda_j}v_{p_j},$ note that $\phi\in O_{P,\Lambda}^{(1)}$ implies
\begin{align*}
0=&\int_M\nabla\phi\nabla\partial_{\lambda_j}v_{p_j}=-\int_M\phi\Delta(\partial_{\lambda_j}v_{p_j})\\
=&-\int_{B_{r_0}(p_j)}\phi\Delta(\partial_{\lambda_j}U_j)\mathrm{d}y+O(\lambda_je^{-\frac{\lambda_j}{1+\alpha_j}}\|\phi\|_{H^1})\\
=&\int_{B_{r_0}(p_j)}2\rho_1h_{p_j}(p_j)|x-p_j|^{2\alpha_j}e^{U_j}\phi\partial_{\lambda_j}U_j\mathrm{d}y+O(\lambda_je^{-\frac{\lambda_j}{1+\alpha_j}}
\|\phi\|_{H^1})
\end{align*}
Together with Lemma \ref{le7.1}, we conclude from the above
\begin{align}
\label{7.19}
&\int_{B_{r_0}(p_j)}2\rho_1h_{p_j}(p_j)|x-p_j|^{2\alpha_j}e^{U_j}\phi\partial_{\lambda_j}v_{p_j}\nonumber\\
&\quad=\int_{B(r_0)(p_j)}2\rho_1h_{p_j}(p_j)|x-p_j|^{2\alpha_j}e^{U_j}\phi(1+\partial_{\lambda_j}U_j
+O(\lambda_je^{-\frac{\lambda_j}{1+\alpha_j}}))\mathrm{d}y
\nonumber\\
&\quad=O(e^{-\frac{\epsilon\lambda_{m+1}}{1+\alpha_{m+1}}})\|\phi\|_{H^1}=O(e^{-\epsilon\lambda_{m+1}-\frac{\lambda_{m+1}}{1+\alpha_{m+1}}}).
\end{align}
While for $\psi\partial_{\lambda_j}v_{p_j},$ we have
\begin{align}
\label{7.20}
&\int_{B_{r_0(p_j)}}\rho_1h_{p_j}(p_j)|x-p_j|^{2\alpha_j}e^{U_j}\psi\partial_{\lambda_j}v_{p_j}\nonumber\\
=&\int_{B_{r_0(p_j)}}\rho_1h_{p_j}(p_j)|x-p_j|^{2\alpha_j}e^{U_j}\psi(p_j)\partial_{\lambda_j}v_{p_j}\nonumber\\
&+\int_{B_{r_0(p_j)}}\rho_1h_{p_j}(p_j)|x-p_j|^{2\alpha_j}e^{U_j}(\psi-\psi(p_j))\partial_{\lambda_j}v_{p_j}\nonumber\\
=&4\pi(1+\alpha_j)\psi(p_j)+O(e^{-\epsilon\lambda_j-\frac{\lambda_j}{1+\alpha_j}}),
\end{align}
The other integrals of $\partial_{\lambda_j}v_{p_j}$ with other terms in (\ref{4.45}) would be smaller than
$O(\max|a_i-1|+e^{-\frac{\lambda_{m+1}}{1+\alpha_{m+1}}-\epsilon\lambda_{m+1}})$.
Since the computations are straightforward, we omit the details here.\\

Now we go to the integral over $M\setminus B_{r_0}(p_j).$ Since
$\partial_{\lambda_j}v_{p_j}=0$ on $M\setminus B_{2r_0}(p_j)$, we only need to consider the integrals on $B_{2r_0}(p_j).$
On $B_{2r_0}(p_j)\setminus B_{r_0}(p_j),$ $e^{2v_1-v_2}=O(e^{\varphi})$, $\frac{\rho_1h_1e^{2v_1-v_2}}{\int_Mh_1e^{2v_1-v_2}}=
O(e^{-\lambda_{m+1}})e^{\varphi}$
and $\partial_{\lambda_j}v_{p_j}=O(e^{-\frac{\lambda_{m+1}}{(1+\alpha_{m+1})}})$. By using Moser-Trudinger's inequality as in the proof in
(\ref{4.36}),
\begin{align}
\label{7.21}
\int_{B_{2r_0}(p_j)\cap\{\varphi\geq\epsilon_2\}}e^{\varphi}\leq e^{-2\lambda_j}\int_{B_{r_0}(p_j)}\exp(\frac{4\pi|\varphi-\overline{\varphi}|^2}
{\|\varphi-\overline{\varphi}\|^2})\mathrm{d}y\leq c_2e^{-2\lambda_j},
\end{align}
which implies
\begin{align}
\label{7.22}
\int_{B_{2r_0}(p_j)}e^{|\varphi|}=O(1)
\end{align}
and
\begin{align}
\label{7.23}
\int_{B_{2r_0}(p_j)\setminus B_{r_0}(p_j)}\frac{\rho_1h_1e^{2v_1-v_2}}{\int_Mh_1e^{2v_1-v_2}}\partial_{\lambda_j}v_{p_j}=
O(e^{-\epsilon\lambda_{m+1}-\frac{\lambda_{m+1}}{1+\alpha_{m+1}}}).
\end{align}
By Lemma \ref{le4.2}, we have
\begin{align}
\label{7.24}
\int_{B_{2r_0}(p_j)\setminus B_{r_0}(p_j)}\Delta(v_{p_j}-\overline{v}_{p_j}-8\pi(1+\alpha_j)G(x,p_j))\cdot\partial_{\lambda_j}v_{p_j}
=O(e^{-\frac{2\lambda_{m+1}}{1+\alpha_{m+1}}}).
\end{align}
By (\ref{4.40}), (\ref{7.15}), (\ref{7.18})-(\ref{7.24}), we obtain,
\begin{align*}
\l\nabla(2v_1+T_1(v_1,v_2)),\nabla\partial_{\lambda_j}v_{p_j}\r=
&-16\pi(1+\alpha_j)(a_j-1)\lambda_j-8\pi(1+\alpha_j)(\theta_j-\psi(p_j))\\
&+O(\max|a_i-1|+e^{-\epsilon\lambda_{m+1}-\frac{\lambda_{m+1}}{1+\alpha_{m+1}}}).
\end{align*}
This proves part (3).

\vspace{0.5cm}
For the proof of part (4), we write
\begin{align*}
\l\nabla(2v_1+T_1(v_1,v_2)),\nabla v_{p_j}\r&=\l\nabla(2v_1+T_1(v_1,v_2),\nabla(v_{p_j}-\overline{v}_{p_j})\r\\
&=-\l\Delta(2v_1+T_1(v_1,v_2)),v_{p_j}-\overline{v}_{p_j}\r.
\end{align*}
First, we have $\l1,v_{p_j}-\overline{v}_{p_j}\r=0$ and $\l\Delta\phi,v_{p_j}-\overline{v}_{p_j}\r=0$ on $M$ because $\phi\in O_{P,\Lambda}^{(1)}.$
To estimate
the other terms, we note that
\begin{align}
\label{7.25}
v_{p_j}-\overline{v}_{p_j}=&2\lambda_j-2\log\Big(1+\frac{\rho_1h_{p_j}(p_j)}{4(1+\alpha_j)^2}e^{\lambda_j}|x-p_j|^{2(1+\alpha_j)}\Big)\nonumber\\
&+8\pi(1+\alpha_j)R(p_j,p_j)+2\log\frac{\rho_1h_{p_j}(p_j)}{4(1+\alpha_j)^2}+O(|x-p_j|)+O(\lambda_je^{-\frac{\lambda_j}{1+\alpha_j}}).
\end{align}

By scaling, we have
\begin{align}
\label{7.26}
&(a_j-1)\lambda_j\int_{\mathbb{R}^2}\rho_1h_{p_j}(p_j)|x-p_j|^{2\alpha_j}e^{U_j}
\log\Big(1+\frac{\rho_1h_{p_j}(p_j)}{4(1+\alpha_j)^2}e^{\lambda_j}|x-p_j|^{2(1+\alpha_j)}
\Big)\mathrm{d}y\nonumber\\
&\quad=2\pi(1+\alpha_j)(a_j-1)\lambda_j.
\end{align}

Let $\vartheta$ represent any constant term in
$$(a_j-1)(U_j+s_j-1),\quad \sum_{i\neq j}(a_i-1)(1+\alpha_i)G(p_j,p_i),\quad\frac{e^{t_j}}{\int_Mh_1e^{2v_1-v_2}}-1.$$
For simplicity of notations, we set $W_j(x)=2\rho_1h_{p_j}(p_j)|x-p_j|^{2\alpha_j}e^{U_j}.$ By comparing (\ref{7.13}) and (\ref{7.25}), we have
\begin{align}
\label{7.27}
\int_{B_{r_0}(p_j)}W_j\vartheta(v_{p_j}-\overline{v}_{p_j})\mathrm{d}y=&\Big(2\lambda_j-1+8\pi(1+\alpha_j)R(p_j,p_j)\nonumber\\
&+2\log\frac{\rho_1h_{p_j}(p_j)}{4(1+\alpha_j)^2}\Big)
\int_{B_{r_0}(p_j)}W_j\vartheta\partial_{\lambda_j}v_{p_j}+O(e^{-\frac{\lambda_j}{1+\alpha_j}}),
\end{align}
where (\ref{7.18}) is used. It is also easy to see the integral of $v_{p_j}-\overline{v}_{p_j}$ with all remainder terms except $\psi$ in (\ref{4.45})
are smaller than
$O(e^{-\frac{\lambda_j}{1+\alpha_j}}+\|\phi\|_{H^1})$. For example, by Lemma \ref{le7.1},
\begin{align*}
&\int_{B_{r_0}(p_j)}2\rho_1h_{p_j}(p_j)|x-p_j|^{2\alpha_j}e^{U_j}\phi(v_{p_j}-\overline{v}_{p_j})\nonumber\\
&\quad\quad=
\int_{B_{r_0}(p_j)}2\rho_1h_{p_j}(p_j)|x-p_j|^{2\alpha_j}e^{U_j}\phi\nonumber\\
&\quad\quad\quad\times\Big[\lambda_j+s_j-2\log(1+\frac{\rho_1h_{p_j}(p_j)
e^{\lambda_j}|x-p_j|^{2(1+\alpha_j)}}{4(1+\alpha_j)^2})+O(|x-p_j|)\Big]\mathrm{d}y
\nonumber\\
&\quad\quad=O(\lambda_je^{-\epsilon\lambda_j})\|\phi\|_{H^1}+\Big(\int_{B_{r_0}(p_j)}2\rho_1h_{p_j}(p_j)e^{U_j}|x-p_j|^{2\alpha_j}\phi^2\Big)^{\frac12}
\nonumber\\
&\quad\quad\quad\times\Big(\int_{B_{r_0}(p_j)}2\rho_1h_{p_j}(p_j)e^{U_j}|x-p_j|^{2\alpha_j}
\Big[\log(1+\frac{\rho_1h_{p_j}(p_j)e^{\lambda_j}|x-p_j|^{2(1+\alpha_j)}}{4(1+\alpha_j)^2})\nonumber\\
&\quad\quad\quad\quad\quad+O(|x-p_j|)\Big]^2\mathrm{d}y\Big)^{\frac12}\nonumber\\
&\quad\quad=O(1)\|\phi\|_{H^1}.
\end{align*}
For $\psi,$ we have
\begin{align}
\label{7.28}
&\int_{B_{r_0}(p_j)}2\rho_1h_{p_j}(p_j)|x-p_j|^{2\alpha_j}e^{U_j}\psi(v_{p_j}-\overline{v}_{p_j})\nonumber\\
&\quad=\int_{B_{r_0}(p_j)}2\rho_1h_{p_j}(p_j)|x-p_j|^{2\alpha_j}e^{U_j}\psi(p_j)(v_{p_j}-\overline{v}_{p_j})\nonumber\\
&\quad\quad+\int_{B_{r_0}(p_j)}2\rho_1h_{p_j}(p_j)|x-p_j|^{2\alpha_j}e^{U_j}(\psi-\psi(p_j))(v_{p_j}-\overline{v}_{p_j})\nonumber\\
&\quad=\int_{B_{r_0}(p_j)}2\rho_1h_{p_j}(p_j)|x-p_j|^{2\alpha_j}e^{U_j}\psi(p_j)\times\Big[\lambda_j+s_j\nonumber\\
&\quad\quad-2\log(1+\frac{\rho_1h_{p_j}(p_j)e^{\lambda_j}|x-p_j|^{2(1+\alpha_j)}}{4(1+\alpha_j)^2})\Big]
+O(e^{-\epsilon\lambda_{m+1}-\frac{\lambda_{m+1}}{1+\alpha_{m+1}}})\nonumber\\
&\quad=8\pi\Big(2\lambda_j-1+8\pi(1+\alpha_j)R(p_j,p_j)+2\log\frac{\rho_1h_{p_j}(p_j)}{4(1+\alpha_j)^2}\Big)(1+\alpha_j)\psi(p_j)
\nonumber\\&\quad\quad+O(e^{-\epsilon\lambda_{m+1}-\frac{\lambda_{m+1}}{1+\alpha_{m+1}}}).
\end{align}

Since $v_{p_j}-\overline{v}_{p_j}=O(1)$ on $M\setminus B_{r_0}(p_j),$ by Lemma \ref{le4.1},
\begin{align}
\label{7.29}
\int_{B_{2r_0}(p_j)\setminus B_{r_0}(p_j)}\Delta(v_{p_j}-8\pi(1+\alpha_j)G(x,p_j))(v_{p_j}-\overline{v}_{p_j})
=O(e^{-\frac{\lambda_{m+1}}{1+\alpha_{m+1}}}).
\end{align}
For the integration on $B_{r_0}(p_i),i\neq j,$ the dominant term can be estimated by
\begin{align}
\label{7.30}
&\int_{B_{r_0}(p_i)}2\rho_1h_{p_i}(p_i)|x-p_i|^{2\alpha_i}e^{U_i}\nonumber\\
&\quad\times\Big[(a_i-1)(U_i-s_i-1)+\big(\frac{e^{t_i}}{\int_Mh_1e^{2v_1-v_2}}-1-\psi\big)\Big](v_{p_j}-\overline{v}_{p_j})\mathrm{d}y\nonumber\\
&=(1+\alpha_j)(1+\alpha_i)\Big[128\pi^2G(p_i,p_j)(a_i-1)\lambda_i\nonumber\\
&\quad+64\pi^2G(p_i,p_j)(\frac{e^{t_i}}{\int_Mh_1e^{2v_1-v_2}}-1-\psi(p_i))\Big]\nonumber\\
&\quad+O(|a_i-1|)+O(e^{-\epsilon\lambda_{m+1}-\frac{\lambda_{m+1}}{1+\alpha_{m+1}}})\nonumber\\
&=64\pi^2(1+\alpha_i)(1+\alpha_j)G(p_i,p_j)\Big[(\frac{e^{t_i}}{\int_Mh_1e^{2v_1-v_2}}-1-\psi(p_i))+2(a_i-1)\lambda_i\Big]\nonumber\\
&\quad+O(|a_i-1|+e^{-\epsilon\lambda_{m+1}-\frac{\lambda_{m+1}}{1+\alpha_{m+1}}})\nonumber\\
&=-8\pi(1+\alpha_j)G(p_i,p_j)\l\nabla(2v_1+T_1(v_1,v_2)),\nabla\partial_{\lambda_i}v_{p_i}\r\nonumber\\
&\quad+O(|a_i-1|+e^{-\epsilon\lambda_{m+1}
-\frac{\lambda_{m+1}}{1+\alpha_{m+1}}}),
\end{align}
where we use the proof part (3) of Lemma \ref{le4.4} in the above, it is easy to see that the other terms on $B_{r_0}(p_i)$ are bounded by
$e^{-\epsilon\lambda_{m+1}-\frac{\lambda_{m+1}}{1+\alpha_{m+1}}}$.

For the integration outside $\bigcup_jB_{r_0}(p_j),$ we have $(\int_Mh_1e^{2v_1-v_2})^{-1}=O(e^{-\lambda_j})$ and
\begin{align}
\label{7.31}
\int_{B_{2r_0}(p_j)\setminus B_{r_0}(p_j)}\frac{\rho_1h_1e^{2v_1-v_2}}{\int_Mh_1e^{2v_1-v_2}}(v_{p_j}-\overline{v}_{p_j})=
\int_{B_{2r_0}(p_j)\setminus B_{r_0}(p_j)}O(e^{-\lambda_j})e^{\varphi}=O(e^{-\lambda_j}).
\end{align}
Similarly
\begin{align}
\label{7.32}
\int_{M\setminus\bigcup_jB_{r_0}(p_j)}\frac{\rho_1h_1e^{2v_1-v_2}}{\int_Mh_1e^{2v_1-v_2}}(v_{p_j}-\overline{v}_{p_j})
=O(e^{-\lambda_j}).
\end{align}
Therefore, by (\ref{7.26})-(\ref{7.32}), we have
\begin{align}
\label{7.33}
&\l\nabla(2v_1+T_1(v_1,v_2)),\nabla(v_{p_j}-\overline{v}_{p_j})\r\nonumber\\
=&\Big(2\lambda_j-1+8\pi(1+\alpha_j)R(p_j,p_j)+2\log\frac{\rho_1h_{p_j}(p_j)}{4(1+\alpha_j)^2}\Big)
\l\nabla(2v_1+T_1(v_1,v_2)),\nabla_{\lambda_j}v_{p_j}\r\nonumber\\
&+8\pi(1+\alpha_j)\sum_{i\neq j}G(p_j,p_i)\l\nabla(2v_1+T_1(v_1,v_2)),\nabla_{\lambda_i}v_{p_i}\r
+16\pi(a_j-1)\lambda_j(1+\alpha_j)\nonumber\\
&+O(1)\|\phi\|_{H^1}+O(e^{-\frac{\lambda_{m+1}}{1+\alpha_{m+1}}}).
\end{align}
The proof of part (4) is complete.
\vspace{0.5cm}

Finally, we prove part (2), we note that
$$\l\nabla(2v_1+T_1(v_1,v_2)),\nabla\partial_{p_j}v_{p_j}\r=
\l\nabla(2v_1+T_1(v_1,v_2)),\nabla\partial_{p_j}(v_{p_j}-\overline{v}_{p_j})\r.$$
From $\phi\in O_{P,\Lambda}^{(1)},$ we have
\begin{align}
\label{7.34}
\int_M\nabla\phi\nabla\partial_{p_j}(v_{p_j}-\overline{v}_{p_j})=0.
\end{align}
Since $\int_M(v_{p_j}-\overline{v}_{p_j})=0,$ we have $\int_M\partial_{p_j}(v_{p_j}-\overline{v}_{p_j})=0$ and
\begin{equation}
\label{7.35}
\int_M\Big[\rho_*-\rho_1+\sum_{j=1}^{n}4\pi(1+\alpha_j)(a_j-1)\Big]
\partial_{p_j}(v_{p_j}-\overline{v}_{p_j})=0.
\end{equation}
On $B_{r_0}(p_j),$ by Lemma \ref{le4.1}
\begin{equation}
\label{7.36}
\partial_{p_j}v_{p_j}=-\nabla_yU_j+\frac{\partial_{p_j}h(p_j)}{h(p_j)}(\partial_{\lambda_j}
U_j-1)+2\partial_{p_j}\log h(p_j)+8\pi\partial_{p_j}R(x,p_j)+O(|x-p_j|).
\end{equation}
Since $\nabla_yU_j$ is an odd function, we have
\begin{align}
\label{7.37}
\int_{B_{r_0}(p_j)}\rho_1h_{p_j}(p_j)|x-p_j|^{2\alpha_j}e^{U_j}(U_j+s_j-1)\nabla_yU_j\mathrm{d}y
=O(1).
\end{align}
Hence, by Lemma \ref{le4.2} and the fact that $\partial_{\lambda_j}U_j$ is bounded,
\begin{align*}
\int_{B_{r_0}(p_j)}\rho_1h_{p_j}(p_j)e^{U_j}(U_j+s_j-1)\partial_{p_j}(v_{p_j}-\overline{v}_{p_j})
=O(\lambda_j),
\end{align*}
where (\ref{4.29}) was used. For the other terms in (\ref{4.45}), we have the following estimates.
\begin{align}
\label{7.38}
\int_{B_{r_0}(p_j)}\rho_1h_{p_j}(p_j)e^{U_j}\partial_{p_j}(v_{p_j}-\overline{v}_{p_j})\mathrm{d}y
=O(1),
\end{align}
\begin{align}
\label{7.39}
\int_{B_{r_0}(p_j)}\rho_1h_{p_j}(p_j)e^{U_j}O(|x-p_j|)\partial_{p_j}(v_{p_j}-\overline{v}_{p_j})
\mathrm{d}y=O(1),
\end{align}
\begin{align}
\label{7.40}
\int_{B_{r_0}(p_j)}2\rho_1h_{p_j}(p_j)e^{U_j}\nabla H_j(p_j)\cdot(x-p_j)\nabla_yU_j\mathrm{d}y
=(8\pi+O(e^{-\lambda_j}))\nabla H_j(p_j),
\end{align}
and
\begin{align}
\label{7.41}
&\int_{B_{r_0}(p_j)}2\rho_1h_{p_j}(p_j)e^{U_j}\nabla H_j(p_j)\cdot(x-p_j)\partial_{p_j}(v_{p_j}-\overline{v}_{p_j})\nonumber\\
&=8\pi\nabla H_j(p_j)+O(\lambda_je^{-\frac32\lambda_j}),
\end{align}
where we used $\nabla H_j(p_j)=O(\lambda_je^{-\lambda_j})$ for $v_1\in S_{\rho_1}(Q,w).$

By Lemma \ref{le4.2} and (\ref{7.36}),
\begin{align}
\label{7.42}
\int_{B_{r_0}(p_j)}\rho_1h_{p_j}(p_j)e^{U_j}\phi\partial_{p_j}(v_{p_j}-\overline{v}_{p_j})\mathrm{d}y
=O(e^{-\epsilon\lambda_{m+1}})\|\phi\|_{H^1(M)}.
\end{align}
While for the term $\frac{e^{t_j}}{\int_Mh_1e^{2v_1-v_2}}-1$ and $\psi$, we have
\begin{align}
\label{7.43}
&\int_{B_{r_0}(p_j)}2\rho_1h_{p_j}(p_j)e^{U_j}\Big(\frac{e^{t_j}}{\int_Mh_1e^{2v_1-v_2}}-1-\psi\Big)\partial_{p_j}(v_{p_j}-\overline{v}_{p_j})\nonumber\\
&=\int_{B_{r_0}(p_j)}2\rho_1h_{p_j}(p_j)e^{U_j}\Big(\frac{e^{t_j}}{\int_Mh_1e^{2v_1-v_2}}-1-\psi(p_j)\Big)\partial_{p_j}(v_{p_j}-\overline{v}_{p_j})
\nonumber\\
&\quad-\int_{B_{r_0}(p_j)}2\rho_1h_{p_j}(p_j)e^{U_j}(\psi-\psi(p_j))\partial_{p_j}(v_{p_j}-\overline{v}_{p_j})\nonumber\\
&=-8\pi\nabla\psi(p_j)+O(|\frac{e^{t_j}}{\int_Mh_1e^{2v_1-v_2}}-1-\psi(p_j)|)
\end{align}
where we used
\begin{align*}
\int_{B_{r_0}(p_j)}2\rho_1h_{p_j}(p_j)e^{U_j}\nabla \psi(p_j)\cdot(x-p_j)\nabla_yU_j\mathrm{d}y
=(8\pi+O(e^{-\lambda_j}))\nabla\psi(p_j),
\end{align*}
and (\ref{7.38}). Since $\partial_{p_j}(v_{p_j}-\overline{v}_{p_j})=O(e^{\frac12\lambda_j}),$ as in the proof of part (3), we have
\begin{align}
\label{7.44}
\int_{B_{r_0}(p_j)}E\partial_{p_j}(v_{p_j}-\overline{v}_{p_j})\mathrm{d}y=O(e^{-\epsilon\lambda_{m+1}-\lambda_j}).
\end{align}
On $M\setminus\bigcup_jB_{r_0}(p_j),$ $\partial_{p_j}(v_{p_j}-\overline{v}_{p_j})=O(1).$ Hence by Lemma \ref{le4.1},
\begin{align*}
\int_{B_{2r_0}(p_j)\setminus B_{r_0}(p_j)}\Delta(v_{p_j}-8\pi(1+\alpha_j)G(x,p_j))\cdot\partial_{p_j}(v_{p_j}-\overline{v}_{p_j})
=O(\lambda_je^{-\lambda_j}).
\end{align*}
Since $(\int_Mh_1e^{2v_1-v_2})^{-1}=O(e^{-\lambda_j})$, the integral of the products of $\partial_{p_j}(v_{p_j}-\overline{v}_{p_j})$ and
the nonlinear term in (\ref{4.45})-(\ref{4.46}) are of order
$$O(e^{-\lambda_j})\int_Me^{\varphi}=O(e^{-\lambda_j}).$$
The estimates above imply
\begin{align}
\label{7.45}
\l\nabla(2v_1+T_1(v_1,v_2)),\nabla\partial_{p_j}(v_{p_j}-\overline{v}_{p_j})\r=&-8\pi\nabla H_j(p_j)+8\pi\nabla\psi(p_j)\nonumber\\
&+O\big(|\frac{e^{t_j}}{\int_Mh_1e^{2v_1-v_2}}-1-\psi(p_j)|\nonumber\\
&+|a_j-1|\lambda_j+e^{-\frac{\lambda_{m+1}}{1+\alpha_{m+1}}}\big).
\end{align}
This proves part (2) and hence the proof of Lemma \ref{le4.4} is complete.      $\square$\\

Next, we give a proof of (\ref{5.14}).\\

\noindent {\em Proof of (\ref{5.14}):} For convenience, we denote
$$\mathfrak{E}_2=\exp(w+2\psi-\frac12 v_{P,\Lambda,A})-\exp\big(w+2\psi-\sum_{j=1}^n4\pi(1+\alpha_j)a_jG(x,p_j)\big).$$

\noindent For $x\in M\setminus\bigcup_{j=1}^{n}B_{r_0}(p_{j}).$ By Lemma \ref{le4.1}, we have
\begin{equation*}
\big|\frac12v_{P,\Lambda,A}-4\pi\sum_{j=1}^n(1+\alpha_j)a_jG(x,p_j)\big|\leq \tilde{c}e^{-\frac{\lambda(p)}{1+\alpha_{m+1}}}
\end{equation*}
for some $\tilde{c}$ independent of $c_0$. Thus $|\mathfrak{E}_2|\leq c_4e^{-\frac{\lambda(p)}{1+\alpha_{m+1}}}$ in $M\setminus\bigcup_{j=1}^{n}B_{r_0}(p_{j})$.\\

\noindent For $x\in B_{r_0}(p_{j}), j\in J_1$, we note
\begin{align*}
4\pi\sum_{j=1}^{n}a_jG(x,p_{j})-\frac12 v_{P,\Lambda,A}=&4\pi a_jG(x,p_{j})-4\pi a_jR(x,p_{j})-a_j\log(\frac{\rho_1h_{p_j}(p_{j})e^{\lambda_{j}}}{4})\\
&+a_j\log\Big(1+\frac{\rho_1h_{p_j}(p_{j})e^{\lambda_{j}}}{4}|x-p_{j}|^2\Big)+O(e^{-\frac{\lambda(P)}{1+\alpha_{m+1}}})\\
=&a_j\log\Big(\frac{4}{\rho_1h_{p_j}(p_{j})e^{\lambda_{j}}|x-p_{j}|^2}+1\Big)+O(e^{-\frac{\lambda(P)}{1+\alpha_{m+1}}}),
\end{align*}
where we have used $\overline{v}_{p_{j}}=O(e^{-\frac{\lambda(P)}{1+\alpha_{m+1}}})$. Then, we have
\begin{align}
\label{8.1}
&\exp\big(w+2\psi-4\pi\sum_{j=1}^na_j(1+\alpha_j)G(x,p_{j})\big)-\exp\big(w+2\psi-\frac12 v_{P,\Lambda,A}\big)\nonumber\\
&=|x-p_j|^{2a_j}\Big(1-\exp\big(4\pi\sum_{j=1}^{n}a_jG(x,p_{j})-\frac12 v_{P,\Lambda,A}\big)\Big)\nonumber\\
&=|x-p_j|^{2a_j}\Big(1-\exp\big(a_j\log(1+\frac{4}{\rho_1h_{p_j}(p_j)e^{\lambda_{j}}|x-p_{j}|^2})
+O(e^{-\frac{\lambda(P)}{1+\alpha_{m+1}}})\big)\Big),
\end{align}
where we used
$$\exp\big(w+2\psi_{k}-4\pi\sum_{j=1}^n(1+\alpha_j)a_jG(x,p_{j})\big)\sim|x-p_{j}|^{2a_j}~\mathrm{for}~x\in B_{r_0}(p_{j}),~j\in J_1.$$
When $|x-p_{j}|=O(e^{-\frac{\lambda_{j}}{2a_j}}),$ we have
$$\exp\big(w+2\psi-4\pi\sum_{j=1}^na_jG(x,p_{j})\big)=O(|x-p_j|^{2a_j}),$$
and
$$1-\exp\big(a_j\log(1+\frac{4}{\rho_1h_{p_j}e^{\lambda_{j}}|x-p_{j}|^2})
+O(e^{-\frac{\lambda(P)}{1+\alpha_{m+1}}})\big)=O(e^{-a_j\lambda_{j}}|x-p_{j}|^{-2a_j}),$$
which implies
$$\mathfrak{E}_2=O(e^{-\frac{\lambda(P)}{1+\alpha_{m+1}}})~\mathrm{for}~|x-p_{j}|=O(e^{-\frac{\lambda_{j}}{2a_j}}),~j\in J_1.$$
When $|x-q_{j}|\gg e^{-\frac{\lambda_{j}}{2a_j}},$ then
\begin{align*}
1-\exp\Big(a_j\log(1+\frac{4}{\rho_1h_{p_j}(p_{j})e^{\lambda_{j}}|x-p_{j}|^2})
\Big)=O(\frac{1}{e^{a_j\lambda_{j}}|x-p_{j}|^{2a_j}}).
\end{align*}
As a result, we have the right hand side of (\ref{8.1}) are of order $O(e^{-\frac{\lambda(P)}{1+\alpha_{m+1}}}).$
Therefore $$\mathfrak{E}_2=O(e^{-\frac{\lambda(P)}{1+\alpha_{m+1}}})~\mathrm{for}~x\in B_{r_0}(p_{j}),~j\in J_1.$$

\noindent For $x\in B_{r_0}(p_{j}),j\in J\setminus J_1,$ we have
\begin{align*}
&4\pi\sum_{j=1}^{n}(1+\alpha_j)a_jG(x,p_{j})-\frac12v_{P,\Lambda,A}\\
=&4\pi(1+\alpha_j)a_jG(x,p_{j})-4\pi(1+\alpha_j)a_jR(x,p_{j})-\frac12a_j\lambda_{j}-a_j\log\Big(\frac{\rho_1h_{p_j}(p_{j})}{4(1+\alpha_j)^2}\Big)\\
&-a_j\Big(\frac12v_{p_{j}}(x)+\frac12\eta_{j}+\frac{d_{j}}{4(1+\alpha_j)}\lambda_{j}e^{-\frac{\lambda_{j}}{1+\alpha_j}}\Big)\\
=&a_j\log\Big(\frac{4(1+\alpha_j)^2}{\rho_1h_{p_j}(p_{j})e^{\lambda_{j}}|x-p_{j}|^{2(1+\alpha_j)}}+1\Big)-
\frac12a_j\big(\eta_{j}+\frac{d_{j}}{2(1+\alpha_j)}\lambda_{j}e^{-\frac{\lambda_{j}}{1+\alpha_j}}\big)\\
&+O(e^{-\frac{\lambda(P)}{1+\alpha_{m+1}}}).
\end{align*}
Therefore
\begin{align*}
&\exp\big(w+2\psi-4\pi\sum_{j=1}^n(1+\alpha_j)a_jG(x,p_{j})\big)-\exp\big(w+2\psi-\frac12v_{P,\Lambda,A}\big)\nonumber\\
=&O(1)|x-p_{j}|^{2a_j(1+\alpha_j)}\Big(1-\exp\big(4\pi\sum_{j=1}^{n}(1+\alpha_j)a_jG(x,p_{j})
-\frac12v_{P,\Lambda,A}\big)\Big)\\
=&O(1)|x-p_{j}|^{2a_j(1+\alpha_j)}\Big(1-\exp\Big[a_j\log\big(\frac{4(1+\alpha_j)^2}{\rho_1h_{p_{j}}(p_{j})e^{\lambda_{j}}|x-p_{j}|^{2(1+\alpha_j)}}+1\big)\\
&+\frac12a_j\big(\eta_{j}+\frac{d_{j}}{2(1+\alpha_j)}\lambda_{j}e^{-\frac{\lambda_{j}}{1+\alpha_j}}\big)
+O\big(e^{-\frac{\lambda(P)}{1+\alpha_{m+1}}}\big)\Big]\Big),
\end{align*}
where we used
$$\exp\big(w+2\psi_{k}-4\pi\sum_{j=1}^n(1+\alpha_j)a_jG(x,p_{j})\big)\sim|x-p_{j}|^{2a_j(1+\alpha_j)}~\mathrm{for}~x\in B_{r_0}(p_{j}),~j\in J\setminus J_1.$$
If $|x-p_{j}|=O(e^{-\frac{\lambda_{j}}{2a_j(1+\alpha_j)}}),$ we have $$\frac12\big(\eta_{j}+\frac{d_{j}}{2(1+\alpha_j)}\lambda_{j}e^{-\frac{\lambda_{j}}{1+\alpha_j}}\big)=O(1),$$
and
\begin{align*}
&|x-p_{j}|^{2a_j(1+\alpha_j)}\exp\Big(a_j\log\big(\frac{4(1+\alpha_j)^2}{\rho_1h_{p_{j}}(p_{j})
e^{\lambda_{j}}|x-p_{j}|^{2(1+\alpha_j)}}+1\big)+O(1)\Big)\\
&=O(e^{-\frac{\lambda(P)}{1+\alpha_{m+1}}}).
\end{align*}
If $|x-p_{j}|\gg e^{-\frac{\lambda(P)}{2a_j(1+\alpha_j)}},$ we have $$\frac12\big(\eta_{j}+\frac{d_{j}}{2(1+\alpha_j)}\lambda_{j}e^{-\frac{\lambda_{j}}{1+\alpha_j}}\big)=O(e^{-\frac{\lambda(P)}{1+\alpha_{m+1}}}
|x-p_j|^{-2(1+\alpha_j)}),$$
and
\begin{align*}
&\exp\Big(a_j\log\big(\frac{4(1+\alpha_j)^2}{\rho_1h_{p_{j}}(p_{j})
e^{\lambda_{j}}|x-p_{j}|^{2(1+\alpha_j)}}+1\big)+O(e^{-\frac{\lambda(P)}{1+\alpha_{m+1}}}
|x-p_j|^{-2(1+\alpha_j)})\Big)\nonumber\\
&=1+O\Big(\frac{1}{e^{a_j\lambda_{j}}|x-p_{j}|^{2a_j(1+\alpha_j)}}+
e^{-a_j\frac{\lambda(P)}{1+\alpha_{m+1}}}|x-p_j|^{-2a_j(1+\alpha_j)}\Big).
\end{align*}
Then
\begin{align*}
\mathfrak{E}_2=&|x-p_{j}|^{2a_j(1+\alpha_j)}\Big(\frac{1}{
e^{a_j\lambda_{j}}|x-p_{j}|^{2a_j(1+\alpha_j)}}+
e^{-\frac{a_j\lambda(P)}{1+\alpha_{m+1}}}|x-p_j|^{-2a_j(1+\alpha_j)}\Big)\\
=&O(e^{-\frac{\lambda(P)}{1+\alpha_{m+1}}}),
\end{align*}
where we used (\ref{5.8}). As a conclusion, we have
$$\mathfrak{E}_2=O(e^{-\frac{\lambda(P)}{1+\alpha_{m+1}}})~\mathrm{provided}~x\in B_{r_0}(p_{j}),~j\in J\setminus J_1.$$
Therefore, we get (\ref{5.14}).


\end{document}